\documentclass[francais,11pt]{smfart}

\usepackage[latin1]{inputenc}
\usepackage{amssymb,footnote}
\usepackage{multicol}
\usepackage{multirow}
\usepackage{dsfont}
\usepackage{color}
\usepackage[all]{xy}
\usepackage{enumerate}
\usepackage[francais]{babel}
\usepackage{tikz}
\usepackage{url}
\usepackage{arydshln}
\usepackage{comment}
\usepackage[pdfpagelabels,colorlinks=true,citecolor=blue]{hyperref}

\makeatletter
\def\@thm#1#2#3{%
  \ifhmode\unskip\unskip\par\fi
  \normalfont
  \trivlist
  \let\thmheadnl\relax
  \let\thm@swap\@gobble
  \thm@notefont{\fontseries\mddefault\upshape}%
  \thm@headpunct{.}%
  \thm@headsep 5\p@ plus\p@ minus\p@\relax
  \thm@space@setup
  \setlength{\parskip}{0mm}
  #1%
  \@topsep \thm@preskip               %
  \@topsepadd \thm@postskip           %
  \def\@tempa{#2}\ifx\@empty\@tempa
    \def\@tempa{\@oparg{\@begintheorem{#3}{}}[]}%
  \else
    \refstepcounter{#2}%
    \def\@tempa{\@oparg{\@begintheorem{#3}{\csname the#2\endcsname}}[]}%
  \fi
  \@tempa
}
\renewenvironment{proof}[1][\proofname]{\par
  \ifx@pushQED \pushQED{\qed}\fi
  \setlength{\parskip}{0mm}
  \normalfont
  \topsep6\p@\@plus6\p@ \trivlist \itemindent\z@ %
  \def\@proofhead{\normalfont\itshape #1}%
  \sbox\@tempboxa{\@proofhead}%
  \ifdim\wd\@tempboxa>0.7\linewidth \smf@skippttrue\fi
  \ifsmf@skippt
    \global\smf@skipptfalse
    \item[]{\@proofhead\@@par}
    \nobreak
  \else
    \item[\hskip\labelsep
          \unhbox\@tempboxa\pointrait]%
  \fi
  \ignorespaces
}{%
  \MakeQed
  \endtrivlist
  \@endpefalse
}
\makeatother

\definecolor{bleu}{rgb}{0,0,0.8}
\definecolor{pournous}{rgb}{0,0.5,0.8}

\newtheorem{thm''}{Th\'eor\`eme}

\newtheorem{prop'}{Proposition}[section]
\newtheorem{conj'}[prop']{Conjecture}
\newtheorem{lem'}[prop']{Lemme}
\newtheorem{thm'}[prop']{Th\'eor\`eme}
\newtheorem{cor'}[prop']{Corollaire}
\theoremstyle{definition}
\newtheorem{hyp'}[prop']{Hypoth\`ese}
\newtheorem{definit'}[prop']{D\'efinition}
\newtheorem{ex'}[prop']{Exemple}
\newtheorem{ques'}[prop']{Question}
\newtheorem{rem'}[prop']{Remarque}

\theoremstyle{plain}
\newtheorem{prop}{Proposition}[subsection]
\newtheorem{conj}[prop]{Conjecture}
\newtheorem{cor}[prop]{Corollaire}
\newtheorem{lem}[prop]{Lemme}
\newtheorem{thm}[prop]{Th\'eor\`eme}
\theoremstyle{definition}
\newtheorem{definit}[prop]{D\'efinition}

\newtheorem{rem}[prop]{Remarque}

\def\={\buildrel {\rm d\acute ef}\over =}

\DeclareMathOperator{\Id}{Id}

\DeclareMathOperator{\ind}{Ind}

\DeclareMathOperator{\End}{End}
\DeclareMathOperator{\Ind}{Ind}
\DeclareMathOperator{\Hom}{Hom}
\DeclareMathOperator{\Gal}{Gal}
\DeclareMathOperator{\Ext}{Ext}
\DeclareMathOperator{\Sym}{Sym}

\DeclareMathOperator{\val}{val}
\DeclareMathOperator{\GL}{GL}
\newcommand{\nr}{\text{\rm nr}}

\DeclareMathOperator{\ab}{\mathrm{ab}}

\DeclareMathOperator{\gal}{\mathrm{gal}}
\DeclareMathOperator{\longu}{\mathrm{long}}
\newcommand{\sep}{\mathrm{sep}}
\newcommand{\expl}{\mathrm{expl}}

\renewcommand{\P}{\mathbb{P}}
\newcommand{\GG}{\mathcal{G}}
\newcommand{\BB}{\mathcal{B}}
\newcommand{\RR}{\mathcal{R}}
\newcommand{\bbM}{\mathbb{M}}

\newcommand{\SSS}{\mathcal{S}}
\newcommand{\SK}{\mathfrak{S}}
\newcommand{\MK}{\mathfrak{M}}
\newcommand{\Mat}{\mathrm{M}}

\newcommand{\I}{\mathrm{I}}
\newcommand{\II}{\mathrm{II}}
\newcommand{\oF}{{\mathcal O}_F}
\newcommand{\oE}{{\mathcal O}_E}
\newcommand{\oEp}{{\mathcal O}_{E'}}
\newcommand{\oEpp}{{\mathcal O}_{E''}}
\newcommand{\ocE}{{\mathcal O}_{\mathcal{E}}}
\newcommand{\ocEnr}{{\mathcal O}_{\mathcal{E}^{\nr}}}
\newcommand{\ocEF}{{\mathcal O}_{\mathcal{E},F}}
\newcommand{\ocEK}{{\mathcal O}_{\mathcal{E},K}}
\newcommand{\ocEL}{{\mathcal O}_{\mathcal{E},L}}

\newcommand{\pK}{\varpi_K}
\newcommand{\pE}{\varpi_E}
\newcommand{\Z}{{\mathbb Z}}
\newcommand{\Q}{{\mathbb Q}}
\newcommand{\N}{{\mathbb N}}

\newcommand{\Qp}{\Q_{p}}

\newcommand{\Zp}{\Z_{p}}

\newcommand{\F}{\mathbb F}

\newcommand{\Fq}{{\mathbb F}_{q}}
\newcommand{\R}{\mathbb R}
\newcommand{\Kbar}{\overline K}
\newcommand{\Qbar}{\overline\Q}
\newcommand{\Qpbar}{\Qbar_p}
\newcommand{\Zpbar}{\overline \Z_p}
\newcommand{\Fpbar}{\overline \F_p}
\newcommand{\rhobar}{\overline\rho}

\newcommand{\D}{\mathcal{D}}

\newcommand{\Gp}{{\mathrm{GL}}_2(\Qp)}
\newcommand{\G}{{\mathrm{GL}}_2(F)}

\newcommand{\Goo}{G_{\infty}}
\newcommand{\GooK}{G_{\infty,K}}
\newcommand{\GooL}{G_{\infty,L}}
\newcommand{\Koo}{K_{\infty}}
\newcommand{\Loo}{L_{\infty}}
\newcommand{\cB}{\mathcal{B}}

\newcommand{\K}{{\mathrm{GL}}_2(\oF)}

\newcommand{\GE}{{\mathrm{GL}}_2(E)}
\newcommand{\gp}{{\Gal}(\Qpbar/\Qp)}
\newcommand{\g}{{\Gal}(\Qpbar/F)}
\newcommand{\vv}{{\bf v}}

\newcommand{\ttt}{{\rm t}}

\newcommand{\cri}{\mathrm{cr}}
\newcommand{\m}{{\mathfrak m}}
\newcommand{\p}{{\mathfrak p}}

\newcommand{\GbarE}{{\mathrm{GL}}_2(k_E)}

\newcommand{\ii}{{\Gal}(\Qpbar/F^{\nr})}

\def\DD{{\mathcal D}}

\def\prodloc#1#2{\prod_{#1}{\vphantom\prod}_{#2}}

\newcommand{\etapb}{\overline{\eta}'}

\def \Ieta{\mathrm I_{\eta}}
\def \Ietap{\mathrm I_{\eta'}}

\def \II{\mathrm{II}}

\renewcommand{\epsilon}{\varepsilon}
\renewcommand{\leq}{\leqslant}
\renewcommand{\geq}{\geqslant}

\author[X. Caruso]{Xavier Caruso}
\address{IRMAR
Universit\'e de Rennes I
Campus de Beaulieu
35042 Rennes Cedex, France}
\email{xavier.caruso@normalesup.org}

\author[A. David]{Agn\`es David}
\address{Unité de recherche mathématique,
Universit\'e du Luxembourg,
6 rue Coudenhove-Kalergi\\
L-1359 Luxembourg, Grand-Duché de Luxembourg}
\email{Agnes.David@ens-lyon.org}

\author[A. M\'ezard]{Ariane M\'ezard}
\address{Institut de Math\'ematiques de Jussieu,  UMR CNRS 7586
Universit\'e Pierre et Marie Curie\\
75005 Paris, France}
\email{ariane.mezard@upmc.fr}

\title[Déformations potentiellement Barsotti--Tate]{Un calcul d'anneaux 
de déformations potentiellement Barsotti--Tate}

\begin{document}

\begin{abstract}
Soit $F$ une extension non ramifiée de $\Qp$ incluse dans une clôture 
algébrique $\Qpbar$ de $\Qp$ fixée. Le premier objectif de ce travail est 
de présenter une méthode purement locale pour calculer les anneaux de 
déformations potentiellement Barsotti--Tate de type galoisien modéré des 
représentations irréductibles de dimension $2$ de $\Gal(\Qpbar /F)$. Nous 
appliquons ensuite cette méthode dans le cas particulier où $F$ est de 
degré $2$ sur $\Qp$, ce qui nous conduit, dans ce cas, à la détermination 
presque exhaustive de ces anneaux de déformations. Notre approche met en 
évidence un lien, apparemment ténu, entre la structure de ces anneaux de 
déformations et la géométrie de la variété de Kisin correspondante.

En guise de corollaire, nous vérifions, toujours dans le cas où $F$ est 
de degré $2$ sur $\Qp$ et à l'exception de deux cas très particuliers, une 
conjecture de Kisin qui prédit que les multiplicités galoisiennes 
intrinsèques valent toutes $0$ ou $1$.
\end{abstract}

\maketitle

\setcounter{tocdepth}{1}
\tableofcontents

\section*{Introduction}\label{intro}

Les anneaux de déformations universelles jouent un rôle central dans le cadre des développements récents autour de la correspondance de Langlands $p$-adique. En effet, une étude fine de ces anneaux de déformations locales et globales (\cite{Ki1}, \cite{Ki2},\cite{TW},\cite{Wil}) est  nécessaire pour démontrer les théorèmes de relèvements modulaires. Si la détermination générale de ces anneaux semble actuellement hors de portée, la conjecture de multiplicités modulaires,  dite de Breuil--Mézard (\cite{BM1}), donne des informations précises sur la structure des anneaux de déformations géométriques locales en général.
Cette conjecture prédit une formule, en termes de théorie des représentations, pour la multiplicité de Hilbert--Samuel $\mu_{\gal}$ de la fibre en caractéristique $p$ des anneaux de déformations potentiellement semi-stables d'une représentation galoisienne fixée $\bar{\rho}:G_{\Q_p}\longrightarrow \GL_2(k_E)$ où $k_E$ est un corps fini suffisamment grand de caractéristique $p$. La multiplicité de ces anneaux de déformations a deux origines : si leur spectre est la réunion de $n$ composantes irréductibles, chacune de multiplicité $\mu_i$, $1\leq i\leq n$, alors, la multiplicité de l'anneau entier est $\sum_i\mu_i$. Par ailleurs, la multiplicité d'une composante est liée à sa géométrie et mesure son degré de nilpotence. D'après la version raffinée de la conjecture de Breuil--Mézard (\cite{BM2},\cite{Ki1},\cite{EG}), il existerait des entiers naturels $m_{\bar{\rho}}(\sigma)$, dits multiplicités 
intrinsèques, ne dépendant que de la représentation galoisienne $\bar{\rho}$ et des poids de Serre $\sigma$, permettant d'écrire une formule explicite $\mu_{\gal}=\sum_{\sigma}m_{\vv,\ttt}(\sigma)m_{\bar{\rho}}(\sigma)$ où la somme porte sur les poids de Serre et les termes $m_{\vv,\ttt}(\sigma)$ s'obtiennent facilement en fonction des contraintes notées $(\vv,\ttt)$ (voir ~\S \ref{poidsec}) sur les déformations. Cette version raffinée s'étend aux représentations galoisiennes de $\Gal(\overline{\Q}_p/F)$ pour $F$ extension finie de $\Q_p$ de dimension $n$ (\cite{EG}).\\ 
Malheureusement, tant la détermination numérique que l'interprétation géométrique des multiplicités intrinsèques sont à ce jour mystérieuses. Elles ne sont connues que pour les représentations de $G_{\Q_p}$ de dimension 2 (\cite{BM1},\cite{Ki1}).
Dans le cas des représentations génériques de dimension 2 sur des extensions $F$ non ramifiées de degré fini de $\Q_p$, elles sont déterminées conjecturalement à 0 ou 1 dans \cite{BM2}. Dans le cas des représentations irréductibles, Kisin conjecture que leurs valeurs sont aussi 0 ou 1 (Conj. 2.3.5 \cite{Ki3}).
F. Sander a exhibé le premier cas de multiplicité 4 pour un anneau de déformations d'une représentation réductible de $G_{\Q_p}$ (\cite{San}). Enfin Gee et Kisin déterminent les multiplicités intrinsèques associés à des poids de Serre dits Fontaine--Laffaille réguliers (\cite{GK}).\\

L'objet de ce travail est la mise en \oe uvre d'une stratégie pour 
déterminer les multiplicités intrinsèques pour les représentations 
irréductibles de $\Gal(\overline{\Q}_p/F)$ pour $F$ non ramifiée de $\Q_p$,
en lien avec la géométrie des anneaux de déformations. D'une part, nous 
programmons sous {\sc sage}\footnote{Voir 
\url{https://cethop.math.cnrs.fr:8443/home/pub/14} pour une démonstration.} 
le calcul des multiplicités $(m_{\vv,\ttt}(\sigma))_{\sigma}$ via la 
d\'ecomposition de Jordan--H\"older de certaines représentations de 
$\GL_2(\F_{p^f})$ suivants les algorithmes théoriques de 
\cite{BP},\cite{Da}. D'autre part, nous calculons explicitement certains 
anneaux de déformations universelles via la théorie de Hodge $p$-adique.

Nous en déduisons les valeurs des multiplicités 
$\mu_{\gal}$ et $(m_{\vv,\ttt}(\sigma))_{\sigma}$ pour différents 
paramètres $(\vv,\ttt)$ de déformations. Nos méthodes conduisent à des 
résultats explicites en accord avec la conjecture de Kisin, tels que

\begin{thm''}[\emph{cf} Corollaire \ref{coro:multiplicites}]
\label{thm:f2}
Soient $F$ l'extension non ramifi\'ee de degr\'e 2 de $\Qp$, $\rhobar$ 
une repr\'esentation de dimension $2$ continue irr\'eductible de $\g$ à 
coefficients dans $\Fpbar$, non totalement non générique, et $\sigma$ un poids de Serre de $\rhobar$ qui n'est pas totalement irrégulier.
Alors la multiplicité intrinsèque de $\sigma$ dans $\rhobar$ est $1$.
\end{thm''}

Toujours dans la situation du théorème \ref{thm:f2}, notre approche nous 
permet en outre d'établir un lien entre la géométrie des espaces de 
déformations potentiellement Barsotti--Tate considérés et celle de la 
variété de Kisin $\GG\RR_{\rhobar,\psi,\vv, \ttt}$ introduite, dans un
contexte légèrement différent, dans \cite{Ki-Moduli} (voir aussi 
\cite{Ki3}).

\begin{thm''}[\emph{cf} Théorème \ref{thm:KisinBT}]
\label{thm:introdeform}
Soient $F$ l'extension non ramifi\'ee de degr\'e 2 de $\Qp$ et $\rhobar$ 
une repr\'esentation de dimension $2$ continue irr\'eductible de $\g$ à 
coefficients dans $\Fpbar$, non totalement non générique. Soient $\vv_0 = (0,2)_{\tau \in \SSS}$ et 
$\ttt$ un type galoisien modéré de niveau $2$.
Nous avons alors la trichotomie suivante :
\begin{itemize}
\item[$\bullet$] soit $\GG\RR_{\rhobar,\psi,\vv_0,\ttt}$ est vide
et $R^\psi(\vv_0, \ttt, \rhobar) = \{0\}$ ;
\item[$\bullet$] soit $\GG\RR_{\rhobar,\psi,\vv_0,\ttt}$ est un point
et $R^\psi(\vv_0, \ttt, \rhobar) \simeq
\oE[[X , Y , T ]] / (XY + p)$ ;
\item[$\bullet$] soit $\GG\RR_{\rhobar,\psi,\vv_0,\ttt}$ est isomorphe
à $\P_{k_E}^1$ et $R^\psi(\vv_0, \ttt, \rhobar) \simeq
\oE[[X , Y , T ]] / (XY + p^2)$.
\end{itemize}
\end{thm''}

Nous énonçons en réalité un résultat plus précis dans le corps de 
l'article (voir Théorème \ref{thm:KisinBT}), où nous déterminons 
entièrement et de façon complètement explicite lequel des trois cas 
précédents se produit en fonction de $\rhobar$, $\vv_0$, $\ttt$ et $\psi$. Pour une 
représentation $\rhobar$ irréductible générique, ces résultats figurent 
déjà dans \cite{BM2} et le troisième cas du théorème 
\ref{thm:introdeform} ne se produit pas. Nous renvoyons également à la partie \ref{sec:degre2} pour des énoncés additionnels dans les situations totalement non génériques.

Plus généralement, la stratégie mise en place est valable pour les représentations réductibles ou irréductibles de $\g$ en tout degré $f=[F:\Qp]$. Les résultats en degré 2 révèlent une nouvelle forme d'anneaux de déformations qu'il conviendra d'interpréter en termes modulaires. 
L'étude en tout degré fera l'objet d'un article ultérieur s'appuyant sur les résultats en degré 2.

Le plan de cet article est le suivant. Dans la partie \ref{sec:motivations}, 
nous donnons un énoncé précis de la conjecture de multiplicités modulaires 
en rappelant toutes les notations utiles dans ce contexte. La partie 
\ref{sec:outils} est consacrée à des rappels et des compléments généraux 
de th\'eorie de Hodge $p$-adique portant principalement sur les modules 
de Breuil--Kisin et ses applications au calcul d'espaces de déformations 
galoisiennes.
Avec la partie \ref{sec:methode} commence l'étude plus détaillée du cas des 
représentations potentiellement Barsotti--Tate sur un corps $p$-adique 
absolument non ramifié. Nous démontrons un théorème de classification des 
modules de Breuil--Kisin et donnons une méthode générale, basée sur ce 
théorème, pour déterminer les espaces de déformations qui nous préoccupent. 
Dans la partie \ref{sec:degre2}, enfin, nous mettons en \oe{}uvre cette 
méthode dans le cas d'une repr\'esentation galoisienne irr\'eductible sur 
une extension non ramifi\'ee de degr\'e 2 ; ceci nous conduit notamment
aux théorèmes \ref{thm:f2} et \ref{thm:introdeform} énoncés précédemment.

Les auteurs remercient Christophe Breuil pour de nombreuses discussions 
lors de la genèse de ce travail et Gabor Wiese pour son accueil à 
l'université du Luxembourg. Les auteurs ont bénéficié du soutien financier 
des ANR Cethop et ThéHopaD et du FNR Luxembourg.

\section{Motivations : la conjecture de Breuil--Mézard}
\label{sec:motivations}

\subsection{Conjecture de multiplicités modulaires}
\label{ssec:conj}
\subsubsection{Notations}
\label{sssec:notations}

Soit $p$ un nombre premier supérieur  ou égal à $5$. Pour toute extension finie $E$ de 
$\Qp$, nous notons $\oE$ son anneau des entiers, $\pE$ une uniformisante 
et $k_E=\oE/\pE$ son corps résiduel. Nous supposons également que 
$E$ est plongée dans une clôture algébrique $\overline{\Qp}$ fixée de 
$\Qp$ et nous notons $G_E=\Gal(\overline{\Qp}/E)$. Ces notations sont 
valables pour toutes les extensions algébriques (notées $E,F,K$ ou $L$) 
de $\Qp$ considérées dans cet article.

Soient $E$ une extension finie de $\Qp$ et $F$ une extension finie non 
ramifiée de $\Qp$ de degré $f$. Posons $q = p^f$. Quitte à agrandir $E$, 
nous supposons que $E$ contient l'extension quadratique non ramifiée de 
$F$, ou de manière équivalente que $k_E$ contient l'extension de degr\'e 
$2$ de $\Fq$. L'extension maximale non ramifiée de $F$ dans $\Qpbar$ est 
notée $F^{\nr}$.

Soit $\SSS$ l'ensemble des plongements de $F$ dans $E$ (de cardinal $f$), 
qui s'identifie à l'ensemble des plongements de $\Fq$ dans $k_E$ puisque 
$F$ est non ramifiée. Par abus, nous considérerons les éléments $\tau\in 
\SSS$ aussi bien comme des
plongements de $F$ dans $E$ que   de $\Fq$ dans $k_E$. Soit $F'$ l'unique 
extension quadratique non ramifiée de $F$ dans $\Qpbar$.  Notons $\SSS'$ 
l'ensemble des plongements de $F'$ dans $E$, identifié à l'ensemble des 
plongements de $\F_{q^2}$ dans $k_E$. 
Nous fixons $\tau_0$ un plongement de $\SSS$ et $\tau'_0$ dans $\SSS'$ un 
relèvement de $\tau_0$. Nous notons également 
$\varphi$ le Frobenius de $\Fq$ et de $\F_{q^2}$.

Si $\widehat F^{\times}$ désigne le complété profini de $F^{\times}$, la théorie du corps de classes local fournit un isomorphisme $G_F^{\ab}\buildrel\sim\over\longrightarrow \widehat F^{\times}$ qui envoie les éléments de Frobenius géométriques sur les uniformisantes et l'image du sous-groupe d'inertie sur $\oF^{\times}$. Par cet isomorphisme, nous voyons implicitement tout caractère de $G_F$ (resp. de $I_F=\ii$) comme un caractère de $F^{\times}$ (resp. de $\oF^{\times}$). De plus, la projection à gauche sur $\Gal(F[\sqrt[q-1]{-p}]/F)$ et à droite sur les représentants multiplicatifs $[\Fq^{\times}]\cong \Fq^{\times}$ (en envoyant $p^{\Z}(1+p\oF)$ sur $1$) induit l'isomorphisme :
\begin{eqnarray}\label{fondamental}
\nonumber \omega_F : \Gal(F[\sqrt[q-1]{-p}]/F) & \buildrel\sim\over\longrightarrow & (\oF/p)^{\times}=\Fq^{\times}\\
g &\longmapsto & \overline{\frac{g(\sqrt[q-1]{-p})}{\sqrt[q-1]{-p}}}
\end{eqnarray}
par lequel nous voyons tout caractère de $\Fq^{\times}$ 
comme un caractère de $\Gal(F[\sqrt[q-1]{-p}]/F)$ et réciproquement. 
Nous notons $\varepsilon_f$ le relèvement de Teichmüller de $\omega_f$.
Pour $\tau$ dans  $\SSS$, notons $\omega_{\tau}$ le caractère fondamental de niveau $f$, induit sur $G_F$ par (\ref{fondamental}) et le plongement $\tau\vert_{\Fq^{\times}}$ ;
notons $\varepsilon_{\tau}$ son relèvement de Teichmüller.
Notons $\varepsilon:G_F\rightarrow \Zp^{\times}$ le caractère cyclotomique $p$-adique et $\omega$ sa réduction modulo $p$.

Fixons enfin une représentation continue $\rhobar:\g\rightarrow \GbarE$
telle que $\End_{k_E[\g]}(\rhobar)=k_E$.

\subsubsection{Déformations et multiplicité galoisiennes}
\label{sssec:defgal}

Fixons $\vv$ la donn\'ee de $f$ couples d'entiers $(w_{\tau},k_{\tau})_{\tau\in \SSS}$ avec $w_{\tau}$ dans $\Z$, $k_{\tau}$ dans $\Z_{\geq 2}$ et $\ttt$ la donn\'ee d'une repr\'esentation de noyau ouvert $\ii\rightarrow \GE$ qui admet un prolongement au groupe de Weil de $F$. Fixons \'egalement un caract\`ere continu $\psi:\g\rightarrow \oE^{\times}$ tel que :
\begin{eqnarray} 
\label{relpsi}
\psi\vert_{\ii}=(\det \ttt)\prod_{\tau\in \SSS}\varepsilon_{\tau}^{2w_{\tau}+k_{\tau}-2}.
\end{eqnarray}
Une repr\'esentation lin\'eaire continue de $\g$ sur un $E'$-espace vectoriel de dimension $2$ (o\`u $E'$ est une extension finie de $E$) est dite de type $(\vv,\ttt,\psi)$ si elle est potentiellement semi-stable, si son d\'eterminant est $\psi\varepsilon$, ses poids de Hodge--Tate $(w_{\tau},w_{\tau}+k_{\tau}-1)_{\tau\in \SSS}$ et si la repr\'esentation de Weil--Deligne qui lui est attach\'ee par \cite{Fo} est isomorphe \`a $\ttt$ en restriction \`a $\ii$.

Notons $R^{\psi}(\vv,\ttt,\rhobar)$ le quotient r\'eduit de l'anneau 
local paramétrant les déformations de $\rhobar$ de type $(\vv,\ttt,\psi)$ 
introduits dans \cite[\S\ 1.4.1]{Ki1}. Plus précisément, si 
$R^{\psi}(\rhobar)$ est la $\oE$-algèbre locale complète noetherienne de 
corps résiduel $k_E$ param\'etrant les déformations de $\rhobar$ sur de 
telles $\oE$-alg\`ebres de d\'eterminant $\psi\varepsilon$, alors 
$R^{\psi}(\vv,\ttt,\rhobar)$ est l'image de $R^{\psi}(\rhobar)$ dans 
$R^{\psi}(\rhobar)[1/p]/\cap\p$, l'intersection \'etant prise sur les 
id\'eaux maximaux $\p$ de $R^{\psi}(\rhobar)[1/p]$ tels que la 
repr\'esenta\-tion~:
$$\g\rightarrow \GL_2\big(R^{\psi}(\rhobar)[1/p]\big)
\twoheadrightarrow \GL_2\big(R^{\psi}(\rhobar)[1/p]/\p\big)$$
est de type $(\vv,\ttt,\psi)$. Notons que cette dernière représentation 
est à coefficients dans une extension finie de $E$. 
Les anneaux de déformations que nous venons d'introduire se comportent
agréablement lorsque le corps $E$ change : plus précisément, si $E'$ est 
une extension finie de $E$, les $\oEp$-algèbres $R^{\psi}(\rhobar)$
et $R^{\psi}(\vv,\ttt,\rhobar)$ calculées à partir de $E'$ s'obtiennent
par extension des scalaires de $\oE$ à $\oEp$ à partir de leurs analogues 
calculés à partir de $E$ (voir lemme 2.2.2.3 de \cite{BM1}).

\begin{rem}
\label{rem:defdefo}
Nous pouvons également définir $R^{\psi}(\vv,\ttt,\rhobar)$ comme le
quotient de $R^\psi(\rhobar)$ par l'intersection des idéaux noyaux des 
morphismes de $\oE$-algèbres $f : R^\psi(\rhobar) \to \Zpbar$ pour
lesquels la représentation :
$$\g\rightarrow \GL_2\big(R^{\psi}(\rhobar)\big)
\to \GL_2\big(\Zpbar) \to \GL_2(\Qpbar)$$
est de type $(\vv,\ttt,\psi)$.
\end{rem}

Soient $A$ un anneau local, d'idéal maximal ${\mathfrak m}_A$, de dimension 
$d$ et $M$ un $A$-module de type fini. Il existe un polyn\^ome $P^A_M(X)$ 
de degré au plus $d$, appelé \emph{polyn\^ome de Hilbert--Samuel} de $M$, 
uniquement d\'eterminé par l'hypothèse $P^A_M(n)=\longu_A M/{\mathfrak 
m}^{n+1}M$ pour $n$ assez grand (voit \cite{Mat}, \S 14). La \emph{multiplicité 
d'Hilbert--Samuel} $e(M,A)$ de $M$ relativement \`a $A$ est par 
d\'efinition $d!$ fois le coefficient de $X^d$ dans le polyn\^ome 
$P^A_M(X)$. Lorsque l'anneau $A$ est sous-entendu, par exemple si $M=A$, 
nous notons simplement $e(M)$.
Dans la suite, nous posons 
\begin{eqnarray}
\mu_{\rm gal}(\vv,\ttt,\rhobar) = e(R^{\psi}\!(\vv,\ttt,\rhobar)/\pE).
\end{eqnarray}

\subsubsection{Poids de Serre et multiplicité automorphe}

Un \emph{poids de Serre} de $F$ est une représentation lisse absolument irréductible de $\K$, ou de manière équivalente de $\GL_2(\Fq)$ sur $k_E$. Un poids de Serre est de la forme :
\begin{eqnarray}\label{poids}
\bigotimes_{\tau\in \SSS}\big((\Sym^{r_{\tau}}k_E^2)^{\tau}\otimes_{k_E}\tau\circ{\det}^{s_{\tau}}\big)
\end{eqnarray}
avec des entiers $r_{\tau}$ et  $s_{\tau}$ dans $\{0,\ldots,p-1\}$ et o\`u $\GL_2(\Fq)$ agit sur $(\Sym^{r_{\tau}}k_E^2)^{\tau}$ par le plongement  $\tau$ de $\Fq$ dans $k_E$ et l'action sur la base canonique de $k_E^2$.
Avec la condition supplémentaire que tous les $s_\tau$ ne sont pas égaux à $p-1$, cette écriture est  unique.
Si $\sigma$ est un poids de Serre comme en (\ref{poids}), nous notons $\sigma^s=\bigotimes_{\tau\in \SSS}\big((\Sym^{p-1-r_{\tau}}k_E^2)^{\tau}\otimes_{k_E}\tau\circ{\det}^{r_{\tau}+s_{\tau}}\big)$ le sym\'etrique de $\sigma$. 
Nous identifions le poids de Serre $\sigma$ avec les $f$-uplets $(r_{\tau},s_{\tau})_{\tau\in\SSS}$ qui le définissent. Quand l'action du caractère central est impos\'ee, le $f$-uplet $(s_{\tau})_{\tau\in\SSS}$ s'obtient à partir des $(r_{\tau})_{\tau\in f}$. Le poids $\sigma$ est alors identifi\'e au $f$-uplet $(r_\tau)_{\tau\in\SSS}$.

Un poids de Serre pour lequel aucun $r_\tau$ n'est égal à $p-1$ est dit \emph{régulier} ; il est dit \emph{totalement irrégulier} si tous les $r_\tau$ sont égaux à $p-1$.

Notons $\D(\rhobar)$ l'ensemble des poids de Serre de $\rhobar$ d\'efini dans \cite[\S\ 3]{BDJ} (voir la partie \ref{sssec:poidsrhobar} pour une définition et des méthodes de calcul de $\D(\rhobar)$).

Nous associons \`a $\vv$ la représentation continue de $\K$ sur $E$ :
$$\bigotimes_{\tau\in \SSS}\big((\Sym^{k_{\tau}-2}E^2)^{\tau}\otimes_{E}\tau\circ{\det}^{w_{\tau}}\big) $$
où $\K$ agit sur $(\Sym^{k_{\tau}-2}E^2)^{\tau}$ via $\tau:\oF\hookrightarrow E$.
Par \cite{He}, nous pouvons également associer à~$\ttt$ une représentation lisse absolument irr\'eductible $\sigma(\ttt)$ de $\K$ sur $E$ (quitte à agrandir $E$) de caractère central $\det \ttt$. Notons $\D(\vv,\ttt)$ l'ensemble des poids de Serre à multiplicité près qui sont des constituants du semi-simplifié sur $k_E$ de $\sigma(\vv)\otimes_E\sigma(\ttt)$ et $m_{\vv,\ttt}(\sigma)$ la multiplicité (dans $\N$) avec laquelle un poids de Serre $\sigma$ (quelconque) appara\^\i t dans ce semi-simplifi\'e. 

\subsubsection{Conjecture de multiplicit\'es modulaires}

L'énoncé de la conjecture suivante résume les conjectures %
de \cite{BM2}, \cite{EG}, \cite{GK}, \cite{Ki1}, \cite{Ki3} : 

\begin{conj}\label{conjprincipale}
 Soit $\rhobar:\g\rightarrow \GbarE$ une représentation 
telle que $\End_{k_E[\g]}(\rhobar)=k_E$.
 
Pour tout poids de Serre $\sigma$ de $F$, il existe un entier naturel $m_{\bar{\rho}}(\sigma)$, dit multiplicit\'e intrins\`eque de $\sigma$ dans $\rhobar$, qui ne d\'epend que de $\bar{\rho}$ et $\sigma$, tel que pour tout $\vv$ et tout $\ttt$ comme précédemment :
$$\mu_{gal}(\vv,\ttt,\rhobar)=\sum_{\sigma}m_{\rhobar}(\sigma)m_{\vv,\ttt}(\sigma)$$ 
De plus $m_{\bar{\rho}}(\sigma)$ est non nul si et seulement si $\sigma$ est dans $\D(\bar{\rho})$.
\end{conj}
\begin{rem} Si $\rhobar$ est une repr\'esentation de $G_{\Qp}$ de dimension deux  avec $\rhobar \not\simeq \left(\begin{array}{ccc} \omega\chi &*\cr 0&\chi\cr\end{array}\right)$ pour tout caract\`ere $\chi$, alors la conjecture \ref{conjprincipale} est vraie. Les multiplicit\'es intrins\`eques sont inf\'erieures \`a deux et sont explicitement connues. En particulier, si $\rhobar$ est irr\'eductible, elles valent 0 ou 1 (\cite{BM1}, \cite{Ki1}, \cite{EG} Théorème 1.1.5).
\end{rem}
\begin{rem}
Si $F/\Qp$ est une extension finie non ramifiée et la représentation $\rhobar$ est g\'en\'erique au sens de \cite{BP}, la conjecture de multiplicit\'es raffinées (\cite{BM2}) implique que les multiplicités intrinsèques sont major\'ees par 1. Plus généralement, Kisin demande si ces multiplicités intrinsèques sont major\'ees par 1 si $\rhobar$ est une représentation irr\'eductible de $G_F$ (\cite{Ki3}, Conjecture 2.3.2).
\end{rem}

\begin{rem}
Nous avons ici énoncé la conjecture de multiplicités modulaires pour des déformations semi-stables.
Il en existe également une version cristalline, qui ne diffère de celle-ci que pour les types galoisiens scalaires.
Il est de plus conjecturé que les deux versions de la conjecture sont satisfaites par les mêmes valeurs des multiplicités intrinsèques (\cite{Ki3}, conjecture 2.2.3)
\end{rem}

\subsection{Calculs de poids de Serre}
\label{poidsec}

Dans cette partie, nous expliquons comme déterminer l'ensemble $\D(\rhobar)$ et, pour certains types de déformations $(\vv_0,\ttt)$, l'ensemble $\D(\vv_0,\ttt)$ et les multiplicités $m_{\vv_0,\ttt}(\sigma) $.

\subsubsection{Poids de Serre d'une représentation irréductible}\label{sssec:poidsrhobar}

Soit $\rhobar:\g\longrightarrow \GbarE$ une représentation continue irréductible fixée.
L'ensemble $\D(\rhobar)$ des poids de Serre de $\rhobar$ d\'efini dans \cite[\S\ 3]{BDJ} (voir aussi Proposition A.3 de \cite{Br2}) est l'ensemble des poids de Serre (la puissance du d\'eterminant fix\'ee est sous-entendue, pour des pr\'ecisions voir \cite{BP}):
$$\bigotimes_{\tau\in \SSS}(\Sym^{r_{\tau}}k_E^2)^{\tau}$$
avec $(r_\tau)_{\tau\in \SSS}\in\{0,\ldots, p-1\}^f$ pour lesquels il existe $\SSS_1\subset \SSS'$ tel que $|\SSS_1|=f$, $\SSS=\{\tau'_{|F},\tau'\in\SSS_1\}$ et $\chi'$ un caract\`ere de l'inertie qui s'\'etend \`a $\g$ tels que
$$
\rhobar_{|\ii}\cong\begin{pmatrix}\prod_{\tau'\in \SSS_1}\omega_{\tau'}^{(r_{\tau'_{|F}}+1)}&0\\0&
\prod_{\tau'\in \SSS' \setminus \SSS_1}\omega_{\tau'}^{(r_{\tau'_{|F}}+1)}\end{pmatrix}\otimes\chi'.
$$

Explicitons un algorithme pratique pour obtenir $\D(\rhobar)$. Soit $\rhobar: \g\longrightarrow\GbarE$ une repr\'esentation continue irr\'eductible telle que 
$$\rhobar_{|\ii}\cong\Big(\omega_{\tau'_0}^{\sum_{j=0}^{f-1}(r_j+1)p^j}\oplus\omega_{\tau'_0}^{p^f\sum_{j=0}^{f-1}(r_j+1)p^j}\Big)\otimes \chi'$$
pour un plongement $\tau'_0\in \SSS'$ fix\'e avec la famille $(r_j)_{0\leq j\leq f-1}$ dans $\{0,\ldots,p-1\}\times\{-1,\ldots,p-2\}^{f-1}$.
 L'ensemble des poids de Serre de $\rhobar$ est l'ensemble des $f$-uplets $(r'_j)_{0\leq j\leq f-1}\{0,\ldots,p-1\}^f$ satisfaisant une congruence de la forme
\begin{equation}
\label{congruencef}
\sum_{j=0}^{f-1}\pm (r_j'+1)p^j\equiv \sum_{j=0}^{f-1}(r_j+1)p^j \mod (p^f+1)
\end{equation}
pour un des $2^f$ choix de signes possibles dans le membre de gauche.
Si $\rhobar$ est g\'en\'erique (\cite{BP}, Définition 11.7), autrement dit si 
la famille $(r_j)_{0\leq j\leq f-1}$ est dans $\{1,\ldots,p-2\}\times\{0\ldots,p-3\}^{f-1}$, 
nous obtenons ainsi $2^f$ $f$-uplets distincts (correspondants aux 
$2^f$ choix de signes dans la congruence (\ref{congruencef})) qui 
s'expriment formellement en fonction des $(r_j)_{0\leq j\leq f-1}$ 
(\cite{BP}, Lemma 11.4). Si $\rhobar$ n'est pas g\'en\'erique, il est 
possible que les formules donnant les $2^f$ $f$-uplets dans le cas 
g\'en\'erique donnent des $f$-uplets n'appartenant pas \`a 
$\{0,\ldots,p-1\}^f$. Il faut alors revenir \`a la congruence 
(\ref{congruencef}) pour modifier convenablement le $f$-uplet pour 
d\'eterminer le poids de Serre, \'el\'ement de $\{0,\ldots,p-1\}^{f}$ 
associ\'e \`a $\rhobar$ correspondant \`a ce $f$-uplet. Dans \cite{Da}, A. David donne 
une description plus directe de l'ensemble $\DD(\rhobar)$ et une formule explicite pour la modification à apporter dans le cas non générique.

\begin{definit}\label{defpdmod}
Soit $\rhobar:G_F\longrightarrow\GbarE$ une repr\'esentation continue irr\'eductible non g\'en\'erique.
Les poids de Serre de $\rhobar$ qui s'obtiennent apr\`es modification des formules donnant les poids de Serre dans le cas g\'en\'erique sont dits \emph{modifiés}.
\end{definit}

Pour $f=2$, notons $\omega_4=\omega_{\tau'_0}$ le caractère fondamental de 
niveau $4$.
Il y a quatre cas possibles de représentations irréductibles 
non génériques.  Nous indiquons pour chacune ses poids de Serre, en 
commen\c{c}ant par l'unique poids modifié et continuant par son 
symétrique (voir \cite{Da}).
Dans tous les cas ci-dessous, $s$ est un entier dans $\Z$ et $\theta$ un élément de $k_E^{\times}$.
 \begin{enumerate}[(i)]
\item 
$\rhobar \simeq \Ind^{G_{F}}_{G_{F'}}\left(\omega_4^{ 1 + r_0} \cdot \nr' (\theta)    \right) \otimes \omega_2^s$
avec $1\leq r_0 \leq p-2$~;\\
$\begin{array}{l l c l}
\DD(\rhobar) = \{ &   (r_0 + 1 , p-1) \otimes \det^{-1} & , & (p-2 - r_0,0) \otimes \det^{-(p-1-r_0)} , \\
 & (p-1-r_0,p-2)\otimes \det^{r_0} & , & (r_0 - 1 , p-1)\otimes \mathds{1}) \} \otimes \det^s. \\
\end{array}
$
\item 
$\rhobar=  \Ind^{G_{F}}_{G_{F'}} \left( \omega_4 \cdot \nr' (\theta) \right) \otimes \omega_2^s$~;\\
$\DD(\rhobar) = \left\{ (1,p-1) \otimes \det^{-1} , (p-2, 0) \otimes \det^{ -(p-1)} , (p-1, p-2) \otimes \mathds{1} \right\} \otimes {\det}^s .$
\item 
$\rhobar \simeq \Ind^{G_{F}}_{G_{F'}} \left( \omega_4^{p( 2 + r_1)} \cdot \nr' (\theta)  \right) \otimes \omega_2^s$ avec $0\leq r_1 \leq p-3$ ;\\
$\begin{array}{l l c l}
\DD(\rhobar) = \{ &  (p-1,r_1+2) \otimes \det^{-p}  & , &  (0,p-3-r_1)  \otimes \det^{p-1 + p(1+r_1)} , \\
 & (p-2,p-2-r_1)  \otimes \det^{p(1+r_1)} & , & (p-1,r_1) \otimes \mathds{1}  \}\otimes {\det}^s. \\
\end{array}$
\item 
$\rhobar \simeq \Ind^{G_{F}}_{G_{F'}} \left( \omega_4^{ p} \cdot \nr' (\theta)  \right) \otimes \omega_2^s$~;\\
$\DD(\rhobar) = \left\{ (p-1,1) \otimes \det^{-p} , (0, p-2) \otimes \det^{p-1} , (p-2, p-1) \otimes \mathds{1} \right\}  \otimes {\det}^s . $
\end{enumerate}

Dans les cas (i) et (iii) de la liste ci-dessus, un seul des deux entiers $(r_0 , r_1)$ prend une valeur non g\'en\'erique, à savoir $0$ ou $p-1$ pour $r_0$ et $-1$ ou $p-2$ pour $r_1$.
Dans les cas (ii) et (iv), les entiers $r_0$ et $r_1$ sont tous les deux non g\'en\'eriques et nous appelons \textit{totalement non g\'en\'eriques} les repr\'esentations associ\'ees.

\subsubsection{Poids du type}
\label{sssec:poidstype}

Fixons deux caractères distincts modérément ramifiés $\eta,\eta':\ii \rightarrow \oE^{\times}$ qui s'\'etendent \`a $G_F$ et considérons le type galoisien $\ttt =\eta\oplus\eta'$. Les conditions sur $\eta, \eta'$ impliquent qu'ils se factorisent par $\Gal(F^{\nr}[\sqrt[q-1]{-p}]/F^{\nr})=\Gal(F[\sqrt[q-1]{-p}]/F)\simeq \Fq^{\times}$ (cf. (\ref{fondamental})) et qu'il y a donc une mani\`ere naturelle de les \'etendre \`a $\g$ (\emph{via} la th\'eorie du corps de classes local, cela revient juste \`a envoyer $p\in \widehat F^{\times}$ vers $1$). Notons $\I(\oF)$ le sous-groupe d'Iwahori de $\K$  des matrices triangulaires sup\'erieures modulo $p$.

Le type $\sigma(\ttt)$ associé par \cite{He} à $\ttt$ est alors $\ind_{\I(\oF)}^{\K}(\eta'\otimes\eta)$ o\`u $\eta'\otimes\eta: \I(\oF)\rightarrow \oE^{\times}$ est le caractère :
\begin{eqnarray*}
\begin{pmatrix} a & b\\ pc & d\end{pmatrix}\mapsto \eta'(\overline a)\eta(\overline d)
\end{eqnarray*}
et  $\ind_{\I(\oF)}^{\K}(\eta'\otimes\eta)$ est le $E$-espace vectoriel des fonctions $f:\K\rightarrow E$ telles que :
$$\forall k\in \I(\oF), \, k'\in \K, \;\; f(kk')=(\eta'\otimes\eta)(k)f(k')$$
 muni de l'action à gauche de $\K$ par translation à droite sur les 
fonctions. Cette action se factorise par $\GL_2(\Fq)$.

Notons de m\^eme $\ind_{\I(\oF)}^{\K}(\overline\eta'\otimes\overline\eta)$ le $k_E$-espace vectoriel des fonctions $f:\K\rightarrow k_E$ telles que $f(kk')=(\overline\eta'\otimes\overline\eta)(k)f(k')$ muni de la m\^eme action de $\GL_2(\Fq)$. C'est la réduction modulo $\pE$ du $\oE$-réseau de $\sigma(t)$ des fonctions à valeurs dans $\oE$. Nous notons $\D(\ttt)$ l'ensemble des poids de Serre de $F$ qui sont des constituants irréductibles de $\ind_{\I(\oF)}^{\K}(\overline\eta'\otimes\overline\eta)$.
L\`a encore, nous disposons d'une détermination de $\D(\ttt)$ d\^ue \`a Breuil et Pa\v{s}k${\overline{\rm u}}$nas (\cite{BP} \S2 lemme 2.2) et d'un algorithme combinatoire dû à David (\cite{Da}).
Notamment pour $f=2$, l'ensemble des poids de Serre qui sont des constituants de $\ind_{\I(\oF)}^{\K}\overline\eta'\otimes\overline\eta$ avec $\overline\eta=\omega_{\tau_0}^{c_0+pc_1}\overline\eta'$ ($0\leq c_0,c_1\leq p-1$)
 sont les couples de $\{0,\ldots,p-1\}^2$ apparaissant dans la liste de couples suivante :
\begin{equation}(c_0,c_1) \otimes \overline\eta',
\begin{array}{cc}
(p-2-c_0,c_1-1) \otimes \det^{c_0 + 1} \overline \eta'  \cr
 (c_0-1,p-2-c_1) \otimes \det^{p(c_1 + 1)}  \overline\eta'  \cr
 \end{array},
(p-1-c_0,p-1-c_1)  \otimes {\det}^{c_0 + pc_1} \overline\eta' .
\end{equation}

Dans la partie \ref{sec:degre2}, nous considérons des contraintes de déformations $(\vv_0,\ttt)$, avec $\vv_0 = ((0,2))_{\tau \in \SSS}$ et $\ttt$ comme ci-dessus.
Dans ce cadre, la représentation $\sigma(\vv_0)$ est triviale.
L'ensemble $\D(\vv_0,\ttt) $ est donc l'ensemble $\D(\ttt) $ ;  la multiplicité $m_{\vv_0, \ttt}(\sigma)$ est $1$ pour tout $\sigma$ dans $\D(\ttt)$, $0$ sinon.

\section{Outils de théorie de Hodge $p$-adique}
\label{sec:outils}

Les méthodes que nous allons utiliser par la suite pour déterminer 
explicitement certains espaces de déformations $R^\psi(\vv, \ttt, \rhobar)$
sont purement locales. Elles sont basées essentiellement sur la théorie 
des modules de Breuil--Kisin (\cite{Br-Corps}, \cite{Ki-Crys}, 
\cite{Ki-Moduli}, \cite{Ki2}, \cite{Ki3}). Cette seconde partie est 
consacrée à présenter succintement et à compléter sur certains points 
les résultats principaux de la théorie de Breuil--Kisin que nous serons
amenés à utiliser couramment dans la suite.

\subsection{Rappels sur la théorie de Breuil--Kisin}

Soit $K$ une extension finie de $\Qp$\footnote{Plus généralement, le 
contenu de cette partie \S \ref{sec:outils} s'étend sans modification à 
un corps complet pour une valuation discrète, de caractéristique mixte et 
de corps résiduel parfait.}. La théorie de Breuil--Kisin permet d'étudier 
les représentations semi-stables de~$G_K$. Pour l'exposer, nous avons
besoin d'introduire quelques notations supplémentaires :
\begin{itemize}
\item $K_0$, l'extension maximale non ramifiée de $\Qp$ 
incluse dans $K$ ;
\item $f = [K_0 : \Qp]$ le degré résiduel de $K$ ;
\item $\varphi$, l'endomorphisme de Frobenius agissant sur 
$K_0$ ;
\item $W$, l'anneau des entiers de $K_0$ ; %
\item $E(u)$, le polynôme minimal de $\pK$ (une uniformisante fixée de $K$) sur $K_0$ ;
\item $(\pi_s)_{s \in \N}$, un système compatible de racines $p^s$-ièmes de $\pK$ dans $\Kbar$ ;
\item $\Koo$, l'extension de $K$ engendrée par les $\pi_s$ pour $s$ dans $\N$ ;
\item $\Goo = \Gal(\Kbar / \Koo)$ ;
\item les anneaux $\SK = W [[ u]]$ et $\ocE$ défini comme le complété $p$-adique de $\SK[1/u]$,
$$
\ocE = \left\{ \sum_{i \in \Z} a_i u^i \;  \Big|\;  a_i \in W, \lim_{i \to - \infty} a_i = 0 \right\} ;
$$
nous les munissons d'un endomorphisme de Frobenius, que par abus nous 
notons encore $\varphi$, défini par :
$$\varphi\Big(\sum a_i u^i\Big) = \sum \varphi(a_i) u^{pi}.$$
\end{itemize}
Remarquons que $\ocE$ contient $\SK$ et que le Frobenius défini 
ci-dessus est compatible à cette inclusion. Notons également que 
$\ocE$ est un anneau de valuation discrète (pour la valuation $p$-adique)
complet et que son corps résiduel s'identifie canoniquement à $k_K((u))$.

\subsubsection{Classification des représentations de $\Goo$} 
\label{sssec:classifGoo}

D'après la théorie du corps des normes de Fontaine et Wintenberger (voir 
\cite{Wi}), le groupe de Galois $\Goo$ s'identifie au groupe de Galois 
absolu d'un corps de séries de Laurent en une variable à coefficients 
dans le corps résiduel $k_K$. Autrement dit, nous avons un isomorphisme 
(qui peut être rendu canonique)
\begin{equation}
\label{eq:corpsnormes}
\Goo \simeq \Gal(k_K((u))^\sep / k_K((u)))
\end{equation}
où $k_K((u))^\sep$ désigne une clôture séparable de
$k_K((u))$.
Notons $\ocEnr$ l'extension étale (infinie) de~$\ocE$ correspondant à 
l'extension résiduelle $k_K((u))^\sep / k_K((u))$. Le groupe de Galois 
absolu de $k_K((u))$ agit naturellement sur $\ocEnr$. Il en va donc de
même de $\Goo$ grâce à l'isomorphisme \eqref{eq:corpsnormes}.

\begin{definit}\label{defi-phi-module}
Soit $R$ une $\Zp$-algèbre locale complète. 

Un \emph{$\varphi$-module} sur $R \hat\otimes_{\Zp} \ocE$ est un
$(R \hat\otimes_{\Zp} \ocE)$-module $M$ libre de rang fini muni d'une 
application $\varphi_M : M \to M$ qui est linéaire par rapport à $R$
et $\varphi$-semi-linéaire par rapport à $\ocE$.

Le $\varphi$-module $M$ est dit \emph{étale} si l'image de $\varphi_M$ 
engendre $M$ comme $(R \hat\otimes_{\Zp} \ocE)$-module.
\end{definit}

\begin{rem}
Étant donné $\varphi_M$ comme dans la définition \ref{defi-phi-module}, il est
souvent pratique de travailler avec son \og linéarisé \fg :
$$\Id \otimes \varphi_M : (R \hat\otimes_{\Zp} \ocE) \otimes_{\varphi, 
R \hat\otimes_{\Zp} \ocE} M \longrightarrow M$$
où l'endomorphisme $\varphi$ de $R \hat\otimes_{\Zp} \ocE$ qui apparaît
ci-dessus agit sur $R$ par l'identité et sur $\ocE$ par le Frobenius.
L'application $\Id \otimes \varphi_M$ est alors $(R \hat\otimes_{\Zp} 
\ocE)$-linéaire et le $\varphi$-module $M$ est étale si et seulement si
$\Id \otimes \varphi_M$ est un isomorphisme.
\end{rem}

Nous avons le résultat suivant qui est une variante à coefficients 
d'un théorème classique de Fontaine (voir \cite{Fo2}) :

\begin{thm}\label{thm:equivFontaine}
Pour toute $\Zp$-algèbre locale complète noetherienne $R$, le foncteur
$$
V \mapsto \bbM^\star(V) = ( V^\star \hat\otimes_{\Zp} \ocEnr )^{\Goo}
$$
réalise une équivalence de catégories entre la catégorie des 
$R$-représentations libres de rang fini de $\Goo$ et celle des 
$\varphi$-modules\footnote{La structure de $\varphi$-module sur
$( V^\star \hat\otimes_{\Zp} \ocEnr )^{\Goo}$ est donnée par
l'application $\Id \otimes \varphi$.} étales sur $R \hat\otimes_{\Zp} 
\ocE$.
\end{thm}

\begin{rem}
Le foncteur $\bbM^\star$, comme toutes les constructions présentées dans ce 
paragraphe, commute au changement de base.
\end{rem}

\paragraph*{Le cas d'une $W$-algèbre}

Si l'anneau des coefficients $R$ est une $W$-algèbre, nous disposons 
d'une variante du théorème \ref{thm:equivFontaine} obtenue
essentiellement en remplaçant $\varphi$ par $\varphi^f$.

\begin{definit}
Soit $R$ une $W$-algèbre locale complète.

Un \emph{$\varphi^f$-module} sur $R \hat\otimes_W \ocE$ est un $(R
\hat\otimes_W \ocE)$-module $M$ libre de rang fini muni d'une application 
$\varphi^{(f)}_M : M \to M$ qui est linéaire par rapport à $R$ et
$\varphi^f$-semi-linéaire par rapport $\ocE$.

Le \emph{$\varphi^f$-module} $M$ est dit \emph{étale} si l'image de 
$\varphi^{(f)}_M$ engendre $M$ comme $(R \hat\otimes_W \ocE)$-module.
\end{definit}

\begin{thm} \label{thm:equivrepPhimod}
Pour toute $W$-algèbre locale complète noetherienne $R$, le foncteur
$$
V \mapsto \bbM^\star_W(V) = ( V^\star \hat\otimes_W \ocEnr )^{\Goo}
$$
réalise une équivalence de catégories entre la catégories des 
$R$-représentations libres de rang fini de $\Goo$ et celle des 
$\varphi^f$-modules\footnote{La structure de $\varphi^f$-module sur 
$(V^\star \hat\otimes_W \ocEnr)^{\Goo}$ est donnée par l'application 
$\Id \otimes \varphi^f$.} étales sur $R \hat\otimes_W \ocE$.
\end{thm}

Si $R$ est une $W$-algèbre et $V$ est une $R$-représentation libre de 
rang fini de $\Goo$, nous disposons donc simultanément des deux objets 
$\bbM^\star(V)$ et $\bbM^\star_W(V)$ que nous allons comparer. Comme $R$ est 
une $W$-algèbre, nous avons une décomposition canonique de l'anneau
$R \hat\otimes_{\Zp} W$ :
\begin{eqnarray*}\label{decomposition}
R \hat\otimes_{\Zp} W & \simeq & \textstyle \prod^{f-1}_{i = 0} R \\
y \otimes x & \mapsto & (\varphi^{-i}(x) \cdot y)_{0 \leq i \leq f-1}.
\end{eqnarray*}
En tensorisant par $\ocE$ sur $W$, nous en déduisons un isomorphisme 
canonique :
\begin{equation}
\label{eq:decompocER}
R \hat\otimes_{\Zp} \ocE \simeq 
\textstyle \prod^{f-1}_{i \in 0} 
R\hat\otimes_{\iota \circ \varphi^{-i}, W} \ocE.
\end{equation}
où $\iota$ désigne le morphisme structurel de $R$ comme $W$-algèbre.
Concrètement, le $i$-ième facteur $R \hat\otimes_{\iota \circ 
\varphi^{-i}, W} \ocE$ admet la description explicite suivante :
\begin{equation}
\label{eq:ocERexpl}
R \hat\otimes_{\iota \circ \varphi^{-i}, W} \ocE \,\, \simeq \,\,
\Bigg\{ \sum_{j \in \Z} a_j u^j \quad  \Big|\quad  a_j \in R, \lim_{j \to - \infty} a_j = 0 \Bigg\}
\end{equation}
l'identification faisant correspondre le tenseur pur $\lambda \otimes
(\sum a_j u^j)$ avec la série $\lambda \sum \varphi^{-i} (a_j) \:u^j$. 
Avec ce choix, l'isomorphisme \eqref{eq:ocERexpl} est compatible à la fois à la
multiplication par les éléments de $W$ agissant sur le facteur $\ocE$
et à la multiplication par les éléments de $R$ agissant sur le facteur
$R$.

Soit $e_i$ l'idempotent de $R \hat\otimes_{\Zp} \ocE$ correspondant au 
$i$-ième facteur de la décomposition \eqref{eq:decompocER}. Si $M$ est un 
module sur $R \hat\otimes_{\Zp} \ocE$, en posant $M^{(i)} = e_i M$, nous 
avons une décomposition canonique de $M$ :
\begin{equation}\label{decompositionmodule}
M = M^{(0)} \oplus M^{(1)} \oplus \cdots \oplus M^{(f-1)}.
\end{equation}

Examinons à présent le comportement du Frobenius. En revenant aux 
définitions, nous nous apercevons que l'action de $\Id \otimes \varphi$ 
sur $R \hat\otimes_{\Zp} \ocE$ correspond à l'endomorphisme de 
$\prod^{f-1}_{i = 0} R \hat\otimes_{W, \iota \circ \varphi^{-i}} \ocE$ 
suivant :
$$(s_0(u), s_1(u), \ldots, s_{f-1}(u)) \mapsto
(s_{f-1}(u^p), s_0(u^p), s_1(u^p), \ldots, s_{f-2}(u^p))$$
où les $s_i(u)$ désignent des séries à coefficients dans $R$, éléments
du membre de droite de l'isomorphisme \eqref{eq:ocERexpl}. Nous en déduisons
que, si $M$ est un $\varphi$-module sur $R \hat\otimes_{\Zp} \ocE$,
l'endomorphisme $\varphi_M$ envoie $M^{(i)}$ sur $M^{(i+1 \text{ mod } f)}$ 
pour $i$ compris entre $0$ et $f-1$. Notons en particulier que la
$f$-ième puissance de $\varphi_M$ stabilise chacun des $M^{(i)}$. La proposition suivante est immédiate.

\begin{prop}\label{propnomfarfelu}
Soient $R$ une $W$-algèbre locale complète et $V$ une $R$-représentation 
libre de rang fini de $\Goo$. Alors, avec les notations introduites 
précédemment, nous avons un isomorphisme canonique de $\varphi^f$-modules 
sur $R \hat\otimes_W \ocE$ :
$$\bbM^\star_W(V) = \bbM^\star(V)^{(0)}$$
où $\bbM^\star(V)^{(0)}$ est muni d'une structure de $\varphi^f$-module
sur $R \hat\otimes_W \ocE$ \emph{via} l'endomorphisme $\varphi_{\bbM^\star(V)}
\circ \varphi_{\bbM^\star(V)} \circ \cdots \circ \varphi_{\bbM^\star(V)}$ ($f$ fois).
\end{prop}

\subsubsection{Modules de Breuil--Kisin}

Lorsque $V$ est une représentation semi-stable de $G_K$ à coefficients 
dans l'anneau des entiers d'une extension finie de $\Qp$, le 
$\varphi$-module $\bbM^\star(V)$ associé à la restriction de $V$ à $\Goo$ par
le théorème \ref{thm:equivFontaine} a une forme particulière qui rend
son étude plus facile. Dans ce paragraphe, nous introduisons la notion de
module de Breuil--Kisin qui rend compte de cette forme particulière
agréable.

\begin{definit}\label{defmoduleBK}
Soit $R$ une $\Zp$-algèbre locale noetherienne complète.

Un \emph{module de Breuil--Kisin} sur $R \hat\otimes_{\Zp} \SK$ est
la donnée d'un $(R \hat\otimes_{\Zp} \SK)$-module de type fini $\MK$ 
et d'un endomorphisme $\varphi_\MK : \MK \to \MK$ vérifiant les 
conditions suivantes :
\begin{enumerate}[(i)]
\item le module $\MK$ n'a pas de $u$-torsion ;
\item l'application $\varphi_\MK$ est linéaire par rapport à 
$R$ et $\varphi$-semi-linéaire par rapport à $\SK$ ;
\item \label{item:hauteurfinie}
le $(R \hat\otimes_{\Zp} \SK)$-module engendré par l'image
de $\varphi_\MK$ contient $E(u)^r \MK$ pour un certain 
entier $r$ positif ou nul.
\end{enumerate}
Le plus petit entier $r$ vérifiant la condition 
\eqref{item:hauteurfinie} est appelé la \emph{$E(u)$-hauteur} de $\MK$.
\end{definit}

\begin{rem} 
De même que pour un $\varphi$-module, il est souvent commode, lorsque
$\MK$ est un module de Breuil--Kisin, de considérer le linéarisé de
$\varphi_\MK$ :
$$\Id \otimes \varphi_\MK :
(R \hat\otimes_{\Zp} \SK) \otimes_{\Id \otimes \varphi, 
R \hat\otimes_{\Zp} \SK} \MK \longrightarrow \MK.$$
qui est une application $(R \hat\otimes_{\Zp} \SK)$-linéaire. La
condition \eqref{item:hauteurfinie} de la définition \ref{defmoduleBK} est
alors équivalente au fait que le conoyau de $\Id \otimes \varphi_\MK$
est annulé par $E(u)^r$.

Lorsque $R$ est de plus une $W$-algèbre, la décomposition de $(R \hat\otimes_{\Zp} \SK)$ analogue à 
\eqref{decomposition} induit une décomposition
\begin{equation}
\label{decompBKmodule}
\MK=\MK^{(0)}\otimes\ldots\otimes\MK^{(f-1)}
\end{equation}
analogue à \eqref{decompositionmodule}.
\end{rem}

\begin{definit}
Soit $R$ une $\Zp$-algèbre locale noetherienne complète et $h$ un entier naturel. 
Une $R$-représentation $V$ de $G_K$ libre de 
rang fini est dite de \emph{$E(u)$-hauteur inférieure ou égale à $ h$} s'il existe un 
module de Breuil--Kisin $\MK$ sur $R \hat\otimes_{\Zp} \SK$ de hauteur inférieure ou égale à
$ h$ tel que nous ayons un isomorphisme de $\varphi$-modules :
$$\bbM^\star(V) \simeq \ocE \otimes_{\SK} \MK,$$
la structure de $\varphi$-module sur $\ocE \otimes_{\SK} \MK$ étant
donnée par $\varphi \otimes \varphi_\MK$.
\end{definit}

\begin{thm}[Kisin]
\label{thm:Eusst}
Soit $R$ une $\Zp$-algèbre locale noetherienne complète que nous
supposons de plus plate sur $\Zp$.
Soit $V$ une représentation de $G_K$ à coefficients dans $R$ qui est 
libre de rang fini. Nous supposons que pour tout morphisme d'anneaux $R 
\to \Qpbar$, la représentation $\Qpbar \otimes_R V$ est semi-stable
à poids de Hodge--Tate compris entre $0$ et $h$. Alors :
\begin{enumerate}[(i)]
\item le $\varphi$-module associé à $V_{|\Goo}$ par le 
théorème~\ref{thm:equivFontaine} est de $E(u)$-hauteur $\leq h$ ;
\item \label{item:uniciteBK}
si $\MK_1$ et $\MK_2$ sont deux modules de Breuil--Kisin tels que :
$$\bbM^\star(V_{|\Goo}) \simeq \ocE \otimes_{\SK} \MK_1 \simeq
 \ocE \otimes_{\SK} \MK_2$$
alors le second isomorphisme ci-dessus identifie $\MK_1$ à 
$\MK_2$ à l'intérieur de $\bbM^\star(V_{|\Goo})$.
\end{enumerate}
\end{thm}

\begin{proof}
En reprenant les notations de \cite{Ki3}, les hypothèses du théorème
assurent que $R = R^{\leqslant h}$. La 
première assertion suit ainsi du corollaire 1.7 de \emph{loc. cit.}

La seconde assertion, quant à elle, suit de la proposition 2.1.12
de \cite{Ki-Crys}.
\end{proof}

Les représentations cristallines étant en particulier semi-stables, le 
théorème \ref{thm:Eusst} vaut \emph{a fortiori} si les représentations 
$\Qpbar \otimes_R V$ sont cristallines pour tout morphisme d'anneaux 
$R\rightarrow \Qpbar$. Dans ce cas, la situation est encore meilleure car 
le théorème de pleine fidélité de Kisin (voir corollaire 2.1.14 de 
\cite{Ki-Crys}) implique que la représentation $V$ de $G_K$ est entièrement caractérisée par sa 
restriction à $\Goo$. Par ailleurs, soulignons qu'il n'est pas vrai, en 
général, qu'une représentation de $E(u)$-hauteur finie de $\Goo$ se 
prolonge en une représentation cristalline de $G_K$. Toutefois, ceci se 
produit dans le cas particulier notable des représentations de 
$E(u)$-hauteur inférieure ou égale à $ 1$, comme le précise le théorème suivant.

\begin{thm}
\label{thm:hauteur1}
Toute $R$-représentation libre de rang fini de $\Goo$ de $E(u)$-hauteur 
inférieure ou égale à $1$ se prolonge de manière canonique à $G_K$. Ce prolongement est
fonctoriel et commute aux changements de base.

De plus, si $R$ est l'anneau des entiers d'une extension finie de $E$,
ce prolongement est caractérisé par le fait qu'il soit cristallin.
\end{thm}

\begin{proof}
L'existence du prolongement et son caractère \og canonique \fg\ découlent 
d'une variante à coefficients du lemme 5.1.2 de \cite{BCDT} après une 
traduction entre le langage des modules de Breuil--Kisin et celui des 
modules fortement divisibles de Breuil.

La caractérisation par le caractère cristallin est, quant à elle, une
conséquence du corollaire 2.1.14 de \cite{Ki-Crys}.
\end{proof}

\subsubsection{Données de descente}
\label{ssec:descente}

Une représentation de $G_K$ est dite \textit{potentiellement 
semi-stable} si elle devient semi-stable après restriction à un 
sous-groupe d'indice fini, c'est-à-dire après restriction à un sous-groupe de la forme $G_L 
= \Gal(\Kbar/L)$, où $L$ est une extension finie de $K$.
Une approche usuelle pour étudier une telle représentation $V$ à l'aide 
de la théorie de Kisin est de munir le module de Breuil--Kisin de $V_{| 
G_L}$ (qui est semi-stable) d'une donnée supplémentaire, appelée 
\textit{donnée de descente}, qui rend compte du prolongement de l'action 
à $G_K$.

À partir de maintenant, nous nous restreignons au cas où $L$ est de la 
forme $L = K(\sqrt[n]{\pK})$ pour un certain entier $n$ premier à $p$. 
Nous supposons de plus que $K$ contient une racine primitive $n$-ième de 
l'unité. L'extension $L/K$ est alors galoisienne et son groupe de Galois 
est canoniquement isomorphe au groupe $\mu_n(K)$ des racines $n$-ièmes de 
l'unité dans $K$. Notons $\eta_n : \Gal(L/K) \to \mu_n(K)$ cet 
isomorphisme.

Dans la partie \ref{sec:methode}, nous considérerons notamment la situation suivante :
$F$ est l'extension non ramifiée de degré $f$ de $\Qp$ ; l'uniformisante $\varpi_F$ choisie est $-p$ ; $n$ est $p^f - 1$.
Le caractère $\eta_n$ est alors le caractère $\varepsilon_f$ introduit dans la partie \ref{sssec:notations}.

Nous ajoutons un indice $K$ ou $L$ à toutes les données associées aux 
corps de base $K$ et $L$ respectivement. Par exemple, nous notons $u_K$ et $u_L$ les deux versions de la variable que nous
notions $u$ jusqu'alors et nous utilisons les notations $\SK_K$ et $\SK_L$ 
(resp. $\ocEK$ et $\ocEL$) pour les deux versions de l'anneau $\SK$ 
(resp. $\ocE$). En particulier, nous avons $\SK_K = W[[u_K]]$ et $\SK_L = 
W[[u_L]]$. Remarquons, qu'étant donné que l'extension $L/K$ est 
totalement ramifiée, l'anneau $W$ est le même que nous travaillions avec 
$K$ ou $L$ ; il est donc inutile de le décorer d'un indice 
supplémentaire.

Nous supposons de plus, à partir de maintenant, que les uniformisantes 
$\pi_K$ et $\pi_L$, ainsi que leurs racines $p^s$-ièmes $\pi_{s,K}$ et 
$\pi_{s,L}$ sont choisies de façon à ce que $\pi_{s,L}^n$ soit $\pi_{s,K}$ 
pour tout entier $s$. Le polynôme minimal $E_L(u_L)$ de $\varpi_L$ sur 
$K_0$ est donc égal à $E_K(u_L^n)$. Nous identifions $\SK$ (resp. $\ocEK$) à 
un sous-anneau de $\SK_L$ (resp. $\ocEL$) en envoyant la variable $u_K$ 
sur $u_L^n$. \emph{Via} cette identification, les polynômes $E_K(u_K)$ 
et $E_L(u_L)$ se correspondent. Par ailleurs, l'extension $\Loo$ 
(obtenue en ajoutant à $L$ tous les $\pi_{s,L}$) est galoisienne sur
$\Koo$ et son groupe de Galois s'identifie à celui de $L/K$ et donc
encore à $\mu_n(K)$ en suivant l'homomorphisme $\omega$.

Définissons une action de $\Gal(L/K) \simeq \Gal(\Loo/\Koo)$ sur $\ocEL$ 
en faisant agir un élément $g$ de $\Gal(L/K)$ sur la variable $u_L$ 
par multiplication par $\eta_n(g)$ considéré comme un élément de
$W$. Concrètement, nous avons :
$$g \cdot \sum_{i \in \Z} a_i u_L^i = \sum_{i \in \Z} 
\eta_n(g)^i a_i u_L^i.$$
Clairement, cette action stabilise $\SK_L$. De plus, en utilisant
que $\eta_n$ est d'ordre exactement $n$, il est facile de déterminer les
points fixes :
$$H^0(\Gal(L/K), \ocEL) = \ocEK \quad \text{et} \quad
H^0(\Gal(L/K), \SK_L) = \SK_K.$$
Si $R$ est une $\Zp$-algèbre locale noetherienne complète, nous prolongeons 
l'action de $\Gal(L/K)$ au produit tensoriel $R \hat\otimes_{\Zp} \ocEL$ en 
convenant que $\Gal(L/K)$ agit trivialement sur $R$. Si $R$ est, de
surcroît, une $W$-algèbre, nous disposons de l'isomorphisme 
\eqref{eq:decompocER} :
$$R \hat\otimes_{\Zp} \ocEL \simeq 
\textstyle \prod^{f-1}_{i \in 0} R \hat\otimes_{\iota \circ 
\varphi^{-i}, W} \ocEL.$$
\emph{Via} cette identification, l'action de $\Gal(L/K)$ se fait
composante par composante et, sur le $i$-ième facteur vu à travers
l'isomorphisme \eqref{eq:ocERexpl}, elle est donnée par la formule :
$$g \cdot \sum_{j \in \Z} a_j u_L^j = \sum_{j \in \Z} 
(\varphi^{-i} \circ \eta_n)(g)^j \: a_j \: u_L^j.$$

Si $V$ une $R$-représentation de $\GooK$ qui est libre de rang fini, nous 
nous apercevons en revenant aux définitions que $\bbM^\star(V_{|\GooL})$ (qui est 
un  $\varphi$-module sur $\ocEL \hat\otimes_{\Zp} R$) hérite d'une 
action semi-linéaire de $\Gal(L/K)$ et que :
\begin{equation}
\label{eq:MKML}
\bbM^\star(V) = H^0(\Gal(L/K), \bbM^\star(V_{|\GooL})).
\end{equation}
Nous en déduisons, en particulier, que le foncteur qui à $V$ comme 
précédemment associe le $\varphi$-module $\bbM^\star(V_{|\GooL})$ muni de
l'action supplémentaire de $\Gal(L/K)$ établit une équivalence de
catégorie entre la catégories des $R$-représentations libres de rang
fini de $\GooK$ et celle des $\varphi$-modules sur $R 
\hat\otimes_{\Zp} \ocEL$ équipés d'une action semi-linéaire de $\Gal(L/K)$.

Examinons maintenant plus attentivement le cas où la représentation
$V_{|\GooL}$ est de $E_L(u_L)$-hauteur finie. Par définition, il existe 
alors un module de Breuil--Kisin $\MK_L$ tel que 
$$\bbM^\star(V_{|G_L}) = \ocEL \otimes_{\SK_L} \MK_L.$$
Si nous supposons de plus que $R$ est plat sur $\Zp$, ce module est 
unique d'après l'alinéa \eqref{item:uniciteBK} du théorème 
\ref{thm:Eusst}. Nous en déduisons qu'il hérite par restriction d'une action 
de $\Gal(L/K)$. Toutefois, l'espace des points fixes $H^0(\Gal(L/K),
\MK_L)$ ne jouit généralement pas de bonnes propriétés\footnote{Typiquement,
ce n'est pas en général un module de Breuil--Kisin et, de fait, la 
représentation $V$ n'a pas de raison d'être de $E_K(u_K)$-hauteur
finie.}. Pour conserver
la trace de la semi-stabilité de $V$, il est donc nécessaire de travailler
avec des modules de Breuil--Kisin munis de données de descente.

\subsection{Sur la restriction à $\Goo$}
\label{ssec:nuisibilite}

La théorie de Breuil--Kisin permet de contrôler de manière efficace 
l'action de $\Goo$ sur une représentation (potentiellement) semi-stable. 
Dans sa version la plus simple --- telle que nous venons de la présenter 
--- elle ne dit toutefois pas grand chose sur l'action complète du groupe 
$G_K$. L'objectif de cette partie est d'étudier, dans les situations qui 
nous intéresserons par la suite, dans quelle mesure nous pouvons nous 
contenter de travailler uniquement avec l'action du sous-groupe $\Goo$.

Fixons une suite compatible $(\zeta_s)_{s \geq 0}$ de racines primitives 
$p^s$-ièmes de l'unité dans $\Kbar$, c'est-à-dire, une suite vérifiant 
$\zeta_0 = 1$, $\zeta_1 \neq 1$ et $\zeta_{s+1}^p = \zeta_s$ pour tout 
entier $s \geq 0$. Ce choix étant fait, pour tout $g \in G_K$, il existe 
un unique élément $c(g)$ de $\Zp$ tel que :
$$g(\pi_s) = \zeta^{c(g)}_s \pi_s
\quad \text{pour tout } s \text{ dans } \N.$$
Cette association définit une application $c : G_K \to \Zp$ vérifiant, 
pour tous $g$ et $g'$ dans $G_K$, la relation de cocycle $c(g g') = c(g) 
+ \epsilon(g) c(g')$ (où nous rappelons que $\epsilon$ désigne le
caractère cyclotomique).

Notons $K(\zeta_{\infty})$ l'extension de $K$ engendrée par tous les 
$\zeta_s$ ; il s'agit d'une extension galoisienne dont le groupe de 
Galois, noté $\Gamma$, s'identifie à un sous-groupe ouvert de 
$\Zp^\times$ \emph{via} $\epsilon$. Le groupe de Galois de l'extension 
$\Koo(\zeta_{\infty}) / K(\zeta_{\infty})$ s'identifie, quant à lui, à 
$\Zp$ \emph{via} le cocycle $c$ (qui, bien sûr, en restriction à 
$\Gal(\Kbar/K(\zeta_{\infty}))$ est un morphisme de groupes). Enfin, le 
lemme 5.1.2 de \cite{Li2} affirme que les extensions $\Koo$ et 
$K(\zeta_{\infty})$ sont linéairement disjointes au-dessus de $K$ 
(rappellons que $p$ est choisi différent de $2$). Nous en déduisons que 
le groupe de Galois de l'extension $\Koo(\zeta_{\infty}) / K$ est 
isomorphe au produit semi-direct $\Zp \rtimes \Gamma$, où $\gamma \in 
\Gamma$ agit sur $\Zp$ par multiplication par $\epsilon(\gamma)$. Soit 
$\tau$ un élément du sous-groupe d'inertie sauvage de $G_K$ dont l'image 
dans $\Zp \rtimes \Gamma$ est $(1,1)$. En particulier, nous avons 
$\epsilon(\tau) = 1$, $c(\tau) = 1$ et la suite des $\tau^{p^s}$ converge 
vers l'identité lorsque $s$ tend vers l'infini. Ce dernier fait permet de 
définir $\tau^a$ pour tout élément $a$ dans $\Zp$.

\subsubsection{Une suite exacte de type inflation-restriction}
\label{sssec:inflrest}

Soit $R$ une $\Zp$-algèbre locale d'idéal maximal $\m_R$ supposée 
noetherienne, séparée et complète pour la topologie définie par $\m_R$. 
Soit encore $V$ une $R$-représentation de $G_K$ qui est de type fini 
comme $R$-module.
Appellons $\tau_V$ l'automorphisme de $V$ donné par l'action de $\tau$
et, pour tout entier $n$ positif ou nul,  posons :
$$[n]_{\tau,V} = \Id_V + \tau_V + \tau_V^2 + \cdots + \tau_V^{n-1}
\in \End(V)$$
Nous avons la relation : %
\begin{equation}
\label{eq:binomequantique}
[a+b]_{\tau,V} = [a]_{\tau,V} + \tau_V^a \circ [b]_{\tau,V}
\end{equation}
pour tous entiers $a$ et $b$.

\begin{lem}\label{lemunique}
L'application 
$$\N \to \End(V), \,\, n \mapsto [n]_{\tau,V}$$
s'étend de manière unique en une application continue $\Zp \to \End(V)$.
\end{lem}

\begin{proof}
L'unicité du prolongement résulte de la densité de $\N$ dans $\Zp$ et
de la séparation de $R$.

Montrons à présent l'existence. Considérons un élément $a$ de $\Zp$ ainsi 
qu'une suite $(a_s)$ de nombres entiers positifs ou nuls tels que $a_s 
\equiv a \pmod{p^s}$. Nous devons démontrer que la suite des $[a_s]_{\tau,V}$
converge. Or, en écrivant $a_{s+1} = a_s + p^s b$ pour un certain
entier $b$, nous vérifions que nous avons la relation :
\begin{equation}
\label{eq:asquantique}
[a_{s+1}]_{\tau,V} = 
[a_s]_{\tau,V} + \tau^{a_s} \circ \big(\Id + \tau^{p^s} + \cdots +
\tau^{(b-1)p^s} \big) \circ [p^s]_{\tau,V}.
\end{equation}
Il suffit donc de montrer que la suite des $[p^s]_{\tau,V}$ converge
vers l'endomorphisme nul de $V$. Or, par continuité de l'action de 
$G_K$, il existe un entier $n$ tel que l'endomorphisme $[p^n]_{\tau,V}$ 
soit congru à l'identité modulo $\m_R$. Comme le corps résiduel de $R$ 
est de caractéristique $p$, nous en déduisons que pour tout $s \geq n$, 
l'endomorphisme
$$f_s = \Id + [p^s]_{\tau,V} + ([p^s]_{\tau,V})^2 +
\cdots + ([p^s]_{\tau,V})^{p-1}$$
prend ses valeurs dans $\m_R \cdot V$. De la relation 
$[p^{s+1}]_{\tau,V} = [p^s]_{\tau,V} \circ f_s$, nous déduisons par 
récurrence sur $s$ que, pour $s \geq n$, l'image de $[p^s]_{\tau,V}$ est 
incluse dans $\m_R^{s-n} \cdot V$. Nous concluons en reportant cette
information dans \eqref{eq:asquantique}.
\end{proof}

\begin{rem}
Si $a$ est un élément de $\Zp$, nous utilisons encore la notation 
$[a]_{\tau,V}$ pour désigner l'image de $a$ par l'unique prolongement
promis par le lemme \ref{lemunique}.
De plus, avec cette notation, la relation \eqref{eq:binomequantique}
s'étend par continuité à tout couple $(a,b)$ d'éléments de $\Zp$.
\end{rem}

Introduisons l'espace $V_\tau$ défini par :
$$V_\tau = \Big\{ \,\, x \in V \quad \Big| \quad
g(x) = [\epsilon(g)]_{\tau,V}(x), \,\, 
\forall g \in \Goo \,\,\Big\}.$$
Nous vérifions immédiatement que l'application $\tau - 1$ induit par
restriction et corestriction un morphisme de $H^0(\Goo, V)$ dans 
$V_\tau$. 

\begin{lem}
Pour tout élément $x$ de $V_\tau$, l'application
$$\gamma_x : G_K \to V, \,\, 
g \mapsto [c(g)]_{\tau,V}(x)$$
est un $1$-cocycle.
\end{lem}

\begin{proof}
Soient $g$ et $g'$ deux éléments de $G_K$. Calculons :
\begin{eqnarray}
\gamma_x(g g') & = & [c(gg')]_{\tau,V}(x) \label{eq:cocycle1} \\
& = & [c(g) + \epsilon(g) c(g')]_{\tau,V}(x) \nonumber \\
& = & \big([c(g)]_{\tau,V} + \tau_V^{c(g)} \circ [\epsilon(g) c(g')]_{\tau,V}\big)(x) \nonumber \\
& = & \gamma_x(g) + \tau_V^{c(g)} \circ [\epsilon(g) c(g')]_{\tau,V}(x). \nonumber
\end{eqnarray}
Posons $g_0 = \tau^{-c(g)} g$. Un calcul immédiat donne 
$\epsilon(g_0) = \epsilon(g)$ et $c(g_0) = 0$. Cette dernière 
égalité signifie que $g_0$ est dans le sous-groupe $\Goo$. Nous en 
déduisons la relation de commutation $g_0 \tau = \tau^{\epsilon(g_0)} 
g_0 = \tau^{\epsilon(g)} g_0$. Ainsi, en supposant en outre 
pour commencer que $c(g')$ et $\epsilon(g)$ sont des entiers 
positifs ou nuls, nous obtenons :
\begin{eqnarray*}
g_0 \cdot \gamma_x(g')
& = & \sum_{i=0}^{c(g')-1} g_0 \tau^i x 
\, = \, \sum_{i=0}^{c(g')-1} \tau^{i \epsilon(g)} g_0 x \\
& = & \sum_{i=0}^{c(g')-1} \tau^{i \epsilon(g)} 
      \sum_{j=0}^{\epsilon(g)-1} \tau^j x \\
& = & \sum_{i=0}^{\epsilon(g) c(g')-1} \tau^i x
\, = \, [\epsilon(g) c(g')]_{\tau,V}(x),
\end{eqnarray*}
la troisième égalité résultant de l'appartenance de $x$ à $V_\tau$.
Par continuité, l'égalité que nous venons d'établir vaut encore lorsque 
$c(g')$ et $\epsilon(g)$ sont, plus généralement, des éléments de 
$\Zp$. En appliquant $\tau^{c(g)}$ à cette égalité, nous obtenons :
$$g \cdot [c(g')]_{\tau,V}(x) = 
\tau_V^{c(g)} \circ [\epsilon(g) c(g')]_{\tau,V}(x).$$
En combinant ceci avec \eqref{eq:cocycle1}, nous trouvons finalement
$\gamma_x(g g') = \gamma_x(g) + g \cdot
\gamma_x(g')$,
c'est-à-dire la relation de cocycle que nous devions établir.
\end{proof}

Notons $\delta$ l'application de $V_\tau$ dans $H^1(G_K, V)$ qui à
$x$ associe le $1$-cocycle $\gamma_x$.  Nous vérifions immédiatement que
$\delta$ est $R$-linéaire.

\begin{prop}
\label{prop:inflrest}
Soit $V$ une représentation de $G_K$ à coefficients dans $R$. 
La suite de $R$-modules :
$$0 \longrightarrow H^0(G_K, V) \longrightarrow H^0(\Goo, V) 
\stackrel{\tau{-}1}\longrightarrow V_\tau
\stackrel{\delta}\longrightarrow
H^1(G_K, V) \longrightarrow H^1(\Goo, V)$$
est exacte.
\end{prop}

\begin{proof} 
Le fait que la composée de deux flèches successives soit nulle se vérifie 
immédiatement. Il en va de même de l'injectivité de la flèche $H^0(G_K, 
V) \to H^0(\Goo, V)$. L'exactitude en $H^0(\Goo, V)$ ne pose guère plus 
de difficulté en remarquant que $G_K$ est topologiquement 
engendré par $\Goo$ et $\tau$.

Vérifions l'exactitude en $V_\tau$. Nous considérons pour cela un élément 
$x$ de $V_\tau$ tel que le cocycle associé $\gamma_x$ soit trivial en 
cohomologie. Cela signifie que $\gamma_x$ est un cobord, c'est-à-dire 
qu'il existe un élément $y$ dans $V$ tel que 
$$\gamma_x(g) = g(y) - y, \quad \forall g \in G_K.$$
En particulier, pour $g = \tau$, nous obtenons $x = \tau(y) - y$ et 
pour $g \in \Goo$, nous trouvons $g(y) = y$. L'élément $y$ 
appartient donc à $H^0(\Goo, V)$ et s'envoie sur $x$ par l'application 
$\tau - 1$. Ceci est exactement ce qu'il fallait démontrer.

Il ne reste plus qu'à démontrer l'exactitude en $H^1(G_K, V)$. Soit
$\gamma : G_K \to V$ un $1$-cocycle dont la restriction à $\Goo$ est
triviale en cohomologie. Quitte à modifier $\gamma$ par un cobord --- 
ce qui ne change pas son image dans $H^1(G_K, V)$ --- nous pouvons supposer
que $\gamma$ s'annule sur $\Goo$. Posons $x = \gamma(\tau)$. Nous nous
proposons de démontrer que $\gamma = \gamma_x$. Une récurrence immédiate 
sur $n$ montre que :
$$\gamma(\tau^n) = [n]_{\tau,V} (x)$$
pour tout entier naturel $n$ et, par continuité, l'égalité précédente
vaut encore si $n$ est un entier $p$-adique. Soit $g$ un élement 
de $G_K$. Posons $g_0 = \tau^{-c(g)} g$. Nous avons ainsi 
$c(g_0) = 0$, ce qui signifie que $g_0$ appartient à $\Goo$. 
Nous pouvons ainsi écrire : 
$$\gamma(g) = \gamma(\tau^{c(g)} g_0) = 
[c(g)]_{\tau,V} (x)$$
puisque $\gamma(g_0) = 0$. Nous en déduisons que $\gamma(g) =
\gamma_x(g)$ comme annoncé. Au final, $\gamma$ est dans l'image
de l'application $\delta$ et l'exactitude en $H^1(G_K, V)$ est
démontrée.
\end{proof}

\begin{cor}
\label{cor:inflrest}
Soit $\rhobar$ une $k_E$-représentation de $G_K$ sur laquelle $\tau$ 
agit trivialement. Alors nous avons une suite exacte :
$$0 \longrightarrow H^0(G_K, V(-1)) \longrightarrow H^1(G_K,V)
\longrightarrow H^1(\Goo,V)$$
et le premier terme non nul de cette suite s'identifie encore à 
$H^0(\Goo, V(-1))$.
\end{cor}

\begin{proof}
En revenant aux définitions, nous déduisons de l'hypothèse, d'une part, que 
$H^0(\Goo, V(-1)) = H^0(G_K, V(-1)) \simeq V_\tau$ et, d'autre part, 
que l'application $\tau - 1$ de $H^0(\Goo, V)$ dans $V_\tau$ est nulle.
Le corollaire \ref{cor:inflrest} en découle.
\end{proof}

\begin{rem}
\label{rem:inflrest}
L'hypothèse du corollaire \ref{cor:inflrest} est satisfaite pour les
représen\-tations \emph{irréductibles} de $G_K$. En effet, nous savons que le 
sous-groupe d'inertie sauvage --- et donc, en particulier $\tau$ --- 
agit trivialement sur une telle représentation.
Par extension, l'hypothèse du corollaire est aussi vérifiée si $\rhobar$ 
est une représentation semi-simple ou un produit tensoriel de 
représentations semi-simples.
\end{rem}

\subsubsection{Application aux anneaux de déformations}

Soient $R_1$ et $R_2$ deux $\oE$-algèbres locales noetheriennes complètes, 
d'idéaux maximaux respectifs $\m_1$ et $\m_2$.
Pour $i = 1, 2$, supposons que le morphisme structurel de $R_i$ induit 
un isomorphisme de $k_E$ sur le corps résiduel de $R_i$.
Considérons en outre un morphisme de $\oE$-algèbres $f$ de $R_1$ dans 
$R_2$ qui, sur les corps résiduels, induit l'identité de $k_E$.

Notons $k_E[\varepsilon]$ l'anneau des nombres duaux sur $k_E$ (nous avons 
donc $\varepsilon^2 = 0$).
Pour $i = 1, 2$, l'espace tangent $\Hom_{\oE-\mathrm{alg}} (R_i , 
k_E[\varepsilon]) $ est naturellement isomorphe au dual du $k_E$-espace 
vectoriel $\frac{\m_i}{\pE R_i + \m_i^2} $. En particulier, il hérite 
ainsi d'une structure de $k_E$-espace vectoriel. Notons 
$$
f^* : \Hom_{\oE-\mathrm{alg}} (R_2 , k_E[\varepsilon]) \longrightarrow 
\Hom_{\oE-\mathrm{alg}} (R_1 , k_E[\varepsilon])
$$ 
l'application $k_E$-linéaire induite par $f$ sur les espaces tangents.
Rappelons le lemme classique suivant.

\begin{lem}\label{lem:fsurj}
Le morphisme $f$ est surjectif si et seulement si l'application tangente 
$f^*$ est injective.
\end{lem}

\begin{proof}
Par dualité, l'application $f^*$ est injective si, et seulement si son
application duale :
$$\frac{\m_1}{\pE R_1 + \m_1^2} \longrightarrow
\frac{\m_2}{\pE R_2 + \m_2^2}$$
est surjective. Le lemme en résulte par un argument classique.
\end{proof}

Supposons donnés à présent une $\oE$-algèbre locale noetherienne $R$, 
de corps résiduel $k_E$, ainsi qu'un morphisme de $\oE$-algèbres $f$ de 
$R^\psi(\vv, \ttt , \rhobar)$ dans $R$.
La théorie des espaces de déformations permet d'identifier l'espace 
tangent de $R^\psi(\vv, \ttt , \rhobar)$ à un sous-$k_E$-espace 
vectoriel de $\Ext_{G_K}^1(\rhobar, \rhobar)$.
Nous pouvons ainsi considérer que l'application tangente $f^*$ prend ses 
valeurs dans $\Ext_{G_K}^1(\rhobar, \rhobar)$.

Notons $f^{\sharp}$ l'application composée :
$$f^{\sharp} : \Hom_{\oE-\mathrm{alg}} (R , k_E[\varepsilon]) 
\stackrel{f^*}{\longrightarrow} \Ext_{G_K}^1(\rhobar, \rhobar) 
\longrightarrow \Ext_{\Goo}^1(\rhobar, \rhobar),$$
où la deuxième flèche est induite par la restriction à $\Goo$ sur les 
représentations.

\begin{rem}
Nous verrons au paragraphe \ref{sssec:calculExt1} qu'au moins dans les cas traités 
dans cet article, l'espace $\Ext_{\Goo}^1(\rhobar, \rhobar)$ et la 
composée $f^\sharp$ se décrivent explicitement en termes de 
$\varphi$-modules.
\end{rem}

Nous déduisons du lemme~\ref{lem:fsurj} que l'injectivité de $f^\sharp$ 
implique la surjectivité de $f$ et que la réciproque est vraie sous 
l'hypothèse supplémentaire que l'application naturelle de 
$\Ext_{G_K}^1(\rhobar, \rhobar)$ dans $\Ext_{\Goo}^1(\rhobar, \rhobar)$ 
est injective.
Le lemme ci-après donne une condition suffisante pour que cette hypothèse 
soit satisfaite.

\begin{lem}
\label{lem:rinj}
Soit $\rhobar$ une $k_E$-représentation absolument irréductible de $G_K$ 
de dimension $d$.
\begin{enumerate}
\item  Si $\rhobar$ n'est pas isomorphe à $\rhobar(1)$, alors l'application :
$$r : \Ext_{G_K}^1(\rhobar, \rhobar) \longrightarrow \Ext_{\Goo}^1(\rhobar, \rhobar)$$
est injective.
\item Si $\rhobar$ est isomorphe à $\rhobar(1)$, alors $p-1$ divise $de_K$, où
$e_K$ désigne l'indice de ramification absolue de $K$.
\end{enumerate}
\end{lem}

\begin{rem}
En combinant les deux assertions du lemme, nous obtenons l'injectivité de 
$r$ sous l'hypothèse que $p-1$ ne divise pas le produit de l'indice de
ramification absolue de $K$ et de la dimension de $\rhobar$.
\end{rem}

\begin{proof}
Nous démontrons d'abord la première assertion.
Posons $V = \rhobar^* \otimes \rhobar$. L'espace
$\Ext^1_G(\rhobar, \rhobar)$ s'identifie canoniquement à $H^1(G, 
V)$ lorsque $G$ désigne l'un des deux groupes $G_K$ ou $\Goo$. Ainsi, 
d'après le corollaire \ref{cor:inflrest}, le noyau de $r$ s'identifie
à $H^0(G_K, V(-1))$, soit encore à $\Hom_{G_K}(\rhobar, \rhobar(-1))$.
Comme $\rhobar$ est irréductible et non isomorphe à $\rhobar(-1)$ 
d'après l'hypothèse, cet espace s'annule et le
morphisme $r$ est injectif.

Pour la deuxième assertion, notons $e$ (resp. $f$) l'indice de 
ramification (resp. le degré résiduel) de l'extension $K / \Qp$. 
Considérons $K_d$ l'unique extension non ramifiée de $K$ de degré $d$ 
vivant à l'intérieur de $\Qpbar$. D'après la classification des 
représentations absolument irréductibles de $G_K$, nous savons que :
$$\rhobar \simeq \Ind_{G_{K_d}}^{G_K} (\omega_{df}^h \otimes \chi_\nr)$$
où $\omega_{df}$ désigne un caractère fondamental de Serre de niveau 
$df$, $h$ est un entier et $\chi_\nr$ est un caractère non ramifié. À 
partir de là, la condition d'isomorphisme entre $\rhobar$ et 
$\rhobar(1)$ implique l'existence d'un entier $n$ vérifiant la 
congruence :
$$p^n h \equiv h + e \cdot \frac{p^d-1}{p-1} \pmod{p^d-1}.$$
Il résulte de ceci que $p-1$ divise $e \cdot \frac{p^d-1}{p-1}$, ce qui 
ne peut se produire que si $p-1$ divise $de$.
\end{proof}

Nous en arrivons au résultat principal de cette partie. Pour l'énoncer, 
considérons une extension finie $E$ de $\Qp$ et une représentation 
galoisienne $\rhobar$ de $G_K$ de dimension finie $d$ à coefficients 
dans le corps résiduel $k_E$ de $E$. Nous supposons que $\rhobar$ n'a 
pas d'endomorphisme autre que les multiplications par les éléments de 
$k_E^\times$. Soient encore $\vv$ un type de Hodge--Tate, $\ttt$ un type 
galoisien ainsi que $\chi: G_K \to \oE^\star$ un caractère %
relevant le déterminant de $\rhobar$. 
L'anneau $R_{\cri}^{\det = \chi}(\vv, \ttt, \rhobar)$ 
qui paramètre les déformations potentiellement cristallines\footnote{La 
proposition \ref{prop:factorisationf} ci-après vaudrait encore si nous 
remplacions \og cristalline \fg\ par \og semi-stable \fg.} de 
$\rhobar$ de type $(\vv, \ttt)$ et de déterminant $\chi$ a alors un 
sens.

\begin{prop}\label{prop:factorisationf}
Nous reprenons les notations ci-dessus. Nous supposons en outre que 
$\rhobar$ est absolument irréductible, n'est pas isomorphe à 
$\rhobar(1)$ et que le type $\ttt$ se factorise par le groupe de
Weil d'une extension modérément ramifiée de $K$.

Soient $R$ une $\oE$-algèbre locale complète noetherienne de corps 
résiduel $k_E$ et $f$ un morphisme de $R^{\det = \chi}_\cri(\vv, \ttt, \rhobar)$ 
dans $R$.
Soient $V$ la $R^{\det = \chi}_\cri(\vv, \ttt, \rhobar)$-représentation universelle de 
$G_K$ et $V_R = R \otimes_{R^{\det = \chi}_\cri(\vv, \ttt, \rhobar)} V$. Nous 
supposons que la restriction de $V_R$ à $\Goo$ est définie sur un 
sous-$\oE$-algèbre $R'$ de $R$, c'est-à-dire qu'il existe une 
$R'$-représentation $V_{R'}$ de $\Goo$ qui est libre comme $R'$-module 
et telle que $V_R \simeq R \otimes_{R'} V_{R'}$. Alors, le morphisme $f$ 
prend ses valeurs dans $R'$.
\end{prop}

\begin{rem}
Par l'équivalence de catégories du théorème~\ref{thm:equivFontaine}, 
dire que la restriction de $V_R$ à $\Goo$ est définie sur $R'$ revient à 
dire que le $\varphi$-module associé $M_R$ à $V_R$ s'écrit sous la forme 
$(R \hat \otimes_{\Zp} \ocE) \otimes_{R' \hat \otimes_{\Zp} \ocE} 
M_{R'}$ pour un certain $\varphi$-module $M_{R'}$ \emph{libre} sur $R' 
\hat \otimes_{\Zp} \ocE$.
\end{rem}

\begin{proof}
Quitte à tordre par une puissance du caractère cyclotomique, nous pouvons
supposer que tous les poids de Hodge--Tate qui interviennent dans 
$\vv$ sont positifs ou nuls. Notons $h$ le plus grand d'entre eux.
Soit $L$ une extension finie modérément ramifiée de $K$ par laquelle
$\ttt$ se factorise. Pour simplifier l'exposition\footnote{Quitte à 
agrandir $L$, dans le cas général, nous pouvons toujours supposer que 
$L$ s'écrit comme la composée d'une extension non ramifiée et d'une 
extension du type que nous considérons ici. Pour éténdre la 
démonstration au cas général, il s'agit de prendre en compte la partie 
non ramifiée dans la donnée de descente, ce qui est standard et ne pose 
pas de difficultés particulières.}, nous supposons que $L$ s'obtient à 
partir de $K$ en ajoutant une racine $n$-ième $\pi_L$ de $\pi$ pour un 
entier $n$, premier avec $p$, tel que $K$ contienne une racine primitive 
$n$-ième de l'unité. Nous sommes ainsi dans la situation du \S 
\ref{ssec:descente} (dont nous reprenons les notations) et pouvons 
décrire les représentations de $G_K$ qui deviennent semi-stables en 
restriction à $G_L$ à l'aide de modules de Breuil--Kisin munis de 
données de descente.

Notons par ailleurs que les représentations paramétrées par $R_\cri^{\det = \chi} 
(\vv, \ttt, \rhobar)$ sont cristallines et de $E_L(u_L)$-hauteur inférieure ou égale à
$h$ en restriction à $\GooL$. Soit $\rhobar_\infty$ la restriction de 
$\rhobar$ à $\Goo$. En copiant la démonstration du théorème 3.2 de 
\cite{Kim}, nous construisons une $\oE$-algèbre complète locale 
noetherienne $R^{\det = \chi}(h, \star, \rhobar_\infty)$ qui paramètre les 
déformations de $\rhobar_\infty$ qui sont des représentations de 
déterminant $\chi$ dont la restriction à $\GooL$ est de $E_L(u_L)$-hauteur 
inférieure ou égale à $ h$. Nous avons un morphisme canonique
$$g : R^{\det = \chi}(h, \star, \rhobar_\infty)
\longrightarrow R_\cri^{\det = \chi}(\vv, \ttt, \rhobar).$$
L'application tangente de $g$ s'identifie à l'application canonique
$$r : \Ext_{G_K}^1(\rhobar, \rhobar) \longrightarrow 
\Ext_{\Goo}^1(\rhobar, \rhobar)$$
convenable restreinte et corestreinte. Elle est donc injective d'après 
le lemme \ref{lem:rinj}. Nous en déduisons, par le lemme \ref{lem:fsurj}, 
que $g$ est surjective.

Par ailleurs, comme la restriction de $V_R$ à $\GooL$ est de 
$E_L(u_L)$-hauteur inférieure ou égale à $h$; il en va de même de $V_{R'}$. Par la 
propriété universelle définissant l'anneau $R^{\det = \chi}(h, \star, 
\rhobar_\infty)$, la représentation $V_{R'}$ correspond à un morphisme 
d'anneaux
$$f'_\infty : R^{\det = \chi}(h, \star, \rhobar_\infty) \longrightarrow R'.$$
En outre, en notant $\iota$ l'inclusion canonique de $R'$ dans $R$, nous
déduisons du fait que $V_{R'}$ redonne $V_R$ après extension des scalaires
que $\iota \circ f'_\infty = f \circ g$. Pour conclure, il suffit de 
construire un morphisme d'anneaux $f'$ de $R_\cri^{\det = \chi}(\vv, \ttt, \rhobar)$ 
dans $R'$ rendant commutatif le diagramme suivant :
$$\xymatrix @C=50pt @R=10pt {
& R^{\det = \chi}_\cri(\vv, \ttt, \rhobar) \ar[dr]^-{f} 
\ar@{.>}^-{f'} [dd] \\
R^{\det = \chi}(h, \star, \rhobar_\infty) 
\ar[dr]_-{f'_\infty} \ar@{->>}[ur]^-{g} & & R \\
& \hspace{0.2cm} R' \hspace{0.2cm} \ar@{^(->}[ur]_-{\iota}}$$
Nous construisons $f'$ par une chasse au diagramme élémentaire. Soit 
$x$ un élément de $R^{\det = \chi}_\cri(\vv, \ttt, \rhobar)$. Par surjectivité 
de $g$, il s'écrit sous la forme $g(y)$ pour un certain $y$ dans 
$R^{\det = \chi}(h, \star, \rhobar_\infty)$. En outre, si $y'$ est un 
autre antécédent de $x$ par $g$, les images de $y$ et $y'$ par $\iota 
\circ f'_\infty$ coïncident dans $R$ par commutativité du diagramme. 
Comme $\iota$ est injectif, cela implique que $f'_\infty(y) = 
f'_\infty(y')$. Nous pouvons ainsi définir $f'(x)$ sans ambiguïté en 
posant $f'(x) = f'_\infty(y)$. Nous concluons enfin en vérifiant que 
$f'$ est un morphisme d'anneaux et fait commuter le diagramme.
\end{proof}

\begin{rem}
Dans cet article, nous appliquerons la proposition 
\ref{prop:factorisationf} uniquement pour des types $\vv$ correspondant 
à des poids de Hodge--Tate dans $\{0, 1\}$.
Or, pour de telles représentations, une autre approche basée sur le 
théorème \ref{thm:hauteur1} serait possible. L'approche que nous allons 
suivre, bien que probablement moins directe, nous paraît toutefois plus 
intéressante car susceptible de s'étendre à des espaces de déformations
qui ne sont pas (potentiellement) Barsotti--Tate, mais peuvent avoir
des poids de Hodge--Tate plus élevés.
\end{rem}

\section{Méthode de calcul des déformations potentiellement Barsotti--Tate}
\label{sec:methode}

\newcommand{\NP}{\text{\rm NP}}

Nous reprenons à présent les notations de la partie \ref{sec:motivations}. 
Notamment, la lettre $F$ (resp. $F'$) désigne l'unique extension non ramifiée 
de $\Qp$ de degré $f$ (resp. $2f$), nous posons $q = p^f$ et notons $L$ le 
corps obtenu en adjoignant à $F$ une racine $(q-1)$-ième de $(-p)$, notée 
$\pi_L$. Afin de pouvoir appliquer la théorie de Breuil--Kisin, nous 
choisissons en outre un système compatible $(\pi_{L,n})$ de racines 
$p^n$-ièmes de $\pi_L$ et posons $\pi_{F,n} = \pi_{L,n}^{q-1}$ pour tout $n$. 
La famille des $\pi_{F,n}$ forme alors un système compatible de racines 
$p^n$-ièmes de $(-p)$. Nous notons $F_\infty$ (resp. $L_\infty$) l'extension 
de $F$ (resp. de $L$) engendrée par tous les $\pi_{F,n}$ (resp. tous les 
$\pi_{L,n}$). Enfin, nous posons $\Goo = \Gal(\Qpbar/ F_\infty)$.

Le but de cette partie est de présenter une méthode pour déterminer les 
anneaux de déformations $R^{\psi}(\vv_0, \ttt, \rhobar)$ pour des valeurs
particulières de $\vv_0$, $\ttt$ et $\psi$ :
\begin{itemize}
\item[$\bullet$]  pour tout plongement $\tau$ de $F$ dans $E$, ${\vv_0}_\tau = (0,2)$ 
(cas potentiellement Barsotti--Tate) ;
\item[$\bullet$] le type galoisien $\ttt$ s'écrit sous la forme $(\eta\oplus \eta')_{| I_F}$, où $\eta$ et
$\eta'$ sont deux caractères distincts 
de $\Gal(L/F)$ dans $\oE^\times $ ;
\item[$\bullet$]  le choix de $\psi$ est compatible avec les $\ttt$ et $\vv_0$
ci-dessus (voir relation (\ref{relpsi}), \S\ref{sssec:defgal}).
\end{itemize}
Cette méthode est largement inspirée de précédents travaux de Breuil et 
Mézard (\cite{BM1}, \cite{BM2}).
Ces deux approches présentent cependant une différence notable : alors 
que Breuil et Mézard emploient des modules fortement divisibles, nous 
travaillons avec les modules de Breuil--Kisin, dont la manipulation nous 
paraît plus aisée et nous permet d'affaiblir les hypothèses.

\begin{rem'}
Les types  galoisiens $\ttt$ considérés ci-dessus étant non scalaires, 
les déformations potentiellement semi-stables paramétrées par 
$R^\psi(\vv_0, \ttt, \rhobar)$ sont en fait potentiellement cristallines.
\end{rem'}

Afin de rester plus proche de \cite{BM2}, nous allons utiliser dans la
suite la version covariante du foncteur de Breuil--Kisin au lieu de sa
version contravariante définie au \S \ref{sec:outils}. Précisément, si $V$ est une représentation de $G_\infty$ ou de 
$\Gal(\Qpbar/L_\infty)$, nous posons :
$$\bbM(V) = \bbM^\star(V^\star(1)),$$
où $V^\star$ désigne la représentation duale de $V$. Le foncteur $\bbM$
ainsi obtenu est alors covariant. À partir de maintenant, lorsque nous
parlerons du $\varphi$-module ou du $\varphi^f$-module associé à $V$, 
ou, réciproquement, de la représentation associée à un $\varphi$-module,
un $\varphi^f$-module ou un module de Breuil--Kisin, ce sera toujours
au sens du foncteur \emph{covariant} $\bbM$.

Le plan de cette partie est le suivant.
Nous classifions au \S \ref{ssec:classification} les modules de 
Breuil--Kisin associés à des représentations potentiellement 
Barsotti--Tate de type galoisien comme ci-dessus. Nous expliquons ensuite au \S 
\ref{ssec:calculdeform} comment déduire de cette classification une 
méthode de calcul des espaces de déformations correspondants. Cela nous 
amène notamment à déterminer la réduction modulo $p$ des représentations 
associées à des modules de Breuil--Kisin (\S \ref{sssec:represid}) et à 
calculer une base explicite de $\Ext^1_{G_\infty}(\rhobar,\rhobar)$, 
groupe qui nous permet de contrôler l'espace tangent aux déformations de 
la représentation $\rhobar$ (\S \ref{sssec:calculExt1}).

\subsection{Un théorème de classification des modules de Breuil--Kisin}
\label{ssec:classification}

Dans ce paragraphe, nous énonçons et démontrons un raffinement de la 
proposition~5.2 de \cite{Br2}. Hormis le fait que nous l'exprimons 
dans le langage des modules de Breuil--Kisin (et non dans celui des
modules fortement divisibles), notre résultat se distingue de celui
de \emph{loc. cit.} par deux aspects essentiels : premièrement, nous traitons le cas d'un anneau de coefficients $R$
qui est une $\oE$-algèbre locale noetherienne complète \emph{plate}
quelconque (et non pas uniquement le cas $R = \oE$) et,
deuxièmement, nous affaiblissons l'hypothèse de généricité sur le type galoisien.
Ce cadre plus général est indispensable pour une application aux représentations non génériques.

\subsubsection{Type des modules de Breuil--Kisin}

Afin de pouvoir démontrer notre résultat de classification dans la 
généralité énoncée ci-dessus, nous avons besoin, en guise de préalable,
d'étendre les notions de type de Hodge et de type galoisien à n'importe
quel module de Breuil--Kisin à coefficients. C'est l'objet de ce paragraphe.

\begin{definit}
\label{def:typeBK}
Soient $R$ une $\oE$-algèbre locale, complète, noetherienne, de corps 
résiduel $k_E$ et $\MK$ un module de Breuil--Kisin \emph{libre} de rang $2$ 
sur $R \hat \otimes_{\Zp} \SK$ avec donnée de descente de $L$ à $F$. Le 
module $\MK$ est dit :
\begin{itemize}
\item[$\bullet$] \emph{de type de Hodge $\vv_0$} si pour tout $i$ dans $\Z / f \Z $, l'idéal déterminant du frobenius $\varphi$ de $\MK^{(i)}$ dans $\MK^{(i+1)}$ est l'idéal principal engendré par $u^e + p$ ;
\item[$\bullet$] \emph{de type galoisien $\ttt$} si pour tout $i$ dans $\Z / f \Z $, il existe une base $\left(e^{(i)}_\eta, e^{(i)}_{\eta'}\right) $ de $\MK^{(i)}$ comme $R[[ u ]]$-module telle que $\Gal(L/F)$ agit sur $e^{(i)}_\eta$ par $\eta$ et sur $e^{(i)}_{\eta'}$ par $\eta'$ ;
\item[$\bullet$] \emph{de déterminant $\psi \epsilon$} si le déterminant 
de la représentation de $G_{\infty}$ qui lui est associée est
$(\psi \epsilon)_{|G_{\infty}}$.
\end{itemize}
Il est dit de type $(\vv_0 , \ttt, \psi )$ s'il satisfait aux trois 
propriétés ci-dessus.
\end{definit}

\begin{rem}
Dans la définition ci-dessus, nous avons supposé que $\MK$ est libre en 
tant que module sur $R \hat \otimes_{\Zp} \SK$. La raison en est, d'une 
part, que la définition est plus simple à écrire sous cette hypothèse 
additionnelle et, d'autre part, que dans le théorème de classification 
que nous allons énoncer (voir proposition \ref{propecriturejolie}), 
l'hypothèse de liberté est absolument nécessaire. Il nous est donc paru 
opportun de restreindre la définition \ref{def:typeBK} à cette situation.
\end{rem}

\begin{rem}
Soit $\MK$ un module de Breuil--Kisin comme dans la définition 
\ref{def:typeBK}. Supposons que $\MK$ soit de type de Hodge $\vv_0$ et 
de type galoisien $\ttt$. Il existe alors un élément $\delta$ de 
$R^\times$ ainsi qu'une base $w$ de $(\det \MK)^{(0)} $ telle que
$$\varphi^f (w) = \delta \cdot (u^e + p) \cdot \varphi(u^e +p) \cdots 
\varphi^{f-1}(u^e +p) \cdot w.$$
De plus, le prolongement canonique de la représentation galoisienne 
associée à $\MK$ donné par le théorème \ref{thm:hauteur1} (et après la renormalisation du début de la partie \ref{sec:methode}) a pour
déterminant $\eta \eta' \cdot \nr(\delta^{-1}) \cdot \epsilon$.
\end{rem}

La proposition suivante, qui est une conséquence simple des résultats
de Kisin, fait le lien entre être de type $(\vv_0, \ttt, \psi)$ pour
une représentation et pour un module de Breuil--Kisin.

\begin{prop}\label{propstroumphe}
Supposons que $R$ soit l'anneau des entiers d'une extension finie de 
$E$. Soit $V$ une représentation de $G_F$ de dimension $2$, libre sur
$R$.
Alors, $V$ est de type $(\vv_0, \ttt, \psi)$ si et seulement s'il
existe un module de Breuil--Kisin $\MK$ de type $(\vv_0, \ttt, \psi)$
dont la représentation de $G_F$ associée \emph{via} le foncteur
$\bbM$ est $V$.
\end{prop}

\begin{rem}
\emph{A priori}, la représentation associée à $\MK$ \emph{via} le
foncteur $\bbM$ n'est qu'une représentation de $\Goo$. Toutefois,
d'après le théorème \ref{thm:hauteur1}, celle-ci se prolonge de façon
canonique au groupe $G_F$ tout entier.
L'énoncé de la proposition \ref{propstroumphe} a donc bien un sens.
\end{rem}

\begin{rem}
Si, comme dans la proposition \ref{propstroumphe}, $R$ est l'anneau
des entiers d'une extension finie de $E$, tout module de Breuil--Kisin
sur $R \hat\otimes_{\Zp} \SK$ est automatiquement libre. Il n'y a donc ici 
aucune subtilité à l'hypothèse de liberté.
\end{rem}

\begin{proof}
Supposons d'abord l'existence de $\MK$. Le fait que la représentation 
$V$ soit de déterminant $\psi\varepsilon$ est alors immédiat. Elle est de type 
galoisien $\ttt$ car la représentation de Weil--Deligne associée à $V$ 
est donnée par l'action de la donnée de descente sur $(\MK/u\MK)[1/p]$.
Enfin, en ce qui concerne la condition sur le type de Hodge, remarquons 
que, d'après le lemme 1.2.2 de \cite{Ki-Crys}, les poids de Hodge--Tate 
de $V$ sont tous positifs ou nuls et, pour tout $i \in \Z/f\Z$, l'idéal 
déterminant de $\varphi : \MK^{(i)} \to \MK^{(i+1)}$ est engendré par 
$(u^e + p)^{h_i}$ où $h_i$ est la somme des poids de Hodge--Tate pour le 
plongement $\tau_i$. Notre hypothèse s'écrit donc $h_i = 1$. Nous en
déduisons que, pour chaque plongement $\tau$, les poids de Hodge--Tate de 
$V$ sont $0$ et $1$, ce qui signifie bien que $V$ est de type de Hodge 
$\vv_0$.

La réciproque se traite de manière similaire.
\end{proof}

Par le théorème \ref{thm:Eusst}, le module $\MK$ de la proposition 
\ref{propstroumphe}, s'il existe, est unique à isomorphisme près. Nous 
obtenons ainsi une bijection canonique entre les représentations $V$ de 
type $(\vv_0, \ttt, \psi)$ et les modules de Breuil--Kisin de type 
$(\vv_0, \ttt, \psi)$ sur $R \hat \otimes_{\Zp} \SK$.

\subsubsection{Genre d'un module de Breuil--Kisin}

Rappelons que $\SSS$ désigne l'ensemble des plongements de $F$ dans 
$\Qpbar$ et que nous avons fixé un élément privilégié $\tau_0$ dans $\SSS$. Les autres 
éléments de $\SSS$ s'obtiennent alors en composant $\tau_0$ par les
puissances successives du Frobenius agissant sur $F$.
Nous pouvons écrire :
$$
\eta \cdot (\eta')^{-1} = (\tau_0 \circ \varepsilon_f)^{\sum\limits^{f-1}_{i = 0} c_ip^i} = \prod_{i \in \Z/f\Z} (\tau_0 \circ \varphi^i \circ \varepsilon_f)^{c_i}
$$
et de même
$$
\eta' = (\tau_0 \circ \varepsilon_f)^{\sum\limits^{f-1}_{i = 0} b_ip^i} = \prod_{i \in \Z/f\Z} (\tau_0 \circ \varphi^i \circ \varepsilon_f)^{b_i}
$$
où les $c_i$ et $b_i$ sont des entiers compris entre $0$ et $p-1$.
Il est commode, et nous le ferons, de considérer que les indices des $b_i$ et des $c_i$ vivent dans $\Z/f\Z$.
Posons également
$$\gamma_i=\sum_{j=0}^{i-1}c_{f-(i-j)}p^j+\sum_{j=i}^{f-1}c_{j-i}p^j
$$
et
$\beta_i$
l'équivalent pour les $b_i$.
Pour tout $i$, ces entiers satisfont aux congruences
$p^i \gamma_0 \equiv \gamma_i \pmod{p^f - 1}$
et $p^i \beta_0 \equiv \beta_i \pmod{p^f - 1}$
de sorte que nous avons
$\eta' = \tau_0 \circ \varphi^{-i} \circ \varepsilon_f^{\beta_i}$
et 
$\eta = \tau_0 \circ \varphi^{-i} \circ \varepsilon_f^{\gamma_i + \beta_i}$.
Notons que l'hypothèse $\eta \neq \eta'$ assure que tous les entiers $\gamma_i$ sont compris entre $1$ et $e-1$,
 ou de manière équivalente que le $f$-uplet $(c_0,\ldots,c_{f-1})$ n'est ni $(0,\ldots, 0 )$ ni $(p-1, \ldots , p-1)$.

Considérons à présent une $\oE$-algèbre $R$ que nous supposons comme 
habituellement locale, complète, noetherienne et de corps résiduel 
$k_E$. Nous notons $\m_R$ l'idéal maximal de $R$. 
Soit $\MK$ un module de Breuil--Kisin sur $R \hat\otimes_{\Zp} \SK$ de 
type $(\vv_0, \ttt, \psi)$. D'après la définition, il existe pour tout 
$i$ dans  $\Z/f\Z$, une base $\left(e_\eta^{(i)},e_{\eta'}^{(i)}\right)$ de $\MK^{(i)}$ 
dans laquelle l'action de la donnée de descente est diagonale et donnée 
par les caractères $\eta$ et $\eta'$. Il en résulte que la matrice du 
Frobenius dans ces mêmes bases est de la forme
\begin{equation}
\label{formematriceG}
G^{(i)} = 
\begin{pmatrix} s_1^{(i)}&u^{e-\gamma_{i+1}}s_2^{(i)}\\
u^{\gamma_{i+1}}s_3^{(i)}& s_4^{(i)}\\ \end{pmatrix}
\end{equation}
avec $s_j^{(i)}\in R[[u^e]]$ pour $1\leq j\leq 4$ et $\det 
G^{(i)}=(u^e+p)\alpha$ avec $\alpha\in R[[u]]$ inversible.
Un changement de base de $\MK^{(i)}$ compatible à l'action de 
$\Gal(L/F)$ a pour matrice
\begin{equation}
\label{formechgtbase}
P^{(i)}= \begin{pmatrix} \sigma_1^{(i)}&u^{e-\gamma_i}\sigma_2^{(i)}\\
u^{\gamma_i}\sigma_3^{(i)}& \sigma_4^{(i)}\\ \end{pmatrix}
\end{equation}
avec $\sigma_j^{(i)}\in R[[u^e]]$ pour $1\leq j\leq 4$. La matrice de 
$\varphi:\MK^{(i)}\rightarrow\MK^{(i+1)}$ après changement de base
s'écrit
\begin{equation}
\label{changementdebase}
H^{(i)} = \left(P^{(i+1)}\right)^{-1}G^{(i)}\varphi(P^{(i)})
\end{equation}
Réciproquement, une famille $(G^{(i)})_{0\leq i\leq f-1}$ de matrices de 
cette forme définit un module de Breuil--Kisin et deux telles familles 
définissent le même module de Breuil--Kisin de type $\eta\oplus \eta'$ à 
isomorphisme près si et seulement si elles s'obtiennent par des 
changements de base de la forme (\ref{changementdebase}) pour $0\leq 
i\leq f-1$. Le lemme suivant nous permet de définir la notion de genre d'un module de Breuil--Kisin :
\begin{lem}
Soit
$$
G=\begin{pmatrix} s_1&u^{e-\gamma}s_2\\
u^{\gamma}s_3& s_4\\ \end{pmatrix}$$
avec $s_j\in R[[u^e]]$ pour $1\leq j\leq 4$, $\gamma$ un entier entre $0$ et $e$ et $\det G=\alpha(u^e+p)$ pour $\alpha$ dans $R[[u^e]]$ inversible. Notons $\bar{s}_i \equiv s_i\pmod {u^e}$.
\begin{enumerate}[(i)]
\item Si $\bar{s}_4$ est inversible dans $R$, notons $\GG(G)=\I_\eta$,
\item Si $\bar{s}_1$ est inversible dans $R$, notons $\GG(G)=\I_{\eta'}$,
\item Sinon, notons
$\GG(G)=\II$.
\end{enumerate}
Soit $\MK$ un module de Breuil--Kisin sur $R \hat\otimes_{\Zp} \SK$ de 
type $(\vv_0, \ttt, \psi)$ et $(G^{(i)})_{0\leq i\leq f-1}$ une famille de matrices de Frobenius pour un choix de bases de $\MK^{(i)}$. Alors la suite des genres $(\GG(G^{(i)}))_{0\leq i\leq f-1}$ est uniquement déterminée par le module de Breuil--Kisin $\MK$ et est dite genre de $\MK$.

\end{lem} 
\begin{proof} 
Il suffit de constater que la définition de $\GG(G)$ est compatible au changement de base (\ref{changementdebase}).
\end{proof}
Dans la suite, nous allons énoncer un théorème de classification (Proposition \ref{propecriturejolie} des modules de Breuil--Kisin excluant certains genres très particuliers dont l'étude fera l'objet d'un article ultérieur. 

\begin{definit}\label{def:mauvaisgenre}
Un module de Breuil--Kisin $\MK$ sur $R \hat\otimes_{\Zp} \SK$ de 
type $(\vv_0, \ttt, \psi)$ est dit de mauvais genre s'il n'a aucun facteur de genre $\II$ et si pour tout $0\leq i\leq f-1$
$$c_{f-1-i}=\left\{\begin{array}{lll} 1&\mbox{ si }&\GG(G^{(i-1)})=\GG(G^{(i)})=\I_\eta,\cr
0&\mbox{ si }& \GG(G^{(i-1)})=\I_\eta\mbox{ et }\; \GG(G^{(i)})=I_{\eta'},\cr
p-2&\mbox{ si }&\GG(G^{(i-1)})=\GG(G^{(i)})=\I_{\eta'},\cr
p-1&\mbox{ si }& \GG(G^{(i-1)})=\I_{\eta'}\mbox{ et }\; \GG(G^{(i)})=I_{\eta}.\cr
\end{array}\right.$$
\end{definit}

\subsubsection{Énoncé du théorème de classification}
L'objectif du 
paragraphe \ref{ssec:classification} est de démontrer les propositions 
\ref{propecriturejolie} et \ref{prop:uniciteecriture} ci-dessous.

\begin{prop} \label{propecriturejolie}
Nous supposons soit que $R$ est plate sur $\oE$, soit que $R = k_E$.

Soit $\MK=\MK^{(0)}\oplus\cdots\oplus\MK^{(f-1)}$ un module de 
Breuil--Kisin \emph{libre} de rang $2$ sur $R \hat\otimes_{\Zp} \SK$ de 
type galoisien $\eta\oplus \eta'$ et de type de Hodge $\vv_0$.

Alors il existe des éléments $\alpha, \alpha'$ dans $R^\times$ ainsi que, 
pour tout $i\in \{0,\ldots,f-1\}$, une base $e_\eta^{(i)}, 
e_{\eta'}^{(i)}$ de $\MK^{(i)}$ et des paramètres $a_i, a'_i$ (vivant 
dans un espace précisé ci-après) tels que :
\begin{enumerate}[(1)]
\item pour tout $i$, la donnée de descente agit sur $e_\eta^{(i)}$
(resp. $e_{\eta'}^{(i)}$) par le caractère $\eta$ 
(resp. $ \eta'$) ;
\item la matrice $G^{(i)}$ pour $0\leq i\leq f-2$ (resp. 
$\begin{pmatrix}\alpha^{-1}&0\\ 0&\alpha'^{-1}\\ \end{pmatrix} 
G^{(f-1)}$ pour $i = f-1$) de $\varphi:\MK^{(i)}\rightarrow\MK^{(i+1)}$ 
est de l'une des formes suivantes :
\begin{itemize}
\item {\bf Genre $\I_\eta$}:
$\begin{pmatrix} u^{e}+p& 0\\a_iu^{\gamma_{i+1}}& 1\\ \end{pmatrix}$ pour $a_i$ dans $R + Ru^e$ et, par convention, $a'_i= 0$ ;
\item {\bf Genre $\I_{\eta'}$}:
$\begin{pmatrix} 1&a'_iu^{e-\gamma_{i+1}}\\0& u^e+p\\ \end{pmatrix}$
pour $a'_i$ dans $R + R u^e$ et, par convention, $a_i = 0$ ;
\item {\bf Genre $\II$}:
$\begin{pmatrix} a_i&u^{e-\gamma_{i+1}}\\u^{\gamma_{i+1}}& a'_i\\ \end{pmatrix}$ avec $a_i, a'_i$ dans $\m_R$ et $a_i a'_i=-p$.
\end{itemize}
\end{enumerate}
De plus, si le module $\MK$ n'a pas mauvais genre, on peut choisir $a_i$ et $a'_i$ dans $R$ pour tout $i$.
\end{prop}

\begin{rem}
\label{rem:constrmodBK}
Réciproquement, la donnée d'une suite de genres $(g_0, \ldots, g_{f-1})$
et de paramètres $(\alpha, \alpha', a_0, a'_0, \ldots, a_{f-1}, a'_{f-1})$
satisfaisant aux conditions de la proposition \ref{propecriturejolie} définit un module
de Breuil--Kisin de type de Hodge $\vv_0$, de type galoisien $\eta 
\oplus \eta'$ et de déterminant $\eta \eta' \cdot \nr\left((-1)^{| \II |} (\alpha\alpha')^{-1} \right)\cdot \varepsilon $, 
où $| \II | $ désigne le nombre de $g_i$ égaux à $\II$.
\end{rem}

\begin{rem} Le résultat de la proposition \ref{propecriturejolie} pour $R=k_E$ se déduit du cas $R=\oE$. En effet, il suffit de constater que tout module de Breuil--Kisin de type galoisien $\eta\oplus\eta'$ et de type de Hodge $\vv_0$ sur $k_E$ se relève en un module de Breuil--Kisin sur $\oE$ de mêmes types puis de réduire modulo $\pE$ la matrice à coefficients dans $\oE$ obtenue par la proposition \ref{propecriturejolie}.
\end{rem}

\begin{prop}
\label{prop:uniciteecriture}
Nous conservons les hypothèses de la proposition \ref{propecriturejolie} et supposons de plus que $\MK$ n'a pas mauvais genre.
Nous nous donnons deux écritures de $\MK$ correspondant à une
suite de genres $(g_0, \ldots, g_{f-1})$ et 
à des paramètres $(\alpha, \alpha', a_0, a'_0, \ldots, a_{f-1}, 
a'_{f-1})$ et $(\beta, \beta', b_0, b'_0, \ldots, b_{f-1}, b'_{f-1})$
respectivement.

En notant $n_i$ 
le nombre de genres $\II$ parmi $g_0, \ldots, g_{i-1}$, il existe un élément 
$\lambda$ de $R^\times$ vérifiant :
$$\text{pour tout } i \in \{0,  \ldots, f-1\}, \quad
b_i = \lambda^{(-1)^{n_i}} 
a_i \text{ et } b'_i = \lambda^{-(-1)^{n_i}} a'_i$$
et, de plus :
\begin{itemize}
\item si $n_f$ est pair, $\alpha = \beta$, $\alpha' = \beta'$,
tandis que
\item si $n_f$ est impair, $\alpha \alpha' = \beta \beta'$ et
$\lambda = \frac{\alpha}{\beta} = \frac{\beta'}{\alpha'}.$
\end{itemize}
\end{prop}
Il est alors immédiat de déduire le résultat suivant :
\begin{cor}
\label{cor:uniciteecriture}
Dans le cadre de la proposition
\ref{prop:uniciteecriture}, nous avons les résultats suivants :
\begin{enumerate}
\item si le genre de $\MK$ possède un nombre impair de facteurs $\II$, 
alors $\MK$ admet une unique écriture avec le paramètre $\alpha$ égal à 
$1$ ;
\item si le genre de $\MK$ possède un nombre pair de facteurs $\II$ et
s'il existe un paramètre $a_i$ ou $a'_i$ qui est inversible dans $R$
(\emph{i.e.} non nul dans le corps résiduel $k_E$), alors $\MK$ admet
une unique écriture où ce paramètre est égal à $1$.
\end{enumerate}
\end{cor}

\subsubsection{Polygones de Newton des éléments de $R[[u]]$}

Avant d'entamer la démonstra\-tion de la proposition 
\ref{propecriturejolie}, nous regroupons dans ce numéro quelques 
résultats (classiques) sur les polygones de Newton que nous utilisons
constamment dans la suite.

L'hypothèse de platitude que nous avons faite sur $R$ implique que 
l'uniformisante $\pE$ n'est pas diviseur de $0$ dans $R$. Nous pouvons 
ainsi définir une fonction $v_R : R \to \N \cup \{\infty\}$ comme suit : 
pour $x$ dans $R$, $v_R(x)$ désigne le plus grand entier $n$ tel que 
$\pE^n$ divise $x$ avec la convention $v_R(0) = +\infty$. Pour $x$ et 
$y$ dans $R$, nous avons :
\begin{itemize}
\item $v_R(xy) \geq v_R(x) + v_R(y)$ avec égalité dès que $x$ ou
$y$ est dans $\oE$, et
\item $v_R(x+y) \geq \min(v_R(x), v_R(y))$ avec égalité dès que
$v_R(x) \neq v_R(y)$.
\end{itemize}

\begin{definit} 
Soit $a = \sum_{i=0}^\infty a_i u^i$ un élément de $R[[u]]$.
Le \emph{polygone de Newton} $\NP(a)$ de $a$ est l'enveloppe convexe
dans le plan des points de coordonnées $(i, v_R(a_i))$ pour $i$
variant dans $\N$ et d'un point supplémentaire situé à l'infini 
dans la direction des ordonnées positives.
\end{definit}

\begin{lem}\label{polygoneNewton}
Soit $a, b\in R[[u]]$. Le polygone de Newton $\NP(ab)$ est inclus 
dans la somme de Minkowski $\NP(a) + \NP(b)$.

Si en outre $a$ ou $b$ est dans $\oE[[u]]$, alors l'inclusion précédente 
est une égalité.
\end{lem}

\begin{proof}
Écrivons $a = \sum_{i=0}^\infty a_i u^i$, $b = \sum_{i=0}^\infty b_i 
u^i$ et $ab = \sum_{i=0}^\infty c_i u^i$ où les $a_i$, les $b_i$ et les 
$c_i$ sont des éléments de $R$. Pour tout entier $s$, nous avons
$c_s = \sum_{i+j=s} a_i b_j$ et donc 
\begin{equation}
\label{eq:inegNewton}
v_R(c_s) \geq \min_{i+j = s} v_R(a_i) + v_R(b_j).
\end{equation}
L'inclusion $\NP(ab) \subset \NP(a) + \NP(b)$ en résulte directement.

Supposons maintenant que tous les $a_i$ appartiennent à $\oE$. Pour 
démontrer l'égalité souhaitée, il suffit de vérifier que tout point 
extrémal de $\NP(a) + \NP(b)$ est dans $\NP(ab)$. Supposons par 
l'absurde que ce ne soit pas le cas et notons $(s,v)$ les coordonnées 
d'un point $M$ qui est un contre-exemple. Par définition des polygones 
de Newton, notre supposition implique en particulier que l'inégalité
\eqref{eq:inegNewton} est stricte. Ainsi, il existe deux couples 
distincts $(i_1,j_1)$ et $(i_2,j_2)$ avec $i_1 + j_1 = i_2 + j_2 = 
s$ et 
$$\begin{array}{rcccl}
v & = & v_R(a_{i_1} b_{j_1}) \,\, = \,\, v_R(a_{i_1}) + v_R(b_{j_1}) \\
& = & v_R(a_{i_2} b_{j_2}) \,\, = \,\, v_R(a_{i_2}) + v_R(b_{j_2})
& < & v_R(c_s).
\end{array}$$
Les points $\big(i_1 + j_2, v_R(a_{i_1}) + v_R(b_{j_2})\big), \big(i_2 + 
j_1, v_R(a_{i_2}) + v_R(b_{j_1})\big)$ sont distincts et appartiennent à 
$\NP(a) + \NP(b)$ par construction. De plus, par ce qui précède, le 
milieu du segment qui les joint est $(s,v)$. Ainsi $(s,v)$ n'est pas un 
point extrémal de $\NP(a) + \NP(b)$, ce qui constitue une contradiction.
\end{proof}

\subsubsection{Démonstration de la proposition \ref{propecriturejolie}}

Nous ne traitons que le cas d'un module n'ayant pas mauvais genre, le cas général étant similaire.
Nous commençons par plusieurs lemmes préparatoires.

\begin{lem}\label{propdefgenrepassage}
Soit
$$
G=\begin{pmatrix} s_1&u^{e-\gamma}s_2\\
u^{\gamma}s_3& s_4\\ \end{pmatrix}$$
avec $s_j\in R[[u^e]]$ pour $1\leq j\leq 4$, $\gamma$ un entier entre $0$ et $e$ et $\det G=\alpha(u^e+p)$ pour $\alpha$ dans $R[[u^e]]$ inversible. Notons $\bar{s}_i \equiv s_i\pmod {u^e}$.
\begin{enumerate}[(i)]
\item Si $\bar{s}_4$ est inversible dans $R$ i.e. $\GG(G)=\I_\eta$, posons
$$\BB(G)=
\begin{pmatrix} 
\frac{s_1-u^eas_2}{u^e+p} & u^{e-\gamma}s_2\\
u^{\gamma} \: \frac{s_3-as_4}{u^e+p} & s_4
\end{pmatrix}$$
où $a$ dans  $R$ est défini par $a\equiv s_3/s_4\pmod {u^e+p}$. Alors
$$G=\BB(G)\begin{pmatrix}u^e+p&0\\ au^\gamma&1\\ \end{pmatrix}.$$
\item Si $\bar{s}_1$ est inversible dans $R$ i.e. $\GG(G)=\I_{\eta'}$, posons
$$\BB(G)=
\begin{pmatrix}s_1 & u^{e-\gamma} \: \frac{s_2-a's_1}{u^e+p}\\
u^{\gamma}s_3 & \frac{s_4-u^ea's_3}{u^e+p}
\end{pmatrix}$$
où $a'$ dans $R$ est défini par $a'\equiv s_2/s_1=u^{-e}s_4/s_3\pmod {u^e+p}$. Alors
$$G=\BB(G)\begin{pmatrix}1&a'u^{e-\gamma}\\ 0&u^e+p\\ \end{pmatrix}.$$
\item Sinon
$\GG(G)=\II$ et posons
$$\BB(G)=
\begin{pmatrix}
\frac{u^es_2-a's_1}{u^e+p} & u^{e-\gamma} \frac{s_1-as_2}{u^e+p} \\
u^{\gamma} \frac{s_4-a's_3}{u^e+p} & \frac{u^es_3-as_4}{u^e+p}
\end{pmatrix}$$
où $a,a'$ dans $R$ sont définis par $a \equiv s_1/s_2\pmod {u^e+p}$ et 
$a'\equiv s_4/s_3\pmod {u^e+p}$. Alors \\
$$G=\BB(G)\begin{pmatrix}a&u^{e-\gamma}\\ u^\gamma&a'\\ \end{pmatrix},$$
et $aa'=-p$.
\end{enumerate}
De plus, dans tous les cas, $\BB(G)\in \GL_2(R[[u]])$.
\end{lem}
\begin{proof}
(i) Commençons par montrer que la matrice $\BB(G)$ est bien définie. 
D'une part, comme $\bar{s}_4$ est inversible dans $R$, $s_4$ est 
inversible modulo $u^e+p$, ce qui montre que l'élément $a$ est bien 
défini. D'autre part, comme $\det G\equiv 0\pmod {u^e+p}$, nous avons 
encore $a\equiv u^{-e}s_1/s_2$. Nous en déduisons que $\BB(G)$ est à 
coefficients dans $R[[u]]$ et même, plus précisément, que la matrice 
$\BB(G)$ est de la forme
$$ \begin{pmatrix} \sigma_1 & u^{e-\gamma}\sigma_2\\
u^{\gamma}\sigma_3 & \sigma_4 \end{pmatrix} 
\mbox{ avec }\sigma_i\in R[[u^e]].$$
Le cas (ii) est analogue.

(iii) Supposons que $\bar{s}_1$ et $\bar{s}_4$ ne sont pas inversibles 
dans $R$. Rappelons que $\det G=\alpha(u^e+p)$ avec $\alpha$ dans $R[[u]]$ 
inversible. Ainsi $u^es_2s_3\equiv - u^e\alpha \pmod {\pE R[[u]]}$ et comme 
$u^e$ n'est pas diviseur de zéro dans $R[[u^e]]/\pE R[[u^e]]$, nous
obtenons $s_2s_3\equiv - \alpha \pmod {\pE R}$. Ainsi $s_2$ et $s_3$ 
sont inversibles dans $R[[u^e]]$, ce qui suffit à entraîner
que la matrice $\BB(G)$ est bien définie et à coefficients dans $R[[u]]$.

Enfin, de l'égalité $\det(G)=\alpha(u^e+p)$ avec $\alpha$ inversible 
dans $R$, nous déduisons que le déterminant de $\BB(G)$ est inversible 
dans $R[[u]]$ dans chacun des cas.
\end{proof}

Pour $t\in \N$, définissons les idéaux de $R[[u^e]]$
$$\textstyle
I_t=\Big\{ \sum_{i=0}^\infty a_iu^{ei}, a_i\in R, v_R(a_i)\geq t-
\frac{pi}{p-1} \Big\},$$
$$\textstyle
I_t^\varphi=\Big\{ \sum_{i=0}^\infty a_iu^{ei}, a_i\in R, v_R(a_i)\geq 
t- \frac i{p-1} \Big\}.$$
En particulier, nous avons $\varphi(I_t)\subset I_t^\varphi$.

\begin{lem}\label{lemtechnique1}
Supposons $t\in \N$. Soit $x\in R[[u^p]]$ tel que $(u^e+p)x\in 
I_t^\varphi$. Alors $x\in I_t$.
\end{lem}

\begin{proof} 
Si $(u^e+p)x\in I_t^\varphi$, son polygone de Newton est inclus dans la 
région du plan $D=\big\{ (a,b) \in\R^2, \, b\geq t- \frac a{p-1}\big\}$.
D'après le lemme \ref{polygoneNewton}, le polygone de Newton de $x$ est 
inclus dans le translaté $(-1,0)+D$, donc en notant $x=\sum_{i=0}^\infty 
a_i u^{ei}$, nous avons $v_R(a_i) \geq t- \frac{i+1}{p-1}$ pour tout 
$i\geq 0$. Nous en tirons $v_R(a_i)> t - \frac{pi}{p-1}$ pour tout $i \geq 0$ 
et, par suite, $x\in I_t$. 
\end{proof}

\begin{lem}\label{lemtechnique2}
Supposons $t\in \N^*$. Soit $x\in R[[u^e]]$ tel que 
$$(u^e+p)x\in u^{e}I_t^\varphi+\varphi(I_t\cap u^eR[[u^e]])+ p\pE^{t+1}R[[u^e]].$$
Alors $x\in I_{t+1}$.
\end{lem}
\begin{proof} Remarquons que, d'après le lemme \ref{polygoneNewton} :
\begin{itemize}
\item le polygone de Newton d'un élément de $u^{e} I_t^\varphi$ est inclus dans 
$$\textstyle
\big\{(a,b) \in\R^2 \,\, |\,\, a\geq 1,\, b\geq 0,\, b\geq t-\frac{a-1}{p-1}\big\},$$
\item le polygone de Newton d'un élément de $\varphi( I_t\cap u^eR[[u^e]])$ est inclus dans 
$$\textstyle
\big\{(a,b) \in\R^2 \,\ | \,\, a\geq p, \, b\geq 0, \, b\geq t-\frac a{p-1}\big\},$$
\item le polygone de Newton d'un élément de $p\pE^{t+1}R[[u^e]]$ est inclus dans
$$\{0\}\times[t+1+v_R(p),\infty[.$$
\end{itemize}
Ainsi le polygone de Newton de $(u^e+p)x$ est inclus dans l'enveloppe 
convexe des trois régions précédentes, représentées en gris sur la 
figure ci-après.

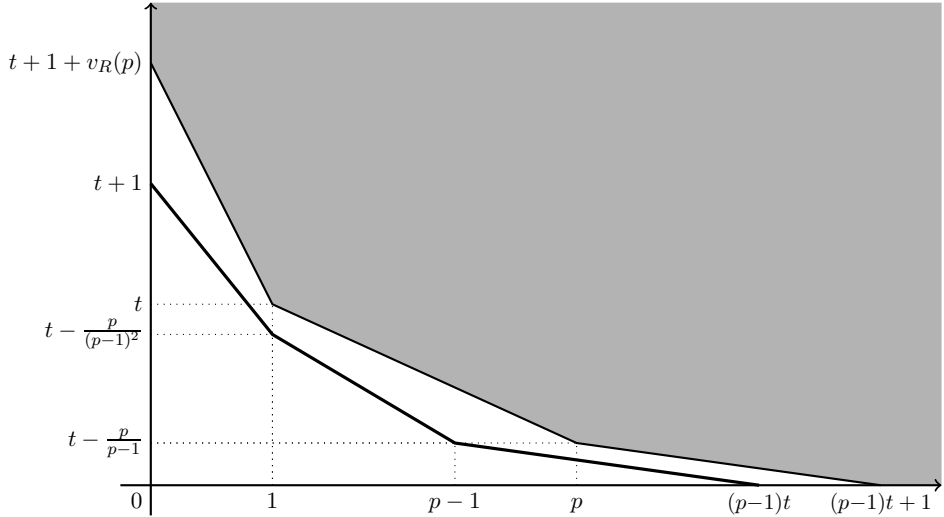
\begin{center}
\begin{tikzpicture}[scale=0.8]
\draw[->, thick] (-0.5,0)--(13,0);   
\draw[->, thick] (0,-0.5)--(0,8);   
\fill[opacity=0.3] (0,8)--(0,7)--(2,3)--(7,0.7)--(12,0)--(13,0)--(13,8)--cycle;
\draw[thick] (0,7)--(2,3)--(7,0.7)--(12,0);
\draw[very thick](0,5)--(2,2.5)--(5,0.7)--(10,0);
\draw[dotted] (0,3)--(2,3);
\draw[dotted] (2,0)--(2,3);
\draw[dotted] (0,2.5)--(2,2.5);
\draw[dotted] (2,0)--(2,2.5);
\draw[dotted] (0,0.7)--(7,0.7);
\draw[dotted] (7,0)--(7,0.7);
\draw[dotted] (0,0.7)--(5,0.7);
\draw[dotted] (5,0)--(5,0.7);
\node[below left, scale=0.8] at (0,0) { $0$ };
\node[below, scale=0.8] at (2,0) { $1$ };
\node[below, scale=0.8] at (5,0) { $p-1$ };
\node[below, scale=0.8] at (7,0) { $p\vphantom{-1}$ };
\node[below, scale=0.75] at (10,0) { $(p{-}1)t\vphantom{+1}$ };
\node[below, scale=0.75] at (12,0) { $(p{-}1)t+1$ };
\node[left, scale=0.8] at (0,7) { $t+1+v_R(p)$ };
\node[left, scale=0.8] at (0,5) { $t+1$ };
\node[left, scale=0.8] at (0,3) { $t$ };
\node[left, scale=0.8] at (0,2.5) { $t- \frac p{(p-1)^2}$ };
\node[left, scale=0.8] at (0,0.7) { $t- \frac p{p-1}$ };
\end{tikzpicture}
\end{center}

\noindent
Il suit alors du lemme \ref{polygoneNewton} que le 
polygone de Newton de $x$ est inclus dans la région délimitée par le 
trait en gras. Le lemme s'en déduit.
\end{proof}

\begin{lem} \label{lemjolieecriture}
Soit $t\in\N$ et
$$P'=\begin{pmatrix}\sigma'_1&u^{e-\gamma'}\sigma'_2\\
u^{\gamma'}\sigma'_3&\sigma'_4\\ 
\end{pmatrix}\in \GL_2(R[[u]])$$
avec 
$0\leq \gamma'\leq e$, $\sigma'_j\in R[[u^e]]$ pour $1\leq j\leq 4$ et
$$\left\{\begin{array}{l}\sigma_1'\equiv\sigma_4'\equiv 1\pmod {u^e, I_t},\cr
\sigma_2'\equiv\sigma_3'\equiv 0\pmod {I_t}.\end{array}\right.$$
Soit $c\in\{0,\ldots,p-1\}$ tel que $0\leq \gamma=p\gamma'-ec\leq e$. Alors
\begin{enumerate}[(i)]
\item
Si $G=\begin{pmatrix}u^e+p&0\\ a u^\gamma &1\\ \end{pmatrix}$ avec $a\in R$, alors
 $\GG(G)=\I_\eta=\GG(G\varphi(P'))$ et
 la matrice $\BB(G\varphi(P'))$ est de la forme
\begin{equation}\label{ecriturepassage}
\BB(G\varphi(P'))=\begin{pmatrix}\sigma_1&u^{e-\gamma}\sigma_2\\
u^{\gamma}\sigma_3&\sigma_4\\ \end{pmatrix}
\mbox{ avec }\left\{\begin{array}{l} 
\sigma_1\equiv\sigma_4\equiv 1\pmod {u^e, I_{t+1}}, \cr
\sigma_2\equiv\sigma_3\equiv 0\pmod {\pE,u^e}, \cr
\sigma_2\equiv0 \pmod {I_{t+1}},\cr 
\sigma_3\equiv 0\pmod {I_t}.\cr\end{array}\right.
\end{equation}
De plus, 
$$\sigma_3\equiv \left\{\begin{array}{ll} 0\pmod {I_{t+1}}&\mbox{ si }c\not\in\{1,p-1\},\cr
\varphi(\sigma'_3)\pmod {I_{t+1}}&\mbox{ si }c=1,\cr
-a^2\varphi(\sigma'_2)\pmod {I_{t+1}}&\mbox{ si }c=p-1.\cr
\end{array}\right.$$
\item
Si $G=\begin{pmatrix}1&a'u^{e-\gamma}\\0& u^e+p\\ \end{pmatrix}$, avec $a'\in R$, alors
 $\GG(G)=\I_{\eta'}=\GG(G\varphi(P'))$ et
 la matrice $\BB(G\varphi(P'))$ est de la forme
\begin{equation}
\BB(G\varphi(P'))=\begin{pmatrix}\sigma_1&u^{e-\gamma}\sigma_2\\
u^{\gamma}\sigma_3&\sigma_4\\ \end{pmatrix}
\mbox{ avec }\left\{\begin{array}{l} 
\sigma_1\equiv\sigma_4\equiv 1\pmod {u^e, I_{t+1}}, \cr
\sigma_2\equiv\sigma_3\equiv 0\pmod {\pE,u^e}, \cr
\sigma_3\equiv0 \pmod {I_{t+1}},\cr 
\sigma_2\equiv 0\pmod {I_t}.\cr\end{array}\right.
\end{equation}
De plus, 
$$\sigma_2\equiv \left\{\begin{array}{ll} 0\pmod {I_{t+1}}&\mbox{ si }c\not\in\{0,p-2\},\cr
\varphi(\sigma'_2)\pmod {I_{t+1}}&\mbox{ si }c=p-2,\cr
-a'^2\varphi(\sigma'_3)\pmod {I_{t+1}}&\mbox{ si }c=0.\cr
\end{array}\right.$$
\item
Si $G=\begin{pmatrix}a&u^{e-\gamma}\\ u^\gamma &a'\\ \end{pmatrix}$, avec $a,a'\in R$ et $aa'=-p$ alors
 $\GG(G)=\II=\GG(G\varphi(P'))$ et
 la matrice $\BB(G\varphi(P'))$ est de la forme
\begin{equation}
\BB(G\varphi(P'))=\begin{pmatrix}\sigma_1&u^{e-\gamma}\sigma_2\\
u^{\gamma}\sigma_3&\sigma_4\\ \end{pmatrix}
\mbox{ avec }\left\{\begin{array}{l} 
\sigma_1\equiv\sigma_4\equiv 1\pmod {u^e, I_{t+1}}, \cr 
\sigma_2\equiv\sigma_3\equiv 0 \pmod {\pE,u^e},\cr
\sigma_2\equiv\sigma_3\equiv0 \pmod {I_{t+1}}.\cr\end{array}\right.
\end{equation}

\end{enumerate}
\end{lem}
\begin{proof} 
Nous traitons en détails le cas (i), le cas (ii) est similaire. Un premier calcul donne
$$G\varphi(P')=
\begin{pmatrix}
(u^e+p)\varphi(\sigma_1') & u^{e-\gamma}(u^e+p)u^{e(p-1-c)}\varphi(\sigma_2')\\
u^{\gamma}(a\varphi(\sigma'_1)+u^{ec}\varphi(\sigma_3')) & 
au^{e(p-c)}\varphi(\sigma_2')+\varphi(\sigma_4')
\end{pmatrix}.$$
Remarquons que $au^{e(p-c)}\varphi(\sigma_2')+\varphi(\sigma_4')\equiv \sigma_4'\pmod {\m_{R[[u^e]]}}$ et $\det P'\equiv\sigma_1'\sigma_4'\pmod {\m_{R[[u^e]]}}$ est inversible car $P'\in\GL_2(R[[u]])$. Donc $\bar{\sigma}_4'$ est inversible dans $R$ et $\GG(G\varphi(P'))=I_\eta$. Il reste à montrer que $\BB(G\varphi(P'))$ satisfait (\ref{ecriturepassage}). 
Soit $b\in R$ tel que
$$b\equiv \frac{a\varphi(\sigma'_1)+u^{ec}\varphi(\sigma_3')}{
au^{e(p-c)}\varphi(\sigma_2')+\varphi(\sigma_4')}\pmod {u^e+p}.$$
Ainsi $\BB(G\varphi(P'))=\begin{pmatrix}\sigma_1&u^{e-\gamma}\sigma_2\\
u^{\gamma}\sigma_3&\sigma_4\\ \end{pmatrix}$ pour
$$\left\{\begin{array}{rcl}
\sigma_1&=&\varphi(\sigma'_1)-u^{e(p-c)}b\varphi(\sigma_2'),\cr
\sigma_2&=&(u^e+p)u^{e(p-1-c)}\varphi(\sigma_2'),\cr
(u^e + p)\: \sigma_3&=&a\varphi(\sigma'_1)+u^{ec}\varphi(\sigma_3')-abu^{e(p-c)}\varphi(\sigma_2')-b\varphi(\sigma_4'),\cr
\sigma_4&=&au^{e(p-c)}\varphi(\sigma'_2)+\varphi(\sigma_4').\cr\end{array}\right.$$
Des hypothèses sur $\sigma'_j$, $1\leq j\leq 4$ découle le résultat souhaité pour $\sigma_1,\sigma_4$. Pour $\sigma_2$, il suffit d'appliquer le lemme \ref{lemtechnique1}.  Pour $\sigma_3$,  nous avons
$$(u^e+p)\sigma_3=(a-b)+a\varphi(\sigma_1'-1)-b\varphi(\sigma_4'-1)+u^{ce}\varphi(\sigma_3')-abu^{e(p-c)}\varphi(\sigma_2').$$
Or
$$\frac{a\varphi(\sigma_1')+u^{ce}\varphi(\sigma_3')}{ au^{e(p-c)}\varphi(\sigma_2')+\varphi(\sigma_4')}\equiv a\pmod {I_t^{\varphi}}$$
 Comme $I_t^\varphi\pmod {u^e+p}=\pE^tR$, nous avons $a\equiv b\pmod {\pE^tR}$. Donc $(u^e+p)\varphi(\sigma_3)\in I_t^{\varphi}$. Le lemme \ref{lemtechnique1} implique $\sigma_3\in I_t$. Enfin 
$$ \frac{a\varphi(\sigma_1')+u^{ce}\varphi(\sigma_3')}{ au^{e(p-c)}\varphi(\sigma_2')+\varphi(\sigma_4')}=a\frac{\varphi(\sigma_1')}{\varphi(\sigma'_4)}+u^{ce}\frac{\varphi(\sigma'_3)}{\varphi(\sigma'_4)}-a^2u^{e(p-c)}\frac{\varphi(\sigma_1')\varphi(\sigma'_2)}{\varphi(\sigma'_4)^2}\pmod{ u^{2e}}.$$ 
Nous en déduisons  que, si $c\not\in\{1,p-1\}$, alors :
$$(u^e+p)\sigma_3\in u^{e}I_t^\varphi+\varphi(I_t\cap u^eR[[u^e]])+ p\pE^{t+1}R[[u^e]].$$
Si $c=1$,
$$(u^e+p)(\sigma_3-\varphi(\sigma_3')) \in u^{e}I_t^\varphi+\varphi(I_t\cap u^eR[[u^e]])+ p\pE^{t+1}R[[u^e]].$$
Si $c=p-1$,
$$(u^e+p)(\sigma_3+a^2\varphi(\sigma_2')) \in u^{e}I_t^\varphi+\varphi(I_t\cap u^eR[[u^e]])+ p\pE^{t+1}R[[u^e]].$$
Le lemme \ref{lemtechnique2} conclut.

\noindent
Dans le cas (iii), la stratégie est analogue. D'abord les termes diagonaux modulo $u^e$ de
$$G\varphi(P')=\begin{pmatrix}a\varphi(\sigma'_1)+u^{e(c+1)}\varphi(\sigma_3')&
u^{e-\gamma}(au^{e(p-1-c)}\varphi(\sigma_2')+\varphi(\sigma_4'))\\
u^{\gamma}(\varphi(\sigma_1')+a'u^{ec}\varphi(\sigma_3'))& u^{e(p-c)}\varphi(\sigma_2')+a'\varphi(\sigma_4')\\
\end{pmatrix}$$
ne sont pas inversibles dans $R$, donc $\GG(G\varphi(P'))=\II$. Ensuite nous définissons $b,b'\in R$ tels que
$$\begin{array}{ll}
\displaystyle
b\equiv\frac{a\varphi(\sigma'_1)+u^{e(c+1)}\varphi(\sigma_3')}{au^{e(p-1-c)}\varphi(\sigma_2')+\varphi(\sigma_4')}\pmod {u^e+p} \medskip \\
\displaystyle
b'\equiv\frac{ u^{e(p-c)}\varphi(\sigma_2')+a'\varphi(\sigma_4')}{\varphi(\sigma_1')+a'u^{ec}\varphi(\sigma_3')}\pmod {u^e+p}.$$
\end{array}$$
 La matrice
$\BB(G\varphi(P'))$ est alors de la forme $\begin{pmatrix}\sigma_1&u^{e-\gamma}\sigma_2\\
u^{\gamma}\sigma_3&\sigma_4\\ \end{pmatrix}$ pour
$$\left\{\begin{array}{rcl}
(u^e + p) \sigma_1&=& u^e(au^{e(p-1-c)}\varphi(\sigma'_2)+\varphi(\sigma_4'))-b'(a\varphi(\sigma_1')+u^{e(c+1)}\varphi(\sigma_3')) \cr
(u^e + p) \sigma_2&=& a\varphi(\sigma_1')+u^{e(c+1)}\varphi(\sigma_3')-b(au^{e(p-1-c)}\varphi(\sigma'_2)+\varphi(\sigma_4')) \cr
(u^e + p) \sigma_3&=& u^{e(p-c)}\varphi(\sigma_2')+a'\varphi(\sigma_4')-b'(\varphi(\sigma_1')+a'u^{ec}\varphi(\sigma_3')) \cr
(u^e + p) \sigma_4&=& u^e(\varphi(\sigma_1')+a'u^{ec}\varphi(\sigma_3'))-b(u^{e(p-c)}\varphi(\sigma_2')+a'\varphi(\sigma_4'))
\end{array}\right.$$
Nous avons ainsi $(u^e+p)\sigma_i\equiv 0\pmod {I_f^\varphi}$ pour $i=2,3$.
Comme dans le cas (i), nous obtenons $b\equiv a\pmod{\pE^tR}$ et $b'\equiv a'\pmod {\pE^tR}$ et $ab'\equiv a'b \equiv -p \pmod {\pE^tR}$. Donc $(u^e+p)(\sigma_i-1) \equiv 0\pmod {I_t^\varphi}$ pour $i=1, 4$.
Le lemme \ref{lemtechnique1} donne alors
$$\left\{\begin{array}{cc} \sigma_1\equiv\sigma_4\equiv 1\pmod {I_{t+1}}, \cr 
\sigma_2\equiv\sigma_3\equiv0 \pmod {I_{t+1}}.\cr\end{array}\right.$$
Les définitions de $b$ et $b'$ permettent d'obtenir les autres congruences annoncées.
\end{proof}

Soit $\MK$ un module de Breuil--Kisin de rang 2 de genre $\eta\oplus \eta'$ et de Frobenius donn\'e par la famille fixée des matrices $\underline{G}=(G^{(0)},\ldots,G^{(f-1)})$. Convenons d'étendre les suites $(\gamma_s)$ et $(G^{(s)})$ en des suites 
périodiques de période $f$ définies sur $\N$ tout entier. Soit la suite de matrices $(P^{(s)})_{s\in\N}$ définies par
$$P^{(0)}=\Id,\quad \BB(G^{(s)}\varphi(P^{(s)}))=\Delta^{(s)} P^{(s+1)}$$
avec 
$$P^{(s+1)}=\begin{pmatrix}\sigma_1^{(s+1)}& u^{e-\gamma_{s+1}}\sigma_2^{(s+1)}\\
u^{\gamma_{s+1}}\sigma_3^{(s+1)}&\sigma_4^{(s+1)}\end{pmatrix},$$
où $ \sigma_j^{(s+1)}\in R[[u^e]]$ ($1\leq j\leq 4$), $\sigma_1^{(s+1)}\equiv \sigma_4^{(s+1)}\equiv 1\pmod {u^e}$
et $\Delta^{(s)}$ est une matrice diagonale à coefficients dans $R$.
\begin{lem} 
\label{leminductionPassage}
Supposons que $\MK$ n'ait pas mauvais genre.
Alors, pour tout $0\leq j\leq f-1$ la suite des $P^{(j+fn)}$ converge vers une matrice $R^{(j)}$ dans $\mbox{GL}_2(R[[u]])$ quand $n$ tend vers l'infini.
\end{lem}

\begin{proof}
Étant donné deux nombres entiers $s$ et $t$ ainsi qu'une matrice $M$
prenant la forme :
$$M = \begin{pmatrix} \sigma_1 & u^{e-\gamma_{s}}\sigma_2\\ 
u^{\gamma_{s}}\sigma_3 & \sigma_4 \end{pmatrix}$$
convenons de dire que $M$ est \emph{$t$-proche de l'identité} si :
$$\sigma_1 \equiv \sigma_4 \equiv 1 \pmod{I_t, u^e}
\quad \text{et} \quad
\sigma_2 \equiv \sigma_3 \equiv 0 \pmod{I_t}.$$

Une récurrence sur $s$, à partir du lemme \ref{lemjolieecriture}, montre que, pour tout entier $s$, la matrice
$Q^{(s)} = P^{(s+f)} \cdot (P^{(s)})^{-1}$
est $t$-proche de l'identité pour $t$ égal à la partie entière de $\frac 
s f$. Ceci permet clairement de conclure.

\end{proof}

Sous les hypothèses du lemme \ref{leminductionPassage}, nous obtenons 
par passage à la limite des matrices 
$R^{(i)}=\lim_{n\rightarrow\infty}P^{(i+fn)}$, $0\leq i\leq f-1$ qui 
définissent un module de Breuil--Kisin isomorphe à $\MK$ de Frobenius 
donné par la famille de matrices : 
$$H^{(i)} = (R^{(i+1)})^{-1} \cdot G^{(i)} \cdot \varphi(R^{(i)})
\quad (0 \leq i \leq f-1)$$
qui prennent toutes l'une des trois formes suivantes :
\begin{itemize}
\item $\Delta^{(i)}\begin{pmatrix}u^e+p&0\\ au^{\gamma_{i+1}}& 1\end{pmatrix}$\\
\item $\Delta^{(i)}\begin{pmatrix}a&u^{e-\gamma_{i+1}}\\u^{\gamma_{i+1}}& b\end{pmatrix}$\\
\item $\Delta^{(i)}\begin{pmatrix}1 &u^{e-\gamma_{i+1}}\\0& u^e+p\end{pmatrix}$
\end{itemize}
où $\Delta^{(i)}$ est une matrice diagonale à coefficients dans $R$. Une dernière normalisation des vecteurs de base pour éliminer ces matrices diagonales $0\leq i<f-2$ permet d'obtenir l'écriture canonique suivante :

\begin{prop} \label{propecriturejolie2}

Soit $\MK=\MK^{(0)}\times\ldots\times\MK^{(f-1)}$ un module de 
Breuil--Kisin libre de rang 2 sur $R \hat\otimes \SK$ de type galoisien 
$\eta\otimes \eta'$ et de type de Hodge $\vv_0$ satisfaisant les 
hypothèses du lemme \ref{lemjolieecriture}.
Alors il existe $\alpha,\alpha'$ inversibles dans 
$R$ et pour tout $i\in \{0,\ldots,f-1\}$, il existe une base 
$e_\eta^{(i)}, e_{\eta'}^{(i)}$ de $\MK^{(i)}$ compatible à l'action de 
$\Gal(K/F)$ telle que si $0\leq i\leq f-2$ la matrice $G^{(i)}$ (resp. si $i=f-1$, la matrice 
$\begin{pmatrix}\alpha^{-1}&0\\0&\alpha'^{-1}\\ \end{pmatrix}$ $G^{(i)}$)
de $\varphi:\MK^{(i)}\rightarrow\MK^{(i+1)}$ soit de l'une des trois 
formes suivantes :
\begin{itemize}
\item {\bf Genre $\I_\eta$}:
$\begin{pmatrix} u^{e}+p& 0\\a_iu^{\gamma_{i+1}}& 1 \end{pmatrix}$ 
pour $a_i\in R$,
\item {\bf Genre $\I_{\eta'}$}:
$\begin{pmatrix} 1&a'_iu^{e-\gamma_{i+1}}\\0& u^e+p \end{pmatrix}$
pour $a'_i\in R$,
\item {\bf Genre $\II$}:
$\begin{pmatrix} a_i&u^{e-\gamma_{i+1}}\\u^{\gamma_{i+1}}& a'_i \end{pmatrix}$ 
avec $a_i,a'_i\in \m_R$ et $a_ia'_i=-p$.
\end{itemize}
\end{prop}

Comparons enfin deux écritures obtenues via la proposition \ref{propecriturejolie} d'un même module $\MK$. Pour la commodité du lecteur, nous rappelons ici l'énoncé (proposition \ref{prop:uniciteecriture}) qu'il nous faut montrer.
\begin{prop}
Nous conservons les hypothèses de la proposition \ref{propecriturejolie2} 
et nous supposons données deux écritures de $\MK$ correspondant à une
suite de genres $(g_0, \ldots, g_{f-1})$  et 
à des paramètres $(\alpha, \alpha', a_0, a'_0, \ldots, a_{f-1}, 
a'_{f-1})$ et $(\beta, \beta', b_0, b'_0, \ldots, b_{f-1}, b'_{f-1})$
respectivement.

En notant $n_i$ 
le nombre de $\II$ parmi $g_0, \ldots, g_{i-1}$, il existe un élément 
$\lambda$ de $R^\times$ vérifiant :
$$\text{pour tout } i \in \{0,  \ldots, f-1\}, \quad
b_i = \lambda^{(-1)^{n_i}} 
a_i \text{ et } b'_i = \lambda^{-(-1)^{n_i}} a'_i$$
et, de plus :
\begin{itemize}
\item si $n_f$ est pair, $\alpha = \beta$, $\alpha' = \beta'$,\\
tandis que
\item si $n_f$ est impair, $\alpha \alpha' = \beta \beta'$ et
$\lambda = \frac{\alpha}{\beta} = \frac{\beta'}{\alpha'}.$
\end{itemize}
\end{prop}

\begin{proof}
Soit $(G^{(i)})_{0\leq i\leq f-1}$ et $(H^{(i)})_{0\leq i\leq f-1}$ les 
matrices correspondant aux deux écritures de $\MK$. Il existe une 
famille $(P^{(i)})_{0\leq i\leq f-1}$ de matrices de passage qui 
préservent la donnée de descente et vérifient :
\begin{equation}\label{genredepassage}
(P^{(i+1)})^{-1}G^{(i)}\varphi(P^{(i)})=H^{(i)} 
\quad \text{pour } 0\leq i\leq f-1.
\end{equation}
Un examen attentif de l'algorithme de réduction de la preuve de la 
proposition \ref{propecriturejolie2}, montre que les $P^{(i)}$ sont
nécessairement de la forme 
$$P^{(i)}=\begin{pmatrix}\lambda_i&0\cr 0&\mu_i\cr\end{pmatrix}
\quad \text{pour } 0\leq i\leq f-1.$$
Les égalités \eqref{genredepassage}
conduisent pour $0\leq i\leq f-1$ :
\begin{itemize}
\item si $g_i=h_i=\I_\eta$ ou $\I_{\eta'}$,
$$\begin{array}{ll}\mbox{ si } 0\leq i\leq f-2 \, :&
\left\{\begin{array}{ll}
a_i=\mu_{i+1}b_i/\lambda_i, \, a_i'=\lambda_{i+1}b'_i/\mu_i, \smallskip \cr 
\lambda_i=\lambda_{i+1}, \, \mu_i=\mu_{i+1},\cr\end{array}\right. \smallskip \cr
\mbox{ si } i=f-1\,:&
\left\{\begin{array}{ll}
a_{f-1}=\beta'\mu_0b_{f-1}/(\alpha'\lambda_{f-1}), \, a'_{f-1}=\beta\lambda_0b'_{f-1}/(\alpha\mu_{f-1})\cr
\mu_0=\alpha'\mu_{f-1}/\beta', \, \lambda_0=\alpha\lambda_{f-1}/\beta.\cr\end{array}\right.\cr\end{array}$$
\item si $g_i=h_i=\II$,
$$
\begin{array}{ll}\mbox{ si } 0\leq i\leq f-2 \, : &
\left\{\begin{array}{ll}
a_i=\lambda_{i+1}b_i/\lambda_i, \, a_i'=\mu_{i+1}b_i'/\mu_i,\cr
 \lambda_i=\mu_{i+1}, \, \mu_i=\lambda_{i+1},\cr\end{array}\right. \smallskip \cr
\mbox{ si } i=f-1 \, :&
\left\{\begin{array}{ll}a_{f-1}=\beta\lambda_0b_{f-1}/(\alpha\lambda_{f-1}), \, a'_{f-1}=\beta'\mu_0b'_{f-1}/(\alpha'\mu_{f-1}) \smallskip \cr
\mu_0=\alpha'\lambda_{f-1}/\beta',\, \lambda_0=\alpha\mu_{f-1}/\beta\cr\end{array}\right.\cr
 \end{array}$$
\end{itemize}
En posant $\lambda=\lambda_0/\mu_0\in R^{\times}$, nous obtenons le résultat annoncé.
\end{proof}

\subsection{Application au calcul d'espaces de déformations}
\label{ssec:calculdeform}

Adaptant les idées de \cite{BM2} à notre situation plus générale, nous 
expliquons, dans ce numéro, comment utiliser les résultats de la partie 
précédente pour déterminer les espaces de déformations $R^\psi(\vv_0, 
\eta \oplus \eta', \rhobar)$ avec $\rhobar$ absolument irréductible et $\vv_0$, $\ttt$ fixés comme au début de la partie \ref{sec:methode}.

Nous commençons par rappeler ou établir quelques résultats préparatoires 
concernant, d'une part, le calcul de la représentation résiduelle d'une 
représentation donnée par un module de Breuil--Kisin et, d'autre part, le 
calcul explicite de certains groupes d'extensions de représentations de 
$\Goo$.

\subsubsection{Calcul de la représentation résiduelle}
\label{sssec:represid}

Soient $R$ une $\oE$-algèbre locale complète noetherienne de corps 
résiduel $k_E$ et $\MK$ un module de Breuil--Kisin sur $R \hat\otimes_{\Zp} \SK$ libre de rang 
$2$, avec donnée de descente de $L$ à $F$. Comme $R$ est en particulier une $W$-algèbre, rappelons la décomposition \eqref{decompBKmodule}
$$\MK=\MK^{(0)}\oplus\ldots\oplus\MK^{(f-1)}.$$
Nous supposons de plus que chaque $\MK^{(i)}$ (pour $i$ entre $0$ et $f-1$) est muni d'une 
base $(e_\eta^{(i)},e_{\eta'}^{(i)})$ sur laquelle la donnée de descente agit par $\eta \oplus \eta'$ et telle que  la matrice $G^{(i)}$ 
du Frobenius $\varphi: \MK^{(i)}\longrightarrow\MK^{(i+1)}$ est donnée par l'une des 
trois formes de la proposition~\ref{propecriturejolie2}.

Notons $\rho_\MK $ la $R$-représentation libre de rang $2$ de 
$G_{\infty}$ associée à $\MK$ et $\rhobar_\MK$ sa réduction modulo 
l'idéal maximal de $R$.
Dans cette partie, nous expliquons comment s'assurer que la 
$k_E$-représentation résiduelle $\rhobar_\MK$ est irréductible et, 
lorsqu'elle l'est, comment la déterminer complètement. À partir de 
maintenant, pour alléger les notations, nous noterons $\phi$ à la
place de $\varphi^f$.

Soit $\bbM_L = \ocEL \otimes_{\SK_L} \MK$ le $\varphi$-module sur 
$R \hat\otimes_{\Zp} \ocEL$ (avec action semi-linéaire de $\Gal(L/F)$) 
associé à $\rho_{\MK | G_{L_\infty}}$ et soit $\bbM_F$ le $\varphi$-module 
sur $R \hat\otimes_{\Zp} \ocEF$ associé à $\rho_\MK$.
D'après la partie \ref{ssec:descente} (relation \eqref{eq:MKML}), $\bbM_F$ est l'ensemble des 
points fixes de $\bbM_L$ par l'action de $ \Gal (L / F)$ : pour 
tout $i$ entre $0$ et $f-1$, nous avons $\bbM_F^{(i)} = H^0(\Gal(L/F),
\bbM_L^{(i)})$. De l'écriture de $\MK$, nous déduisons

\begin{prop}\label{prop:matrices} 
Pour tout $i$ entre $0$ et $f-1$, la famille
$$(1\otimes u^{e - \gamma_i - \beta_i } e^{(i)}_{\eta} , 1\otimes u^{e - \beta_i } e^{(i)}_{\eta'})$$
est une base de $\bbM_F^{(i)}$ comme $R \hat\otimes_{W, \iota \circ \varphi^{-i}} \ocEF $-module. Dans ces bases, la matrice $B^{(i)}$ 
(resp., si $i = f-1$, la matrice 
$\begin{pmatrix}
\alpha^{-1} & 0 \\
0 & \alpha'^{-1} 
\end{pmatrix}
B^{(i)}
 $)
 du 
Frobenius de $\bbM_F^{(i)}$ dans $\bbM_F^{(i+1)}$ est :
	\begin{itemize}
	\item {\rm Genre $\Ieta$ :}
	$\begin{pmatrix}
	 v^{d_i}(v + p) & 0\\
	 a_i v^{d_i} & v^{p-1}
	\end{pmatrix}
	v^{- b_{f-1-i}}$ ;
	\item  {\rm Genre $\Ietap$ :}
	$\begin{pmatrix}
	 v^{d_i} & a'_i v^{p} \\
	 0 & v^{p-1}(v + p)
	\end{pmatrix}
	v^{- b_{f-1-i}}$ ;
	\item  {\rm Genre $\II$ :}
	$\begin{pmatrix}
	 a_i v^{d_i} & v^{p} \\
	 v^{d_i} &  a'_i v^{p-1} %
	\end{pmatrix}
	v^{- b_{f-1-i}}$.
	\end{itemize}
où $v=u^e$ et $d_i=p-1-c_{f-1-i}$ pour tout  $i$ dans $ \{0, \ldots, f-1\}$.
\end{prop}

\begin{proof}
Soit $i$ un indice entre $0$ et $f-1$.
Tout élément de $\bbM_L^{(i)}$ s'écrit de manière unique sous la
forme $\lambda e^{(i)}_{\eta} + \lambda' e^{(i)}_{\eta'}$ où $\lambda$ 
et $\lambda'$ s'écrivent eux-mêmes respectivement
$$\lambda = \sum\limits_{j \in \Z} a_j u^j
\quad \text{et} \quad 
\lambda' = \sum\limits_{j \in \Z} a'_j u^j$$
avec $a_j, a'_j$ dans $R$ et $\lim\limits_{j \to - \infty} a_j = 
\lim\limits_{j \to - \infty} a'_j = 0$. Avec ces notations, un calcul 
immédiat montre que, pour tout $g \in \Gal(L/K)$, nous avons
$g\big(\lambda e^{(i)}_{\eta} + \lambda' e^{(i)}_{\eta'}\big) =
\mu e^{(i)}_{\eta} + \mu' e^{(i)}_{\eta'}$ avec
$$\mu = \sum\limits_{j \in \Z} a_j \cdot \varphi^{-i}\big( \varepsilon_f(g)^{j + \beta_i + \gamma_i} \big)\cdot u^j
\quad \text{et} \quad
\mu' = \sum\limits_{j \in \Z} a'_j \cdot \varphi^{-i}\big( \varepsilon_f(g)^{j + \beta_i} \big)\cdot u^j.
$$
Ainsi $\lambda e^{(i)}_{\eta} + \lambda' e^{(i)}_{\eta'}$ est fixe par 
$\Gal(L/F)$ si, et seulement si les seuls coefficients non nuls dans 
$\lambda$ (resp. $\lambda'$) sont ceux pour lesquels $j + \beta_i + 
\gamma_i$ (resp. $j + \beta_i$) est divisible par $e$. Ceci 
démontre que la famille indiquée est une base de $\bbM_F^{(i)}$.

Les matrices du Frobenius dans ces nouvelles bases sont obtenues en 
multipliant les matrices de la proposition \ref{propecriturejolie} à 
gauche par
$ 
\begin{pmatrix}
u^{\gamma_{i + 1}} & 0\\
0 & 1 \\
\end{pmatrix}
u^{\beta_{i + 1} - e}
$
et à droite par 
$ 
\begin{pmatrix}
u^{-p\gamma_i} & 0\\
0 & 1 \\
\end{pmatrix}
u^{-p(\beta_i - e)}
$
puis en utilisant la formule
$\gamma_{i + 1} - p\gamma_i = -ec_{f-1-i}$ et son équivalent pour les $b_i$.
\end{proof}

Par le théorème \ref{thm:equivrepPhimod} et la proposition 
\ref{propnomfarfelu}, $\rho_\MK$ est aussi la $R$-représentation 
associée à $\bbM_F^{(0)}$ muni d'une structure de $\phi$-module sur 
$R \hat\otimes_{W} \ocEF$ par l'endomorphisme $\varphi_{\bbM_F}^f$
que, par un léger abus de notations, nous notons encore $\phi$.
Avec les notations de la proposition \ref{prop:matrices}, pour tout $i$ 
entre $0$ et $f-1$, $B^{(i)}$ désigne la matrice associée à $\varphi$ de 
$\bbM_F^{(i)}$ dans $\bbM_F^{(i+1)}$. Ainsi, dans la base de 
$\bbM_F^{(0)}$ donnée par la proposition \ref{prop:matrices}, la matrice 
à coefficients dans $R \hat\otimes_{W} \ocEF$ de l'endomorphisme $\phi$ est :
\begin{equation}
\label{eq:prodB}
B = B^{(f-1)} \cdot \varphi(B^{(f-2)}) \cdots \varphi^{f-2}(B^{(1)}) \cdot
\varphi^{f-1}(B^{(0)}).
\end{equation}
Remarquons que, d'après la forme des matrices $B^{(i)}$ dans la proposition \ref{prop:matrices},
les coefficients de la matrice $B$ obtenue par la formule \eqref{eq:prodB} sont en fait dans $R \hat\otimes_{W} \SK_F = R[[v]]$.

Par construction (voir \S \ref{sssec:classifGoo}), la 
$k_E$-représentation $\rhobar_\MK$ est la $k_E$-représentation de 
$G_{\infty}$ associée à $k_E \otimes_R \bbM_F^{(0)}$, vu comme 
$\phi$-module sur $k_E((v))$.
Notons 
$\overline{B} = 
\begin{pmatrix}
a  & b\\
c & d
\end{pmatrix}
$ 
la réduction, à coefficients dans $k_E((v))$, de la matrice $B$. Nous 
définissons la trace et le déterminant tordus :
$$\textstyle
T(\overline{B}) = \phi(a) + d\cdot \frac{\phi(c)}{c}, \quad
\Delta(\overline{B}) = \frac{\phi(c)}{c} \cdot (ad -bc),$$
avec la convention $\frac{\phi(c)} c = 0$ si $c=0$.
Avec ces notations et en notant en plus $\nr'(\delta)$ le caractère non 
ramifié de $G_{F'}$ dans $k_E^\times$ qui envoie le Frobenius arithmétique 
sur un élément $\delta$ de $k_E^\times$, le corollaire III.1.4.8 de \cite{LB} s'énonce 
ainsi :

\begin{prop}
\label{prop:represid}
La représentation $\rhobar_\MK$ est absolument irréductible si et 
seulement si $c$ est non nul et les deux conditions suivantes sont 
réalisées :
\begin{enumerate}[(i)]
\item $(p^f + 1) (\val_v(T(\overline{B})) > p^f \val_v ( \Delta(\overline{B} ))$ ;
\item $\val_v ( \Delta(\overline{B} )) \not \equiv 0 \pmod {p^f + 1}$.
\end{enumerate}
Lorsque c'est le cas, notons $h$ et $\delta$ les éléments respectifs de 
$\Z$ et $k_E^\times$ définis par la congruence $\Delta(\overline{B}) 
\equiv  - \delta v^h \pmod{v^{h+1}}$ (\emph{i.e.} $- \delta v^h$ est le terme 
de plus bas degré de $\Delta(\overline{B})$). Alors, il existe une
base de $k_E \otimes_R \bbM_F^{(0)}$ dans laquelle la matrice de $\phi$ 
est :
$$\begin{pmatrix}
0 & \delta v^h \\
1 & 0
\end{pmatrix}$$
et la représentation $\rhobar_\MK$ est isomorphe à 
$\Ind^{G_F}_{G_{F'}} \big(\omega_{2f}^{-h} \cdot \nr'(\delta^{-1})\big) \otimes
\omega$ (rappelons que $\omega$ désigne le caractère cyclotomique).
\end{prop}

\begin{rem}
\label{rem:IIimpair}
Supposons que les paramètres $a_i$ et $a'_i$ correspondant à $\MK$ 
introduits dans la proposition \ref{propecriturejolie2} sont tous dans l'idéal maximal 
de $R$ et que le genre de $\MK$ fait intervenir un nombre pair de
composantes $\II$.
Alors, la matrice $\overline{B}$ est le produit de matrices diagonales 
et d'un nombre pair de matrices anti-diagonales. Par conséquent, elle est 
diagonale et la représentation résiduelle $\rhobar_\MK$ est réductible.
\end{rem}

\subsubsection{Calcul de $\Ext_{\Goo}^1(\rhobar,\rhobar)$}
\label{sssec:calculExt1}

Nous fixons dans ce paragraphe $\rhobar$ une~$k_E$-représentation 
irréductible de dimension~$2$ de~$G_F$.
Une telle représentation est de la forme $\Ind^{G_F}_{G_{F'}}\big( 
\omega_{2f}^{-h} \otimes \nr'(\delta^{-1}) \big) \otimes \omega$, pour un certain entier $h$ 
non nul modulo $p^f + 1$ et $\delta$ un élément de $k_E^\times$.

L'objectif de cette partie est de décrire l'espace~$\Ext_{\Goo}^1(\rhobar,\rhobar)$ des extensions de $\rhobar$ par elle-même dans la catégorie des représentations de $\Goo$.
D'après le lemme~\ref{lem:rinj} et la remarque qui le suit (avec $d = 2$ et $p \geq 5$),
cet espace contient (canoniquement) celui
des extensions de $\rhobar$ par elle-même comme représentation de $G_F$.

D'après le théorème \ref{thm:equivrepPhimod} (pour $R = k_E$),
la catégorie des $k_E$-représentations de~$\Goo$ de dimension finie
est équivalente à 
celle des~$\phi$-modules étales sur~$k_E (( v ))$.
Soit $D(\rhobar)$ le $\phi$-module étale sur~$k_E (( v ))$, de dimension $2$, associé à $\rhobar$.
Ainsi l'espace~$\Ext_{\Goo}^1(\rhobar,\rhobar)$ que nous voulons décrire est  isomorphe 
à l'espace $\Ext^1(D(\rhobar), D(\rhobar)) $  des extensions de $D(\rhobar)$ par lui-même dans la catégorie des $\phi$-modules étales sur $k_E((v))$.
Fixons une base $\cB$ de $D(\rhobar)$ et notons $B$ la matrice 
de $\phi$ dans cette base.
Pour toute matrice~$A$ dans $\Mat_2(k_E((v)))$, nous notons $D_A$ le $\phi$-module %
dont l'espace vectoriel sous-jacent est $D(\rhobar) \oplus D(\rhobar)$ et de Frobenius dans la base $\cB \oplus \cB$ donné par la matrice 
$
\begin{pmatrix}
B & A \\
0 & B 
\end{pmatrix}.
$
L'application
$$
\begin{array}{r c l}
\frac{\Mat_2(k_E((v)))}{\mathrm{Im}( X \mapsto B\phi(X) - XB ) } & \longrightarrow & \Ext^1(D(\rhobar), D(\rhobar)) \\
A & \longmapsto & \text{classe de } D_A
\end{array}
$$
est alors un isomorphisme de $k_E$-espaces vectoriels.

Il s'agit maintenant de déterminer l'image de l'application de $\Mat_2(k_E((v)))$ dans lui-même qui envoie une matrice $X$ sur  la matrice $B\phi(X) - XB$.
Nous pouvons supposer que la base $\cB$ a été choisie de façon à ce que
la matrice~$B$ du Frobenius $\phi$ dans $\cB$ soit de la forme
$$
B = 
\begin{pmatrix}
0 & \delta v^h \\
1 & 0 
\end{pmatrix},
$$
avec $\delta$ dans~$k_E^\times$, $0\leq h\leq q^2-2$ et $h\not\equiv 
0\pmod{q+1}$.

\begin{prop}\label{propdecompositionext1}
L'application
$$
\begin{array}{r c l}
\frac{\Mat_2(k_E((v)))}{\mathrm{Im}( X \mapsto B\phi(X) - XB ) } & \longrightarrow & \frac{k_E((v))}{\mathrm{Im}( v^{(q-1)h} \phi^2 - \Id )} \oplus \frac{k_E((v))}{\mathrm{Im}( \phi^2 - \Id )} \\
\begin{pmatrix}
a & b \\
c & d
\end{pmatrix}
& \longmapsto & \left( \frac{1 }{\delta v^{h}} \left(  \phi(a)  +d  \right) , \frac{1}{\delta v^{h}}    b +  \phi(c)  \right)
\end{array}
$$
est un isomorphisme de $k_E$-espaces vectoriels.
\end{prop}

\begin{lem}\label{lemdecompositionext1}\ 
\begin{enumerate}[(i)]
\item Le quotient $\frac{k_E((v))}{\mathrm{Im}( v^{(q-1)h} \phi^2 - \Id )}$ a pour base 
$$\textstyle
\big\{
v^n, \,\,
n < - \frac{h}{q+1} \text{ et } n \not\equiv (q-1)h \!\! \pmod {q^2}
\big\}
.$$
\item Le quotient $ \frac{k_E((v))}{\mathrm{Im}( \phi^2 - \Id )}$ a pour base
$$
\big\{
v^n, \,\,
n < 0 \text{ et } n \not\equiv 0 \!\! \pmod {q^2}
\big\}
\, \cup \,
\left\{1 
\right\}
.$$
\end{enumerate}
\end{lem}

\begin{proof}
Soit $m$ un entier relatif.
Notons  $f_m$ l'application $k_E$-linéaire de $k_E(( v))$ dans lui-même qui envoie $x$ sur $ v^m \phi(x) - x$.
Notons $M$ le rationnel $\frac{-m}{ q^2 - 1} $, point fixe de $z \mapsto q^2 z+ m $.
Nous traitons simultanément les deux cas du lemme en démontrant que le quotient $k_E ((v)) / \mathrm{Im} f_m   $ a pour base l'ensemble
$$
\cB_m =
\big\{
v^n, \,\,
n < M \text{ et } n \not\equiv m \!\! \pmod {q^2}
\big\} 
\, \cup \,
\big\{v^M, \,\, \text{ si } M \text{ est entier} \big\}
.
$$
Remarquons que dans le point (i), $M$ ne peut être un entier
à cause de l'hypothèse $h \not\equiv 0 \pmod{q+1}$.

Les égalités et inégalités 
\begin{equation}
\label{ineg}
\forall x \in k_E((v)) \quad 
\left\{
\begin{array}{r c l}
\val(x) > M & \Longrightarrow &  \val_v(f_m(x)) = \val_v(x) > M \\
\val(x) < M & \Longrightarrow & \val_v(f_m(x)) = m + q^2\val_v(x) < M
\end{array}
\right.
\end{equation}
assurent la continuité de $f_m$ et que l'image de $f_m$ est fermée.
Notons que, comme $v^M$ est dans le noyau de $f_m$, tout élément de l'image de $f_m$ est de valuation différente de $M$ et a un antécédent de valuation différente de $M$.

Nous affirmons que, pour tout entier $n$ strictement supérieur à $M$, 
$v^n$ est dans l'image de $f_m$. En effet, la suite $(n_i)_{i \in \N}$ 
définie par $n_0 = n$ et $n_{i+1} = q^2n_i + m$ tend vers $+ \infty$ et 
la série $- \sum_{i \geq 0} v^{n_i} $ ainsi définie a pour image $v^n$ 
par $f_m$. Ainsi, toute classe dans $k_E ((v)) / \mathrm{Im} f_m $ a un 
représentant qui est un polynôme de Laurent de degré inférieur ou égal à 
$M$.

De plus, pour tout entier $n$, l'élément $v^{m + q^2 n} - v^n$ est dans 
l'image de $f_m$. Soit $n$ un entier strictement inférieur à $M$ et 
congru à $m$ modulo $q^2$. Alors l'entier $n' = \frac{n - m}{q^2}$ est 
strictement compris entre $n$ et $M$ et les puissances $v^n$ et $v^{n'}$ 
coïncident dans $k_E ((v)) / \mathrm{Im} f_m $. Ceci implique que 
l'ensemble $\cB_m$ engendre $k_E ((v)) / \mathrm{Im} f_m $.

Pour démontrer enfin que la famille $\cB_m$ est libre, il nous suffit enfin de remarquer que tout élement de $ \mathrm{Im} f_m $
de valuation inférieure ou égale à $M$ est en fait de valuation strictement inférieure à $M$ et congrue à $m$ modulo $q^2$. 

\end{proof}

\begin{rem}
\label{rem:imageExt1}
Avec les notations de la démonstration précédente, nous pouvons 
déterminer pratiquement l'image d'une série $\sum a_i v^i $ dans $k_E 
((v)) / \mathrm{Im} f_m $ en procédant comme suit :
\begin{enumerate}[(1)]
\item nous tronquons les puissances strictement supérieures à $\frac{-m}{ q^2 - 1} $ ;
\item pour tout $n$ strictement inférieur à $\frac{-m}{ q^2 - 1} $, nous remplaçons $v^{m + q^2 n}$ par $v^n$, jusqu'à avoir fait disparaître toutes les puissances congrues à $m$ modulo $q^2$.
\end{enumerate}
\end{rem}

\subsubsection{Exposé sommaire de la stratégie}
\label{sssec:methode}

Soient $\rhobar$ une représentation irréductible de $G_F$ dans 
$\GL_2(k_E)$, $\ttt$ un type galoisien de la forme $\eta \oplus \eta'$ 
et $\vv_0$ le type de Hodge $((0,2))_{\tau \in \SSS} $ choisis comme au début de la partie \ref{sec:methode}.
Dans cette partie, nous présentons une méthode générale de détermination de l'anneau de déformations $ R^\psi(\vv_0, \ttt , \rhobar)$. Nous appliquons notamment cette méthode dans la partie \ref{sec:degre2} lorsque $F$ est de degré $2$ sur $\Qp$.

Rappelons que le caractère $\psi$ est fixé et satisfait une relation de compatibilité \eqref{relpsi} avec $\vv_0$ et $\ttt$.
Pour nos choix de $\vv_0$ et $\ttt$, $\psi$ est donc de la forme 
\begin{equation}
\label{eqn:psidelta}
\psi = \eta \eta' \cdot \nr(- \delta^{-1})
\end{equation}
où $\delta$ est un élément dans $\oE^\times$, fixé.

\paragraph{Première étape : les engeances}

Dans \cite{Ki-Moduli} puis dans \cite{Ki3}, Kisin démontre qu'il existe 
des variétés définies sur $\oE$ qui paramètrent les modules de 
Breuil--Kisin de hauteur $\leq h$ (pour un certain entier $h$ fixé) à 
coefficients dans des $\oE$-algèbres locales complètes et noetheriennes 
variables qui sont munis éventuellement de structures supplémentaires. 
Les fibres spéciales de ces variétés sont des variétés algébriques 
quasi-projectives réduites définies sur $k_E$ que Pappas et Rapoport 
appelent \emph{variétés de Kisin} dans \cite{PR}.

Étant donnés une représentation irréductible $\rhobar: G_F \rightarrow 
\GL_2(k_E)$ et un type galoisien $t = \eta \oplus \eta'$, notons 
$\GG\RR_{\rhobar, \psi,\vv_0, \ttt}$ la\footnote{Les variétés de Kisin 
apparaissant naturellement comme des sous-schémas localement fermés et 
réduits d'espaces projectifs, la donnée de leurs ensembles de points dans 
toutes les extensions finies de $k_E$ suffit à les déterminer 
entièrement.} variété de Kisin dont les $k$-points s'identifient, pour 
toute extension finie $k$ de $k_E$, à l'ensemble des modules de 
Breuil--Kisin sur $k \otimes_{\Zp} \SK$ de type $(\vv_0, \ttt, \psi)$ qui 
sont inclus dans $k \otimes_{k_E} \bbM(\rhobar_{|\GooL})$ muni de sa 
donnée de descente.

\begin{definit} 
Soient $\rhobar: G_F \rightarrow \GL_2(k_E)$ une représentation 
irréductible de $G_F$ et $t = \eta \oplus \eta'$ un type galoisien.
Si $E'$ est une extension finie de $E$, un $k_{E'}$-point de la variété 
de Kisin $\GG\RR_{\rhobar,\psi,\vv_0,\ttt}$ est appelé une 
\emph{$E'$-engeance} de type $t$ de $\rhobar$.
\end{definit}

\begin{rem}
Lorsque $E'$ est $E$ et que la situation ne prête pas à confusion, nous
dirons simplement \emph{engeance} à la place de $E$-engeance.
\end{rem}

De la proposition \ref{propecriturejolie} appliquée à $R = k_E$ et du 
fait que $\rhobar$ est absolument irréductible, nous déduisons que la 
donnée d'une engeance est équivalente à la donnée d'un genre $(g_0, 
\ldots, g_{f-1})$ et d'une famille de paramètres $(\alpha, \alpha', a_0, 
a'_0, \ldots, a_{f-1}, a'_{f-1})$ vivant dans $k_E$, satisfaisant aux 
conditions de la proposition \ref{propecriturejolie} et considérés modulo 
la relation d'équivalence décrite par la proposition 
\ref{prop:uniciteecriture}.

La première étape consiste à déterminer la liste des engeances de type 
$\ttt$ de la représentation $\rhobar$. Pour y parvenir, une possibilité 
consiste à parcourir tous les genres et paramètres possibles puis, en 
utilisant la proposition \ref{prop:represid}, à sélectionner ceux  donnant lieu à une représentation résiduelle isomorphe à $\rhobar$. 
Cette étape est entièrement algorithmique et nous l'avons implémentée en 
{\tt sage}. Nous renvoyons le lecteur à la page web
\begin{center}
\url{https://cethop.math.cnrs.fr:8443/home/pub/14}
\end{center}
pour une présentation rapide de notre logiciel {\tt sage} ainsi que des
exemples de calculs.

Remarquons toutefois qu'il devrait être également possible de déterminer 
la liste des engeances en revenant à la définition et en calculant la 
variété de Kisin $\GG\RR_{\rhobar,\psi,\vv_0,\ttt}$. Cette approche sera 
développée dans un travail ultérieur.

\paragraph{Deuxième étape : construction d'une famille de morphismes}
\label{par:etape2}

Soit $E'$ une extension finie de $E$.
Pour chaque $E'$-engeance $\xi' = ( (g_i)_i , (\alpha, \alpha') , (a_i, 
a'_i)_i )$ de type $\ttt$ de $\rhobar$, nous définissons une 
$\oEp$-algèbre locale complète noetherienne $R_{\xi',E'}$, de corps 
résiduel $k_{E'}$, ainsi qu'un module de Breuil--Kisin $\MK_{\xi',E'}$ 
sur $R_{\xi',E'}$ comme suit :

\begin{enumerate}[(a)]
\item Supposons que le genre de $\xi'$ comporte un nombre impair de 
facteurs $\II$. Nous posons :
$$
R_{\xi',E'} = \frac{\oEp [[X_i, Y_i, i \in \II , T_i, i \not\in \II ]]}{(X_iY_i + p, i \in \II)}.
$$
Par la remarque \ref{rem:constrmodBK}, nous obtenons un module de 
Breuil--Kisin $\MK_{\xi',E'}$ sur $R_{\xi',E'}$ en remplaçant dans les formules 
de la proposition \ref{propecriturejolie} :
\begin{itemize}
\item[$\bullet$] $\alpha$ par $1$ et $\alpha'$ par $\delta$ ;
\item[$\bullet$] pour les indices $i$ de genre $\I_\eta$, $a_i$ par $[a_i] + T_i$ ;
\item[$\bullet$] pour les indices $i$ de genre $\I_{\eta'}$, $a'_i$ par $[a'_i] + T_i$ ;
\item[$\bullet$] pour les indices $i$  de genre $\II$, $a_i$ par  $X_i$ et $a'_i$ par $Y_i$.
\end{itemize}
\item Supposons que le genre de $\xi'$ comporte un nombre pair de facteurs 
$\II$. Alors, d'après la remarque \ref{rem:IIimpair}, le genre de $\xi'$ 
possède aussi un facteur (de genre $\I_\eta$ ou $I_{\eta'}$) avec un paramètre non 
nul dans $k_E$. Fixons $i_0$ l'indice d'un tel facteur.
Nous posons :
$$
R_{\xi',E'} = \frac{\oEp [[X_i, Y_i, i \in \II , T_i, i \not\in \II \cup \{ i_0 \} , Z ]]}{(X_iY_i + p, i \in \II)}.
$$
Par la remarque \ref{rem:constrmodBK}, nous obtenons un module de 
Breuil--Kisin $\MK_{\xi',E'}$ sur $R_{\xi',E'}$ en remplaçant dans les formules
de la proposition \ref{propecriturejolie2}:
\begin{itemize}
\item[$\bullet$] $\alpha$ par $[\alpha] + Z$ et $\alpha'$ par $ - \delta( [\alpha] + Z )^{-1} $;
\item[$\bullet$] $a_{i_0}$ (si $i_0$ est de genre $\I_\eta$) ou  $a'_{i_0}$ (si $i_0$ est de genre $\Ietap$) par $1$ ;
\item[$\bullet$] pour les indices $i \neq i_0$ de genre $\I_\eta$, $a_i$ par $[a_i] + T_i$ ;
\item[$\bullet$] pour les indices $i \neq i_0$ de genre $\I_{\eta'}$, $a'_i$ par $[a'_i] + T_i$ ;
\item[$\bullet$]  pour les indices $i$  de genre $\II$, $a_i$ par  $X_i$ et $a'_i$ par $Y_i$.
\end{itemize}
\end{enumerate}
Dans les deux cas (a) et (b) ci-dessus, la $R_{\xi',E'}$-représentation de
$G_F$ associé à $\MK_{\xi',E'}$ est de déterminant $\psi\varepsilon$ et se réduit sur $\rhobar$ modulo l'idéal
maximal.
Nous obtenons donc un morphisme de $\oEp$-algèbres (locales) $f_{\xi',E'}$ 
de $\oEp \otimes_{\oE} R^\psi(\rhobar)$ dans $R_{\xi',E'}$, qui induit 
l'identité sur les corps résiduels $k_{E'}$.

\begin{lem}
\label{lem:factorf}
Pour tout $E'$-engeance $\xi'$, le morphisme $f_{\xi',E'}$ se factorise par 
$\oEp \otimes_{\oE} R^{\psi} (\vv_0, \ttt, \rhobar)$.
\end{lem}

\begin{proof}
Convenons de dire qu'un morphisme de $\oEp$-algèbres $\alpha : 
\oEp \otimes_{\oE} R^\psi(\rhobar) \to \Zpbar$ est de type $(\vv_0, \ttt, \psi)$ si la 
$\Qpbar$-représentation de $G_F$ qu'il définit est de type $(\vv_0, 
\ttt, \psi)$. La $\oEp$-algèbre $\oEp \otimes_{\oE} R^\psi(\vv_0 , \ttt, 
\rhobar)$ est alors égale au quotient de $\oEp \otimes_{\oE} 
R^\psi(\rhobar)$ par l'idéal $I = \cap \ker f$ où l'intersection est 
étendue à tous les morphismes $f$ de type $(\vv_0, \ttt, \psi)$ (voir la remarque \ref{rem:defdefo}).

Pour démontrer le lemme, il suffit de démontrer que tout élément $x$ dans $I$ 
s'envoie sur $0$ par $f_{\xi',E'}$. Considérons un élément $x$ dans $I$ et 
posons $y = f_{\xi',E'}(x)$. Soit également $g : R_{\xi',E'} \to \Qpbar$ un 
morphisme de $\oEp$-algèbres. La composée $g \circ f_{\xi',E'}$ est alors 
de type $(\vv_0, \ttt, \psi)$. Nous déduisons qu'elle s'annule sur $x$, ce qui 
revient à dire que $g(y) = 0$. Ainsi $y$ appartient à l'idéal $J = \cap 
\ker g$ où l'intersection est étendue à tous les morphismes $g$ comme 
ci-dessus. Or, à partir de la forme explicite de $R_{\xi', E'}$, nous 
déduisons aisément que l'idéal $J$ est nul. Ainsi nous avons bien démontré que $y = 
f_{\xi',E'}(x) = 0$ comme annoncé.
\end{proof}

Notons
$\tilde{f}_{\xi',E'} : \oEp \otimes_{\oE} R^{\psi} (\vv_0, \ttt, \rhobar)
\to R_{\xi',E'}$ l'application
correspondant à la factorisation du lemme précédent. Soit aussi $f_{E'}$ 
l'application produit $\prod_{\xi'} f_{\xi',E'}$ où le produit est étendu 
à toutes les $E'$-engeances $\xi'$ de type $\ttt$ de $ \rhobar$. Elle prend ses valeurs dans le
\og produit local \fg
$$R_{\expl,E'}= \prodloc {\xi'}{k_{E'}} R_{\xi', E'}$$
défini comme le sous-anneau du produit $\prod_{\xi'} R_{\xi', E'}$ formé
des familles $(x_{\xi'})$ dont toutes les composantes se réduisent sur le
même élément dans le corps résiduel $k_{E'}$. Par ce qui précède, le morphisme $f$ se factorise par un morphisme 
$$\tilde{f}_{E'} : \oEp \otimes_{\oE} R^{\psi} 
(\vv_0, \ttt, \rhobar) \longrightarrow R_{\expl, E'}$$
qui n'est autre que le produit des $\tilde f_{\xi',E'}$.

Les morphismes $\tilde f_{E'}$ que nous venons de construire vérifient en 
outre des conditions de compatibilité. Pour les énoncer, considérons 
$E''$ une extension finie de $E'$. L'ensemble des $E'$-engeances est 
alors naturellement inclus dans celui des $E''$-engeances. De plus, si 
$\xi'$ est une $E'$-engeance, il est clair d'après les définitions que
$\oEpp \otimes_{\oEp} R_{\xi',E'} \simeq R_{\xi', E''}$. Nous pouvons
ainsi définir un morphisme de $\oEpp$-algèbres :
$$\psi_{E'',E'} : R_{\expl, E''} \to \oEpp \otimes_{\oEp} 
R_{\expl,E'}$$
en envoyant une famille $(x_{\xi''})$ sur la restriction de cette
famille aux $E'$-engeances. Ces morphismes font commuter les diagrammes
suivants :
\begin{equation}
\label{eq:diagcompat}
\xymatrix @C=50pt @R=5pt {
& R_{\expl,E''} \ar[dd]^-{\psi_{E'',E'}} \\
\oEpp \otimes_{\oE} R^\psi(\vv, \ttt, \rhobar)
\ar[ru]^-{\tilde f_{E''}} \ar[rd]_-{\tilde f_{E'}} \\
& \oEpp \otimes_{\oEp} R_{\expl,E'}.
}
\end{equation}

\paragraph{Interlude : calcul de l'image d'un morphisme de source 
$R^\psi(\vv_0, \ttt, \rhobar)$}
\label{par:interlude}

Nous mettons de côté momentanément les algèbres $R_{\expl, E'}$ et nous 
intéressons au problème suivant (qui apparaîtra à plusieurs reprises dans 
la suite) : 
\begin{center}
\begin{minipage}{13.5cm}
Soient $R$ une $\oE$-algèbre et $M_R$ un $\phi$-module\footnotemark\ sur 
$R \hat\otimes_W \ocE$ tous deux explicites. Nous supposons que $M_R$ 
correspond à un morphisme $f : R^\psi(\vv_0, \ttt, \rhobar) \to R$.

Comment déterminer l'image de $f$ ?
\end{minipage}
\end{center}

\footnotetext{Nous rappelons que $\phi = \varphi^f$.} 

Dans la suite, nous supposerons pour simplifier que la représentation 
$\rhobar$ est telle qu'il existe une base de son $\phi$-module sur $k_E((v))$ dans laquelle 
l'action de $\phi$ est donnée par la matrice 
$$\begin{pmatrix} 0 & \delta v^h \\ 1 & 0 \end{pmatrix}.$$

\noindent
Une première étape vers la résolution de la question ci-dessus consiste 
à remarquer que, d'après le lemme \ref{lem:fsurj}, l'application $f$ est 
surjective si et seulement si l'application tangente 
$\Hom_{\oE-\mathrm{alg}} (R, k_E[\varepsilon]) \to \Ext_{G_F}^1(\rhobar, 
\rhobar)$ est injective. De plus, la condition $p \geq 5$ 
assure, par le lemme \ref{lem:rinj}, que l'application restriction de $ 
\Ext_{G_F}^1(\rhobar, \rhobar)$ dans $ \Ext_{\Goo}^1(\rhobar, \rhobar)$ 
est injective. Ainsi, $f$ est surjective si et seulement si l'application 
$$f^\sharp : \Hom_{\oE-\mathrm{alg}} (R, k_E[\varepsilon]) \to 
\Ext_{\Goo}^1 (\rhobar , \rhobar)$$
est injective.

Or le fait que $R$ et $M_R$ soient explicites permet de
décrire entièrement $f^\sharp$. Expliquons à présent comment procéder
concrètement pour y parvenir.
L'algèbre $R$ étant explicite, nous pouvons déterminer une base $(f_1, 
\ldots, f_n)$ du $k_E$-espace vectoriel $\Hom_{\oE-\mathrm{alg}}(R, k_E[\varepsilon])$. 
Pour chaque $f_i$, considérons le 
$\phi$-module $M_i$ sur $k_E[\varepsilon]((v))$ déduit de $M_R$ par 
extension des scalaires \emph{via} $f_i$. Étant donné que $f_i$ correspond 
à une déformation de $\rhobar$, il existe une base de $M_i$ dans laquelle
la matrice de $\phi$ prend la forme :
$$\begin{pmatrix} 0 & \delta v^h \\ 1 & 0 \end{pmatrix} + \varepsilon A_i,$$
avec $A_i$ dans $\Mat_2(k_E((v)))$.
Avec ces conventions, l'image de $f_i$ dans $\Ext_{\Goo}^1 (\rhobar, 
\rhobar)$ s'identifie à celle de $A_i$ par l'application de la proposition 
\ref{propdecompositionext1}, que nous pouvons calculer de façon entièrement
explicite à l'aide du lemme \ref{lemdecompositionext1} et de la remarque \ref{rem:imageExt1}.

Maintenant que nous avons déterminé l'application $f^\sharp$, nous pouvons
vérifier si, oui ou non, elle est injective. Si elle l'est, le morphisme
$f$ est surjectif et son image est déterminée : c'est $R$ tout entier.
Si, au contraire, elle ne l'est pas, notons $N$ son noyau. Pour tout 
$\alpha$ dans $N$, le morphisme composé $\alpha \circ f$ prend alors ses
valeurs dans $k_E \subset k_E[\varepsilon]$. Ceci signifie que l'image
de $f$ est incluse dans la sous-$\oE$-algèbre de $R$ :
$$R_1 = \bigcap_{\alpha \in N} \alpha^{-1}(k_E).$$
En outre, une lecture attentive de la démonstration du lemme \ref{lemdecompositionext1}
fournit une nouvelle base de $M_R$ dans laquelle la matrice de $\phi$
est à coefficients dans $R_1$.

Nous nous ramenons, de cette manière, au même problème que précédemment où 
$R$ est remplacée par $R_1$. En appliquant à nouveau la même méthode, nous 
pouvons déterminer si l'application $f : R^\psi(\vv_0, \ttt, \rhobar) \to 
R_1$ est surjective. Si elle l'est, nous avons déterminé son image tandis 
que si elle ne l'est pas, nous pouvons exhiber une sous-algèbre explicite 
$R_2$ de $R_1$ dans laquelle $f$ prend ses valeurs et poursuivre ainsi
notre recherche...

Malheureusement, le processus que nous venons de décrire peut ne pas 
prendre fin après un nombre fini d'étapes ; nous entendons par là qu'il 
est possible qu'aucun des $R_n$ (pour $n$  dans $\N$) ne soit l'image de $f$. 
Usuellement, ce problème peut être résolu en définissant $R_\omega$ comme 
l'intersection de tous les $R_n$ puis en continuant la construction sur 
les ordinaux. Toutefois, nous ne pouvons alors plus vraiment parler de \og 
construction explicite \fg. En pratique, nous retiendrons donc que la 
méthode présentée ci-dessus est efficace si $f$ est \og presque \fg\ 
surjective mais qu'elle ne l'est plus dès lors que l'image de $f$ est très 
petite dans $R$.

\paragraph{Troisième étape : recollement des $R_{\expl,E'}$} \label{par:recollement}

Nous revenons à présent à la situation exposée à la fin de la seconde 
étape (\S \ref{par:etape2}). Rappelons que nous avons construit des 
algèbres $R_{\expl,E'}$ (où $E'$ désigne une extension finie variable de 
$E$) munies de morphismes ${\tilde f_{E'}} : \oEp \otimes_{\oE} R^\psi(\vv_0, \ttt, 
\rhobar) \to R_{\expl,E'}$ vérifiant des conditions de compatibilité \eqref{eq:diagcompat}. 
Le 
but de cette troisième et dernière étape --- qui est, de loin, la plus 
délicate --- est de mettre ensemble toutes les $R_{\expl,E'}$ afin de 
déterminer $R^\psi(\vv_0, \ttt, \rhobar)$.

Pour ce faire, notre méthode consiste à exhiber une $\oE$-algèbre locale 
noetherienne complète $R_\expl$ munie d'un morphisme de $\oE$-algèbres $\tilde g : 
R^\psi(\vv, \ttt, \rhobar) \to R_\expl$ et, pour toute extension $E'$ de 
$E$, d'un morphisme de $\oEp$-algèbres $\oEp \otimes_{\oE} R_\expl \to 
R_{\expl, E'}$ de sorte que les diagrammes suivants commutent :
$$\xymatrix @C=50pt @R=5pt {
& & R_{\expl,E''} \ar[dd]^-{\psi_{E'',E'}} \\
\oEpp \otimes_{\oE} R^\psi(\vv, \ttt, \rhobar) 
\ar[r]^-{\text{id} \otimes \tilde g} 
\ar@/^12pt/[rru]^-{\tilde f_{E''}} \ar@/_10pt/[rrd]_-{\tilde f_{E'}} 
& \oEpp \otimes_{\oE} R_\expl \ar[ru] \ar[rd] \\
& & \oEpp \otimes_{\oEp} R_{\expl,E'}
}$$
pour $E'$ et $E''$ deux extensions finies de $E$ avec $E' \subset E''$.

\begin{rem}
Il existe, à vrai dire, une formule qui fournit une algèbre $R_\expl$ 
convenable, à savoir :
\begin{equation}
\label{eq:Rexpl}
R_\expl = \varprojlim_{E'} \: {\tilde f_{E'}}(R^\psi(\vv_0, \ttt, \rhobar))
\end{equation}
où la limite projective est étendue à toutes les extensions finies 
$E'$ de $E$ et où les morphismes de transition sont induits par les
$\psi_{E'',E'}$.
Toutefois, il est généralement difficile d'appliquer la méthode du \S 
\ref{par:interlude} pour déterminer ${\tilde f_{E'}}(R^\psi(\vv_0, \ttt, 
\rhobar))$ car le morphisme ${\tilde f_{E'}}$ est souvent loin d'être surjectif. 
En conséquence, la formule \eqref{eq:Rexpl} est rarement exploitable.
\end{rem}

\begin{rem}
\label{rem:dimnulle}
Il existe toutefois un cas où déterminer un anneau $R_\expl$ convenable 
est simple ; c'est celui où la variété de Kisin 
$\GG\RR_{\rhobar,\psi,\vv_0,\ttt}$ est de dimension nulle. En effet, 
sous cette hypothèse additionnelle, les morphismes $\psi_{E'',E'}$ sont 
des isomorphismes dès que $E'$ et $E''$ sont assez grands. Nous en 
déduisons que, quitte à agrandir $E$, nous pouvons choisir dans ce cas 
$R_\expl = R_{\expl,E}$.
\end{rem}

\begin{lem}
\label{lem:tildeginj}
Avec les notations précédentes, le morphisme
$\tilde g : R^\psi(\vv_0, \ttt, \rhobar) \to R_\expl$ est injectif.
\end{lem}

\begin{proof}
Il suit de la définition de $R^\psi(\vv_0, \ttt, \rhobar)$ que, pour élément 
$x$ non nul dans  $R^\psi(\vv_0, \ttt, \rhobar)$, il existe un morphisme de 
$\oE$-algèbres $\alpha : R^\psi(\vv_0, \ttt, \rhobar) \to \Zpbar$ tel que 
$\alpha(x)$ est non nul. Pour démontrer que $\tilde g$ est injectif, il suffit 
donc de montrer que tout morphisme $\alpha$ comme ci-dessus se factorise 
par $\tilde g$.

Soit donc $\alpha : R^\psi(\vv_0, \ttt, \rhobar) \to \Zpbar$ un morphisme 
de $\oE$-algèbres. Remarquons, d'une part, qu'il prend ses valeurs dans 
l'anneau des entiers $\oEp$ d'une extension finie $E'$ de $E$ et, d'autre 
part, que, par définition, la $E'$-représentation de $G_F$ qu'il définit, 
est de type $(\vv_0, \ttt, \psi)$. D'après la proposition \ref{propstroumphe}, le théorème de classification des
modules de Kisin (\emph{cf} Proposition \ref{propecriturejolie}) et la 
construction de $R_{\expl, E'}$, il existe un morphisme de $\oEp$-algèbres $\beta : R_{\expl, 
E'} \to \oEp$ qui fait commuter le diagramme suivant :
$$\xymatrix @C=35pt {
\oEp \otimes_{\oE} R^\psi(\vv_0, \ttt, \rhobar) 
\ar[r]^-{\Id \otimes \tilde g} \ar[d]_-{\alpha}
& \oEp \otimes_{\oE} R_\expl \ar[r]
& R_{\expl,E'} \ar[lld]^-{\beta} \\
\oEp
}$$
Il en résulte directement que $\alpha$ se factorise par $\tilde g$ comme
nous le souhaitions.
\end{proof}

D'après le lemme \ref{lem:tildeginj} ci-dessus, l'anneau de déformations 
$R^\psi(\vv_0, \ttt, \rhobar)$ que nous cherchons à déterminer s'identifie 
à l'image du morphisme $\tilde g$. Or, au moins si nous avons été adroits
dans le choix de $R_\expl$ afin que $\tilde g$ soit \og presque \fg\ 
surjectif, nous pouvons déterminer l'image de $\tilde g$ à l'aide de la 
méthode expliquée au \S \ref{par:interlude}.

\section{Mise en \oe uvre de la méthode en degré $2$}
\label{sec:degre2}

Nous  mettons en \oe uvre dans cette partie la méthode du paragraphe \S \ref{ssec:calculdeform} dans le cas où $\rhobar$ est une représentation irréductible de $F = \Q_{p^2}$.
De telles représentations $\rhobar$ sont de la forme :
$$\rhobar \simeq \Ind^{G_{F}}_{G_{F'}}\left(
\omega_4^{1+r_0+p(1+r_1)} \cdot \nr' (\theta) \right)\otimes\omega_2^s$$
où $G_{F'}$ est le groupe de Galois absolu de l'extension non ramifiée $F'$ de degré $2$  de $F$, $\omega_4$ (resp. $\omega_2$) est le caractère fondamental de niveau $4$ (resp. $2$) associ\'e au plongement  fix\'e $\tau'_0$ de $F'$ dans $E$ (resp. $\tau_0 = \tau'_{0| F}$ de $F$ dans $E$), $s$ est dans $\Z$, $r_0$ est un entier entre $0$ et $p-1$ et  $r_1$ est un entier entre $-1$ et $p-2$. 

Les cas génériques, \emph{i.e.} pour lesquels $1\leq r_0\leq p-2$ et $0\leq r_1\leq p-3$,
 ont déjà été traités dans \cite{BM2}, Théorème 5.2.1.
Les auteurs obtiennent une description explicite des anneaux $R^\psi(\vv_0,\ttt,\rhobar)$, pour les types galoisiens $\ttt$ comme dans la partie \ref{sec:methode} :
$$
\begin{array}{rcll}
R^\psi(\vv_0,\ttt,\rhobar) & \simeq & \{0\} & \text{si } \D(\ttt) \cap \D(\rhobar) = \emptyset, \smallskip \\
 & \simeq & \oE[[ X , Y , T ]] /  (XY +p) & \text{sinon. } \\
\end{array}
$$

Nous traitons donc dans cette partie les cas restants.
Nous renvoyons à la partie \ref{sssec:poidsrhobar} pour la forme des représentations non génériques et la liste de leurs poids de Serre.

Dans la suite du texte, nous traitons le cas des repr\'esentations pour lesquelles l'entier $s$ est nul ;  le cas g\'en\'eral s'en d\'eduit par torsion par un caract\`ere de $G_F$.
Rappelons également que nous avons fixé le caractère $\psi$ d'une manière compatible avec le type galoisien $\ttt$ (relations  \eqref{relpsi} et \eqref{eqn:psidelta}) en choisissant un élément $\delta$ dans $\oE^\times$.
La compatibilité entre $\psi$ et le déterminant de $\rhobar$ impose de plus que $\delta$ se réduise sur $\theta^{-1}$ dans $k_E$.

Dans toute la fin du texte, $\ttt$ désigne un type galoisien choisi comme au début de la partie~\ref{sec:methode} : $\ttt = (\eta \oplus \eta')_{| I_F}$, où $\eta$ et $\eta'$ sont deux caractères distincts de $\Gal (L / F)$ dans $\oE^\times$.

Dans \cite{GK} (théorème A), Gee et Kisin démontrent l'existence de multiplicités intrinsèques satisfaisant la version cristalline de la conjecture de Breuil--Mézard pour le type $\vv_0$ et tout type galoisien $\ttt$.
Nos calculs permettent de déterminer presque toutes les valeurs de ces multiplicités intrinsèques, confirmant ainsi une conjecture de Kisin (\cite{Ki3}, conjecture 2.3.2).

\subsection{Liste des engeances de type $\ttt$ de $\rhobar$}

Un type galoisien $\ttt$ \'etant fix\'e,
la première étape  de  la stratégie expliquée au \S \ref{sssec:methode} pour le calcul de l'anneau $R^{\psi}(\vv_0, \ttt , \rhobar)$
consiste à établir la liste des engeances de type $\ttt$ de la 
représentation $\rhobar$.

Pour cela, nous listons les genres un par un et, 
pour chacun d'eux, en utilisant la méthode du \S \ref{sssec:represid}, nous 
calculons en fonction des paramètres $(\alpha, \alpha ', (a_i , a'_i)_{i = 0,\ldots, f-1}) $ la représentation résiduelle 
associée. Nous ne conservons enfin que les engeances pour lesquelles la 
représentation résiduelle correspondante est irréductible et isomorphe à $\rhobar$.

Il convient de remarquer que, pour $f$ égal à $2$, les mauvais genres exclus par la définition \ref{def:mauvaisgenre} donnent toujours lieu à des représentations résiduelles réductibles.
Cela se vérifie en appliquant  les techniques de calcul de la représentation résiduelle de la partie \ref{sssec:represid}.
Pour les bons genres, le calcul est facile et ne présente pas d'intérêt majeur. Afin de ne pas
allonger démesurément la taille de l'article, nous ne le présentons pas ici
mais regroupons simplement les résultats que nous avons obtenus dans le
tableau de la figure \ref{fig:engeances} (page \pageref{fig:engeances}) :
pour chaque représentation $\rhobar$, nous indiquons les types galoisiens $\ttt$ pour lesquels la variété de Kisin $\GG\RR_{\rhobar, \psi, \vv_0, \ttt}$ est non vide et les engeances qui la décrivent
(les paramètres $a_i$, $a'_i$, $\alpha$ et $\alpha'$ vivent
dans $k_E$).
Nous renvoyons également au \S 3.2 de \cite{Da}
pour la présentation d'un calcul analogue.

\newlength\epaisLigne
\newcommand\gline{\noalign{\global\epaisLigne\arrayrulewidth\global\arrayrulewidth 5pt}
\hline \noalign{\global\arrayrulewidth\epaisLigne}}

\newcommand\mline{\noalign{\global\epaisLigne\arrayrulewidth\global\arrayrulewidth 2.5pt}
\hline \noalign{\global\arrayrulewidth\epaisLigne}}

\newcommand\sline{\noalign{\global\epaisLigne\arrayrulewidth\global\arrayrulewidth 0.1pt}
\hline \noalign{\global\arrayrulewidth\epaisLigne}}

\renewcommand{\arraystretch}{1.5}

\begin{figure}
\begin{center}
\scriptsize
\begin{tabular}{| c |  c | c | c | c | c | c | c | }
\mline
\multicolumn{2}{ |c |}{\multirow{2}{*}{Représentation $\rhobar = \Ind_{G_{F'}}^{G_F} ( \chi \cdot \nr'(\theta))$}} &\multicolumn{3}{c|}{Type $\ttt$}  & \multicolumn{3}{c|}{Engeances }\\
\cline{3-8}
\multicolumn{2}{ |c |}{} & $\eta$ & $\eta'$ & $d$ & Genre & $(\alpha , \alpha')$ & $(a_0 , a_1) $\\
\mline
\multirow{5}{*}{$\chi = \omega_4^{ 1 + r_0}$} 
& $0 \leq r_0 \leq p - 2 $ & $ \omega_2^{r_0} $ & $ \omega_2^{-p}$  & $( p - 2 , p - 1 - r_0)$ & $\I_\eta \times \II  $ &  $(1,\theta^{-1})$ & $(0,0)$ \\
\cline{2-8}
& $1 \leq r_0 \leq p-2$ & $ \omega_2^{r_0 - p} $ & $ \mathds{1}$ & $(  0 , p - r_0)$ & $ \II \times \I_{\eta'} $ & $(1,\theta^{-1})$ &  $(0,0)$ \\
\cline{2-8}
& \multirow{3}{*}{\color{bleu} $0 \leq r_0 \leq p-3$} 
& \multirow{3}{*}{\color{bleu} $ \omega_2^{1 + r_0 - p} $} 
& \multirow{3}{*}{\color{bleu} $ \omega_2^{-1}$} 
& \multirow{3}{*}{\color{bleu} $(0 , p-2 - r_0)$} 
  & {\color{bleu} $\I_{\eta'} \times \II $} 
  & {\color{bleu} $(1 , \theta^{-1})$}
  & {\color{bleu} $(0, 0)$} \\
\cline{6-8}
& & & & 
  & {\color{bleu} $\II \times \I_{\eta}$ }
  & {\color{bleu} $(1 , \theta^{-1})$}
  & {\color{bleu} $(0, 0)$ }\\
\cline{6-8}
&  & & & 
  & {\color{bleu} $\I_{\eta'} \times \I_{\eta}$}
  & {\color{bleu} $(\alpha , \frac {-1} {\theta \alpha} )$} %
  & {\color{bleu} $( 1 , \theta \alpha^2)$} \\
\mline
\multirow{5}{*}{$\chi = \omega_4^{ p(2 + r_1)}$} & $- 1 \leq r_1 \leq p-3 $ & $ \omega_2^{p(1 + r_1)} $ &$ \omega_2^{-1}$ & $( p-2-r_1 , p-2)$  &  $\II  \times \I_\eta $ & $(1,\theta^{-1})$ & $(0,0)$   \\
\cline{2-8}
& $0 \leq r_1 \leq p-3$ & $ \omega_2^{ - 1 + p (1 + r_1) } $ & $\mathds{1}$ & $( p - 1 - r_1 , 0 )$ &  $\I_{\eta'} \times \II  $ & $(1,\theta^{-1})$  &  $(0,0)$  \\
 \cline{2-8}
& \multirow{3}{*}{\color{bleu} $- 1 \leq r_1 \leq p-4 $} 
& \multirow{3}{*}{\color{bleu} $\omega_2^{- 1 + p(r_1 + 2) }$} 
& \multirow{3}{*}{\color{bleu} $\omega_2^{-p}$} 
& \multirow{3}{*}{\color{bleu} $(p-3 - r_1 , 0)$} 
  & {\color{bleu} $\I_{\eta} \times \II$}
  & {\color{bleu} $(1,\theta^{-1})$}
  & {\color{bleu} $(0,0)$}  \\
\cline{6-8}
& & & & 
  & {\color{bleu} $\II  \times \I_{\eta'}$}
  & {\color{bleu} $(1,\theta^{-1})$}
  & {\color{bleu} $(0,0)$} \\
\cline{6-8}
& & & & 
  & {\color{bleu} $\I_{\eta} \times \I_{\eta'}$}
  & {\color{bleu} $(\alpha , \frac {-1} {\theta \alpha} )$} %
  & {\color{bleu} $(1 , -1)$} \\
\mline
\end{tabular}
\end{center}
\caption{Liste des engeances de type $\ttt$ de la représentation $\rhobar$}
\label{fig:engeances}
\end{figure}

\renewcommand{\arraystretch}{1}

Nous remarquons que, pour chaque couple $(\ttt, \rhobar)$ correspondant aux 
lignes du tableau qui apparaissent en noir, il n'y a qu'une seule engeance. 
La variété de Kisin correspondante est donc réduite à un point. Au 
contraire, dans les autres situations, celles qui apparaissent en bleu dans 
le tableau, nous sommes en présence de trois possibilités pour le genre, 
les deux premiers conduisant à une unique engeance et la troisième 
conduisant à une famille d'engeances paramétrée par $k_E^\times$. Il se 
trouve que, dans ces cas, la variété de Kisin est isomorphe à $\P^1_{k_E}$ 
: les engeances paramétrées par le troisième genre correspondent 
naturellement à l'ouvert Zariski $k_E^\times$ dans  $\P^1(k_E)$ tandis que 
les deux autres genres correspondent respectivement aux points $0$ et 
$\infty$ de $\P^1(k_E)$.

\subsection{Détermination de l'anneau de déformations}\label{ssec:recollement}

Nous continuons d'appliquer la méthode du \S \ref{sssec:methode}, sachant 
que nous arrivons à la deuxième étape. Dans la suite, nous traitons 
uniquement le cas où 
$\rhobar= \Ind_{G_{F'}}^{G_F}\! \big(\omega_4^{ 1 + 
r_0} \cdot \nr' (\theta )\big)$, avec $r_0$ entre $0$ et $p-2$, l'autre étant absolument identique.

\subsubsection{Cas d'une variété de Kisin de dimension nulle}

Supposons que la variété de Kisin $\GG\RR_{\rhobar,\psi,\vv_0,\ttt}$ soit 
réduite à un point (c'est-à-dire que nous sommes dans une ligne noire du 
tableau de la figure \ref{fig:engeances}). Notons $\xi$ l'unique engeance 
correspondante et appelons $R_\xi$ l'anneau qui lui est associé par la
construction du \S \ref{par:etape2} : nous avons
$$R_\xi \simeq \oE[[T,X,Y]]/(X Y + p).$$
La troisième étape de notre stratégie consiste à exhiber une $\oE$-algèbre
$R_\expl$ vérifiant un certain nombre de propriétés (\S \ref{par:recollement}). Toutefois, nous sommes
ici dans le cas simple où la variété de Kisin est de dimension nulle.
D'après la remarque \ref{rem:dimnulle}, nous pouvons donc simplement 
choisir $R_\expl = R_\xi$. Le lemme \ref{lem:tildeginj} entraîne alors
que nous avons une injection
$\tilde{g} : R^\psi(\vv_0,\ttt,\rhobar)\to R_\xi$.

\begin{prop} 
\label{prop:uniengeance}
Nous conservons les notations et les hypothèses précédentes et supposons de plus que la repr\'esentation $\rhobar$ n'est pas totalement non générique.
Alors
le morphisme $\tilde g$ est un isomorphisme.
\end{prop}

\begin{proof}
Nous traitons d'abord le cas du type galoisien $\omega_2^{r_0} \oplus 
\omega_2^{-p} $, correspondant à $d = (p-2, p-1 - r_0) $, avec $r_0$ 
entre $1$ et $p-2$ (l'hypothèse sur $\rhobar$ revient à exclure le cas $r_0 = 0$).
La matrice du $\phi$-module $\bbM_\xi$ sur $ R_\xi   \otimes_W \ocE$ 
associé est, d'après la partie \ref{sssec:represid} :
$$B = 
\begin{pmatrix}
 v^{p(p-1)} \left(X_1 v^{d_1} (v^p + p) + T_0 v^p \right) &  v^{p^2 + p} \\
\alpha' v^{p(p-1)} \left( v^{d_1} (v^p + p) + T_0Y_1v^{p-1}\right) & \alpha' Y_1 v^{p^2 + p -1} \\
\end{pmatrix}.
$$
Après un changement de base correspondant à multiplier le deuxième 
vecteur de base par $\alpha' v^{d_1 + p^2}$, nous obtenons la nouvelle 
matrice
$$
B' = 
\begin{pmatrix}
 v^{p(p-1)} \left(X_1 v^{d_1} (v^p + p) + T_0 v^p \right) & \delta v^{p^2(d_1 + 1) + p + p^4} \\
v^{-d_1 - p} \left( v^{d_1} (v^p + p) + T_0Y_1v^{p-1}\right) & \delta Y_1 v^{p+p^4 -1 + (p^2 - 1) d_1} \\
\end{pmatrix}
$$
qui, réduite dans $k_E[[v]]$, donne
$
\begin{pmatrix}
0  & \overline{\delta} v^h \\
1 & 0 \\
\end{pmatrix}
$
avec $h = p^4 + p + p^2(d_1 + 1)$.

Pour une variable $Z$ dans $\left\{ T_0, X_1, Y_1 \right\}$, désignons 
par $Z^\star$ le morphisme de $\oE$-algèbres de $R_\xi$ dans les nombres 
duaux $k_E[ \varepsilon]$ qui envoie $Z$ sur $\varepsilon$ et les autres 
variables sur $0$.
D'après la partie \ref{sssec:calculExt1}, les images de $T_0^\star$, 
$X_1^\star$ et $Y_1^\star$ dans $\Ext^1_{\Goo} (\rhobar, \rhobar)$ 
(décrit par la proposition \ref{propdecompositionext1}) sont 
respectivement 
$(\theta v^{-p-p^2(d_1 + 1)} , 0 ) $, $(\theta v^{-p-p^2} , 0 ) $ et $(v^{-p^2 - d_1 - 1} , 0 ) $.
 Par le 
lemme \ref{lemdecompositionext1}, ces trois vecteurs forment une famille 
libre. L'application $\tilde{g}$ est donc surjective, ce qui démontre que 
l'anneau $R^\psi(\vv_0,\ttt,\rhobar)$ est isomorphe à $\frac{\oE [[ T_0, 
X_1, Y_1 ]] }{X_1 Y_1 + p}$.
 
Nous traitons maintenant le cas du type galoisien $\omega_2^{r_0 - p} 
\oplus \mathds{1}$, avec $d = (0, p-r_0)$ et $r_0$ entre $1$ et $p-2$. 
L'anneau explicite est dans ce cas $\frac{\oE [[ X_0, Y_0, T'_1 ]] }{X_0 
Y_0 + p}$ et la matrice du $\phi$-module associé peut être mise sous la 
forme
$$
\begin{pmatrix}
 X_0 v^{d_1} + T'_1v^p & \delta v^{p^3 + p^2} (v^{d_1} + Y_0 T'_1)\\
 v^{-1}(v + p) & \delta Y_0 v^{(p+1)(p^2  - 1)}(v+p) \\
\end{pmatrix}.
$$
Les images de $X_0^\star$, $Y_0^\star$ et $T_1^{'\star}$ dans 
$\Ext^1_{\Goo} (\rhobar, \rhobar)$ sont respectivement 
$(\theta v^{-1-p}, 0)$, $(v^{-d_1 - p}, 0) $ et $(\theta v^{-d_1-p^2}, 0)$,
 qui forment une famille libre. Nous
concluons enfin de la même manière que précédemment.
\end{proof}

\begin{rem} 
Le résultat de la proposition \ref{prop:uniengeance} n'était connu
auparavant que pour les représentations génériques \cite{BM2}.
\end{rem}

\begin{rem}
\label{rem:totnongen}
L'hypothèse \og $\rhobar$ non totalement non générique \fg\ de  la proposition \ref{prop:uniengeance} n'exclut qu'un seul cas dans les deux premières lignes du tableau :
$r_0 = 0$ et $\ttt = \omega_2^{r_0} \oplus \omega_2^{-p} $, soit $d = (p-2 , p-1)$.
Dans ce cas, le début de la démonstration de la proposition reste valable, mais nous observons que les images de $X_1^\star$ et $Y_1^\star$ dans $\Ext^1_{\Goo} (\rhobar, \rhobar)$ 
sont toutes les deux égales à $(\delta^{-1}v^{-p-p^2} , 0 ) $.
L'injection $\tilde g$ n'est donc pas surjective.
\end{rem}

\subsubsection{Cas d'une variété de Kisin de dimension supérieure}
\label{sssec:recoll}

Nous traitons, dans ce paragraphe, le cas de la représentation
$\Ind_{G_{F'}}^{G_F} \big(\omega_4^{1 + r_0} \cdot \nr' (\theta) \big)$,
avec  $r_0$ dans $\{0, \ldots, p-3\}$, et du type galoisien 
donné par $\eta = \omega_2^{1 + r_0 - p}$ et $\eta' = \omega_2^{-1}$ ; $d$ est alors $(0, p-2-r_0) $, avec $p-2-r_0$ dans $\{1, \ldots, p-2\}$. 

Pour ces données, les $E'$-engeances sont en bijection naturelle avec 
$\P^1 (k_{E'})$ pour toute extension $E'$ de $E$.
En suivant à nouveau la stratégie expliquée dans la partie \ref{sssec:methode}, 
nous considérons une extension $E'$ de $E$ et nous associons à chaque 
$E'$-engeance $\xi'$, la $\oEp$-algèbre explicite $R_{\xi'} $ ainsi qu'un 
module de Breuil--Kisin $\MK_{\xi'}$ sur $R_{\xi'} \hat \otimes_{\Zp} \SK$ 
définis dans le \S \ref{sssec:methode}. Pour tout $\xi'$, posons
$$M_{\xi'} = H^0\big(\!\Gal(L/F), \ocE \otimes_\SK \MK_{\xi'}^{(0)}\big).$$
D'après les résultats de la partie \ref{ssec:descente}, l'espace $M_{\xi'}$ 
muni du Frobenius $\phi = \varphi^2$ s'identifie canoniquement au
$\phi$-module associé à la restriction à $\Goo$ de la représentation 
de $G_F$ correspondant à $\MK_{\xi'}$. De plus, la proposition 
\ref{prop:matrices} et la formule \eqref{eq:prodB} fournissent un 
moyen explicite pour calculer la matrice de $\phi$ agissant sur $M_{\xi'}$.

Dans le cas qui nous intéresse, nous trouvons, après calcul, les résultats
regroupés dans le tableau de la figure \ref{fig:matrices} (page
\pageref{fig:matrices}).
\begin{figure}
\begin{center}
\scriptsize
\begin{tabular}{|c|c|l|c|c|}
\mline
Engeance & Anneau & ~\hfill Paramètres \hfill~ & Matrice $B_{\xi'}$ de $\phi$ sur
$M_{\xi'}$
\\
\mline 
\begin{tabular}{c}
$\I_{\eta'} \times \I_\eta$ \\
$\xi' \in k_{E'}^\times$ 
\end{tabular} &
$\oEp[[ T_{\xi'} , Z_{\xi'}]]$ &
\hspace{-0.5em}\renewcommand{\arraystretch}{1}
\begin{tabular}{l}
$\alpha = [{\xi'}] + Z_{\xi'}$ \\
$\alpha' = - \delta\alpha^{-1}$ \\
$a'_0 = 1$ \\
$a_1 = [\theta{\xi'}^2] + T_{\xi'}$
\end{tabular} 
\hspace{-1em}\renewcommand{\arraystretch}{1.2}
&
$\begin{pmatrix}
\alpha v^{d_1+1} (v+p) & \alpha a'_0 v^{p^2 + d_1+1} (v+p) \\
\alpha' a_1 v^{d_1+1} & \alpha' a'_0 a_1 v^{p^2 + d_1+1} + \alpha' v^{p^2} (v^p + p)
\end{pmatrix}$
 \\
\hline
\begin{tabular}{c}
$\I_{\eta'} \times \II$ \\
$\xi' = 0$
\end{tabular} &
$\frac{\oEp[[T'_0, X_0, Y_0]]}{X_0 Y_0 + p}$ &
\hspace{-0.5em}\renewcommand{\arraystretch}{1}
\begin{tabular}{l}
$\alpha = 1$, 
$\alpha' = \delta$ \\
$a'_0 = T'_0$ \\
$a_1 = X_0$ \\ $a'_1 = Y_0$
\end{tabular} &
\hspace{-1em}\renewcommand{\arraystretch}{1.2}
$\begin{pmatrix}
\alpha a_1 v^{d_1+1} & \alpha a'_0 a_1 v^{p^2 + d_1+1} + \alpha v^{p^2+1} (v^p + p) \\
\alpha' v^{d_1+1} & \alpha' a'_0 v^{p^2 + d_1+1} + \alpha' a'_1 v^{p^2} (v^p + p)
\end{pmatrix}$
\\
\hline
\begin{tabular}{c}
$\II \times \Ieta$ \\
$\xi' = \infty$
\end{tabular} &
$\frac{\oEp[[X_\infty, Y_\infty, T_\infty]]}{X_\infty Y_\infty + p}$ &
\hspace{-0.5em}\renewcommand{\arraystretch}{1}
\begin{tabular}{l}
$\alpha = 1$, 
$\alpha' = \delta$ \\
$a_0 = X_\infty$ \\ $a'_0 = Y_\infty$ \\
$a_1 = T_\infty$
\end{tabular} &
\hspace{-1em}\renewcommand{\arraystretch}{1.2}
$\begin{pmatrix}
\alpha a_0 v^{d_1+1} (v+p) & \alpha v^{p^2+d_1+1} (v+p) \\
\alpha' a_0 a_1 v^{d_1+1} + \alpha' v^{p} & \alpha' a_1 v^{p^2 + d_1+1} + \alpha' a'_0 v^{p^2}
\end{pmatrix}$
\\
\mline
\end{tabular}
\end{center}
\caption{Matrices du Frobenius sur les $\phi$-modules $M_{\xi'}$}
\label{fig:matrices}
\end{figure}
Au niveau des anneaux de déformations, nous obtenons un morphisme de 
$\oEp$-algèbres $f_{E'}$ de $\oEp \otimes_{\oE} R^\psi(\rhobar)$ dans 
l'anneau explicite
$$R_{\expl,E'}= \prodloc {\xi'}{k_{E'}} R_{\xi'}$$
où la notation ci-dessus désigne le \og produit local \fg.
D'après le lemme \ref{lem:factorf}, ce morphisme se factorise en un morphisme 
$\tilde f_{E'} : \oEp \otimes_{\oE} R^\psi(\vv_0, \ttt, \rhobar) \to 
R_{\expl,E'}$. De plus, si $E''$ désigne une extension finie de $E'$, nous 
avons un morphisme $\psi_{E'', E'} : R_{\expl,E''} \to \oEpp \otimes_{\oEp} 
R_{\expl,E'}$ rendant commutatif le diagramme suivant :
$$\xymatrix @C=50pt @R=5pt {
& R_{\expl, E''} \ar[dd]^-{\psi_{E'',E'}} \\
\oEpp \otimes_{\oE} R^\psi(\vv_0, \ttt, \rhobar) 
\ar[ru]^-{\tilde f_{E''}} \ar[rd]_-{\tilde f_{E'}} \\
& \oEpp \otimes_{\oEp} R_{\expl, E'}.
}$$
Notons $V_{E'}$ la $R_{\expl, E'}$-représentation de $G_F$ correspondant à 
$f_{E'}$. L'action de $\phi$ sur le $\phi$-module $M_{E'}$ associé à 
$(V_{E'})_{|\Goo}$ s'obtient en mettant ensemble les matrices qui 
apparaissent dans le tableau ci-dessus.

Suivant la stratégie de détermination de l'espace de déformations expliquée 
dans le \S \ref{ssec:calculdeform}, nous devons maintenant exhiber une 
$\oE$-algèbre explicite $R_\expl$ telle que, pour tout $E'$ comme 
précédemment, le morphisme $\tilde f_{E'}$ se factorise par $\oEp 
\otimes_{\oE} R_\expl$.
Pour construire $R_\expl$, nous cherchons pour chaque engeance $\xi' \in 
\P^1(k_{E'})$ une nouvelle base de $M_{\xi'}$, de façon à ce que les 
matrices des $\phi$-modules exprimées dans ces nouvelles bases \og se 
ressemblent \fg\ davantage. Une première tentative conduit à considérer 
les changements de base que voici :

\begin{center}
\small
\begin{tabular}{|c|c|c|}
\mline
$\xi'$ & Changement de base $P$ & Matrice $P^{-1} B_{\xi'} \phi(P)$  de $\phi$ dans la nouvelle base \\
\mline
$\xi' \in k_{E'}^\times$ &
$\begin{pmatrix}
-a'_0 & 0 \\ v^{-1} & \alpha' v^{-1}
\end{pmatrix}$ &
$v \cdot \begin{pmatrix}
0 & \delta v^{d_1} (v+p) \\
v^p + p & p \alpha'  + p \alpha v^{d_1} + (\alpha' a'_0 a_1 + \alpha) v^{d_1 + 1} + \alpha' v^p
\end{pmatrix}$ \\
\hline
$\xi' = 0$ &
$\begin{pmatrix}
-a'_0 & \alpha \\ v^{-1} & \alpha' a'_1 v^{-1}
\end{pmatrix}$ &
$v \cdot \begin{pmatrix}
0 & \delta v^{d_1} (v+p) \\
v^p + p & p \alpha' a'_1 + \alpha a_1 v^{d_1} + \alpha' a'_0 v^{d_1 + 1} + \alpha' a'_1 v ^p
\end{pmatrix}$ \\
\hline
$\xi' = \infty$ &
$ \begin{pmatrix}
1 & 0 \\ - a_0 v^{-1} & \alpha' v^{-1}
\end{pmatrix}$ &
$v \cdot \begin{pmatrix}
0 & \delta v^{d_1} (v+p) \\
v^p + p & \alpha' a'_0  + p \alpha a_0 v^{d_1} + (\alpha a_0 + \alpha' a_1) v^{d_1 + 1}
\end{pmatrix}$ \\
\mline
\end{tabular}
\end{center}

Nous observons que les matrices sont analogues à la différence près 
qu'aucun terme en $v^p$ n'apparaît sur le coefficient en bas à droite 
dans le cas $\xi' = \infty$. Le but du lemme suivant est d'unifier 
définitivement ces écritures.

\newcommand{\ocER}{R \hat\otimes_{\Zp} \ocE}
\newcommand{\NPoly}{\mathrm{NP}}

\begin{lem}\label{groscalcul} 
Soit $R$ une $\oEp$-algèbre locale noetherienne complète de corps 
résiduel $k_{E'}$ et soient $\delta$, $a$, $b$, $c$ et $d$ des éléments de 
$R$. Nous supposons que $p$ divise $a$, que $\delta$ est inversible dans 
$R$ et que $ab = -\delta p^2$.
Alors il existe un élément $c' \in R$ tel que les matrices
$$M=\begin{pmatrix}
0 & \delta v^{d_1}(v+p) \\ v^p + p & a+bv^{d_1}+c v^{d_1+1}+dv^p
\end{pmatrix}\quad \text{et} \quad 
M' = \begin{pmatrix}
0 & \delta v^{d_1}(v+p) \\ v^p + p & a+bv^{d_1}+c' v^{d_1+1}
\end{pmatrix}$$
définissent des $\phi$-modules isomorphes sur $\ocER$.

De plus, nous pouvons choisir $c'$ congru à $c$ modulo $p$ et 
dépendant de façon analytique de $\delta, a, b, c, d$.
\end{lem}

\begin{rem}
\label{rem:depanalytique}
La dépendance analytique énoncée dans le lemme \ref{groscalcul} signifie 
qu'il existe une série $\mathcal F(\Delta, A, B, C, D)$ à coefficients 
dans $R$ telle que $c' = \mathcal F(\delta, a, b, c, d)$.
\end{rem}

\begin{proof}
\newcommand{\RHS}{g}
Tout au long de la démonstration, nous posons :
$$A = a+bv^{d_1}+c v^{d_1+1}+dv^p, \, B = a+bv^{d_1}+c' v^{d_1+1}, \, 
S = v^p + p \, \text{ et }\, T = \delta v^{d_1}(v+p).$$
Les matrices $M$ et $M'$ définissent des $\phi$-modules sur $\ocER$ 
isomorphes entre eux si et seulement s'il existe une matrice
$$P=\begin{pmatrix}x&y\\ z&t\end{pmatrix}\in \GL_2(\ocER)$$
telle que $M'\phi(P)=PM$.
Si tel est le cas,  les éléments $x,y,z,t$ de $\ocER$ satisfont le système suivant
\begin{equation}\label{Systeme-bidouille}
\left\{\begin{array}{l} Sy=T\phi(z)\cr Tx+Ay=T\phi(t)\cr
St=S\phi(x)+B\phi(z)\cr
Tz+At=S\phi(y)+B\phi(t)\cr
\end{array}\right.
\end{equation}
Nous allons donc exhiber un élément $c'$ dans $R$, analytique en 
$\delta,a,b,c,d$ tel que les matrices $M$ et $M'$ soient équivalentes, 
\emph{i.e.} que le système \eqref{Systeme-bidouille} ait une solution. 
Nous cherchons $z$ sous la forme $z=(v+p^p)(\xi+z')$ avec $\xi$ dans $R$ et 
$z'$ dans $vR[[v]]$. Le système \eqref{Systeme-bidouille} se réécrit
\begin{equation}\label{systeme'-bidouille}
\left\{\begin{array}{l}
Sy=T(v^{p^2}+p^p)(\xi+\phi(z'))\cr
Tx+Ay=T\phi(t)\cr
St=S\phi(x)+B(v^{p^2}+p^p)(\xi +\phi(z'))\cr
T(v+p^p)(\xi+z')=S\phi(y)+B\phi(t)-At\cr
\end{array}\right.
\end{equation}
En posant
$U={v^{p^2}+p^p\over v^p+p}=v^{p(p-1)}-pv^{p(p-2)}+\cdots+(-p)^{p-1}$,
nous arrivons à :
\begin{eqnarray*}
y & = & TU(\xi+\phi(z')) \\
 \phi(t) - x & = & AU(\xi +\phi(z')) \\
t - \phi(x) & = & BU(\xi+\phi(z'))
\end{eqnarray*}
En éliminant $x$, nous trouvons :
$$(\Id-\phi^2)(t) = BU(\xi+\phi(z')) - \phi(AU)(\xi+\phi^2(z')).$$
Nous en déduisons une expression en $t$ en fonction 
de $z'$ :
$$t=\tau+\xi\sum_{i=0}^\infty 
\phi^{2i}(BU-\phi(AU))+\sum_{i=0}^\infty\phi^{2i}(BU\phi(z')-\phi(AU)\phi^2(z')).$$
pour un certain élément $\tau$ de $R$. Nous allons démontrer que, pour 
$\tau=1$, le système \eqref{systeme'-bidouille} a une solution. En 
remplaçant $y$ et $t$ par leurs valeurs dans la dernière équation de 
\eqref{systeme'-bidouille}, ceci revient à montrer qu'il existe $c'$,
$\xi$ et $z'$ solutions de 
\begin{equation}
\label{equation-bidouille}
T(v+p^p)(\xi+z') = \RHS(c', \xi, z')
\end{equation}
avec :
{\small
$$ \begin{array}{l}
\RHS(c', \xi, z') = S \phi(TU) (\xi + \phi^2(z')) \smallskip \\
\displaystyle \hspace{5em}
                  + \: B + B\xi\sum_{i=0}^\infty\phi^{2i+1}(BU-\phi(AU))
                  + \: B\sum_{i=0}^\infty\phi^{2i+1}(BU\phi(z')-\phi(AU)\phi^2(z')) \nonumber \smallskip \\
\displaystyle \hspace{5em}
                  - \: A - A\xi\sum_{i=0}^\infty \phi^{2i}(BU-\phi(AU))
                  - \: A\sum_{i=0}^\infty\phi^{2i}(BU\phi(z')-\phi(AU)\phi^2(z')).
\end{array}
$$}
Pour résoudre cette équation, nous procédons par approximations 
successives. Nous allons construire par récurrence des suites 
convergentes $(\xi_n)_{n\geq 0}$, $(c'_n)_{n\geq 0}$ à valeurs dans $R$ 
et $(z'_n)_{n\geq 0}$ à valeurs dans $vR[[v]]$ qui convergent vers une 
solution de \eqref{equation-bidouille}. Posons pour initialiser la
construction $\xi_0=0$, $z'_0=0$, $c'_0=c$. Supposons à présent que 
$\xi_n$, $c'_n$ et $z_n'$ sont construits pour un certain entier $n$.
Nous allons déterminer $\xi_{n+1}$ et $c_{n+1}'$ de façon à ce que 
$\RHS(c'_{n+1}, \xi_{n+1}, z'_n)$ soit divisible par $T(v+p^p)$.

Nous remarquons que tous les termes de $\RHS(c'_{n+1}, \xi_{n+1}, z'_n)$ 
sont divisibles par $v^{d_1}$. Comme $T = \delta v^{d_1}(v+p)$, la 
divisibilité par $T(v+p^p)$ est équivalente à ce que $\RHS(c'_{n+1}, 
\xi_{n+1}, z'_n)$ s'annule en $v = -p$ et en $v = -p^p$. En isolant
les termes 
$$(c'_{n+1} - c) v^{d_1 + 1}
\quad ; \quad
d v^p 
\quad \text{et} \quad 
\xi_{n+1} \delta p^{p+1} v^{d_1}$$
dans l'expression de $\RHS(c'_{n+1}, \xi_{n+1}, z'_n)$, nous voyons
(après calcul) que ces conditions d'annulation se réécrivent sous la forme :
$$\begin{array}{rcl}
c'_{n+1}-c-d(-p)^{p-d_1-1}-\delta p\xi_{n+1}&=&p^{p-1} \cdot S_2(c'_{n+1}-c-d(-p)^{p-d_1-1},\xi_{n+1}), \smallskip \cr
c_{n+1}'-c-d(-p)^{p-d_1-1}&=&p^{p} \cdot S_1(c'_{n+1}-c-d(-p)^{p-d_1-1},\xi_{n+1}), \cr
\end{array}$$
où $S_1$ et $S_2$ sont des séries formelles en deux variables à coefficients
dans $R$.
Nous en déduisons que $(c'_{n+1}-c-d(-p)^{p-d_1-1},\xi_{n+1})$ est un
point fixe de l'application :
$$\begin{array}{rcl}
f_{n+1} : \quad R^2 & \longrightarrow & R^2 \\
(x,y) & \mapsto & \big(p^p S_1(x,y), \, 
p^{p-1}\delta^{-1} S_1(x,y) - p^{p-2} \delta^{-1} S_2(x,y)\big).
\end{array}$$
Or cette application est manifestement contractante. L'existence des
éléments $c_{n+1}$ et $\xi_{n+1}$ s'ensuit.
À partir de là, nous définissons $z'_{n+1}$ par
$$z'_{n+1}=-\xi_{n+1}+{g(c'_{n+1}, \xi_{n+1}, z'_n) \over\delta v^{d_1}(v+p)(v+p^p)}$$
et vérifions que $z_{n+1}'$ est multiple de $v$ en regardant le coefficient en $v^{d_1}$ dans $\RHS(c'_{n+1}, \xi_{n+1}, z'_n)$.

Il s'agit à présent de montrer que les suites ainsi définies $(\xi_n)$, $(z'_n)$ et $(c_n')$ convergent. Pour cela, montrons par récurrence sur $n$ que
\begin{itemize}
\item[$\bullet$] $p^n$ divise $c'_n-c'_{n-1}$ et $\xi_n-\xi_{n-1}$,
\item[$\bullet$] le polygone de Newton $\NPoly(z_n'-z_{n-1}')$ est 
situ\'e dans la zone grisée $\mathcal P_n$ définie dans le graphe 
ci-dessous.
\end{itemize}

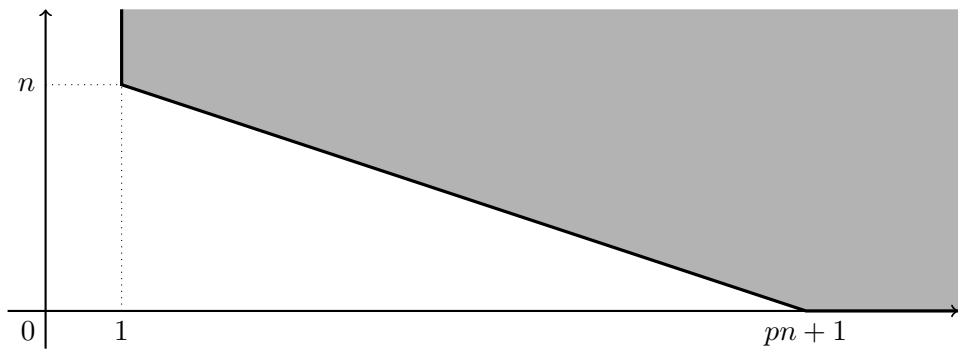
\begin{center}
\begin{tikzpicture}
\draw[->, thick] (-0.5,0)--(12,0);   
\draw[->, thick] (0,-0.5)--(0,4);   
\fill[opacity=0.3] (1,4)--(1,3)--(10,0)--(12,0)--(12,4)--cycle;
\draw[very thick] (1,4)--(1,3)--(10,0)--(12,0);
\draw[dotted] (0,3)--(1,3);
\draw[dotted] (1,0)--(1,3);
\node[below left] at (0,0) { $0$ };
\node[below] at (1,0) { $1$ };
\node[below] at (10,0) { $pn+1$ };
\node[left] at (0,3) { $n$ };
\end{tikzpicture}
\end{center}

\noindent
Par un abus de notation limité à la fin de cette démonstration, notons 
$\phi:(x,y)\mapsto (p^2x,y)$. Supposons que les hypothèses de récurrence 
sont satisfaites pour un entier $n\geq 1$. Alors 
$$\NPoly(\phi(z'_n)-\phi(z'_{n-1}))\subset\phi(\mathcal P_n)$$
d'où nous déduisons 
$$\NPoly(U\phi(z'_n)-U\phi(z'_{n-1}))\subset \phi(\mathcal P_n)+(0,p-1)$$ 
car $\NPoly(U)$ a une seule pente qui vaut $-1/p$ comme celle de 
$\phi(\mathcal P_n)$.
Un calcul montre que $f_n\equiv f_{n-1}\pmod {p^{n+p^2+p-1}}$.
Ainsi, si $t_n$ et $t_{n-1}$ sont respectivement les points fixes de 
$f_n$ et $f_{n-1}$, nous avons
$$f_n(t_{n-1})\equiv f_{n-1}(t_{n-1})\equiv t_{n-1} \pmod {p^{n+p^2+p-1}}$$
Donc $f_n^k(t_{n-1})\equiv t_{n-1}\pmod {p^{n+p^2+p-1}}$ pour tout $k$ dans $\N$. 
Comme $f_n$ est contractante, en passant à la limite sur $k$, nous trouvons
$t_n\equiv t_{n-1}\pmod {p^{n+p^2+p-1}}$ et, par suite :
\begin{equation}
\label{eq:cpn}
c'_n\equiv c_{n-1}'\pmod {p^{n+1}}
\quad \text{et} \quad \xi_n\equiv \xi_{n+1}\pmod {p^{n+1}}.
\end{equation}
Enfin :
$$z'_{n+1}-z'_n=\xi_n-\xi_{n+1}+{\RHS(c'_{n+1}, \xi_{n+1}, z'_n)
- \RHS(c'_n, \xi_n, z'_{n-1})
\over \delta v^{d_1}(v+p)(v+p^p)}$$
Les calculs précédents garantissent que le polygone de Newton du 
numérateur de la fraction ci-dessus est situé au dessus de la droite 
reliant les points $(0, n+p^2-1)$ et $(p^2(n+p^2-p-1), 0)$. Or, la 
division par $\delta v^{d_1}(v+p)(v+p^p)$ translate le polygone de 
Newton d'au maximum $(-d_1-2,0)$. Nous concluons alors la récurrence en 
remarquant que $z'_{n+1}-z'_n$ n'a pas de terme constant.

De \eqref{eq:cpn}, nous déduisons que $c' \equiv c \pmod p$. Et, enfin,
en reprenant point par point les étapes de la démonstration, nous
constatons que le $c'$ que nous avons construit dépend effectivement
de façon analytique de $a$, $b$, $c$, $d$ et $\delta$ (sachant que le
point fixe d'une application contractante s'obtient comme la limite
des itérés successifs d'un point initial quelconque).
\end{proof}

Posons à présent $R_\expl = \oE[[X,Y,Z]]/(XY + \delta p^2)$ et 
considérons le $\phi$-module libre de rang $2$ sur $R_\expl \hat
\otimes_{\Zp} \ocE$ sur lequel l'action de $\phi$ dans une base est
donnée par la matrice
$$\begin{pmatrix}
0 & \delta v^{d_1 + 1} (v+p) \\
v(v^p + p) & Xv + Y v^{d_1 + 1} + Z v^{d_1+2}
\end{pmatrix}.$$
Pour tout $\xi'$  dans $E'$, considérons le morphisme
$$h_{\xi'} : \oEp \otimes_{\oE} R_\expl \simeq \oEp[[X,Y,Z]]/(XY + \delta 
p^2) \longrightarrow R_{\xi'}$$
défini par :
\begin{itemize}
\item si $\xi' \in k_{E'}^\times$ :
$\left\{
\begin{array}{l}
X \mapsto p\alpha' \\
Y \mapsto p\alpha \\
Z \mapsto
 \mathcal F(\delta,\, p\alpha',\, 
            p \alpha, \, \alpha' a'_0 a_1 + \alpha, \,\alpha')
\end{array}
\right.$
\item si $\xi' = 0$ :
$\left\{
\begin{array}{l}
X \mapsto p \alpha' a'_1 \\
Y \mapsto \alpha a_1 \\
Z \mapsto
 \mathcal F(\delta,\, p\alpha' a'_1,\, 
            \alpha a_1, \, \alpha' a'_0, \,\alpha' a'_1)
\end{array}
\right.$
\item si $\xi' = \infty$ :
$\left\{
\begin{array}{l}
X \mapsto \alpha' a'_0 \\
Y \mapsto p\alpha a_0 \\
Z \mapsto \alpha a_0 + \alpha' a_1
\end{array}
\right.$
\end{itemize}
où $\alpha, \alpha', a_0, a'_0, a_1, a'_1$ s'expriment en fonction des
variables de $R_{\xi'}$ comme indiqué dans le tableau du début de la partie \ref{sssec:recoll}
et où
$\mathcal F(\Delta, A, B, C, D)$ est la série dont il est question dans
la remarque \ref{rem:depanalytique}. Le produit des flèches $h_{\xi'}$
définit un morphisme $h_{E'} : \oEp \otimes_{\oE} R_\expl \to 
R_{\expl, E'}$. En outre, les applications $h_{\xi'}$ sont définies de façon 
à ce que $R_{\expl, E'} \hat \otimes_{R_\expl} M_\expl$ soit isomorphe à 
$M_{E'}$. D'autre part, nous remarquons que $h_{E'}$ est injective pour 
tout $E'$. La proposition \ref{prop:factorisationf} entraîne alors 
l'existence d'une factorisation de $\tilde f_{E'}$ :
$$\tilde f_{E'} : \oEp \otimes_{\oE} R^\psi(\vv_0, \ttt, \rhobar)
\stackrel{\tilde g_{E'}}{\longrightarrow}
\oEp \otimes_{\oE} R^\expl 
\stackrel{\tilde h_{E'}}{\longrightarrow}
R_{E'}.$$
De plus, ces factorisations commutent aux changements de base, dans le sens 
où l'application $\tilde g_{E'}$ s'identifie à $\Id \otimes \tilde g_E$ 
pour toute extension finie $E'$ de $E$. Le lemme \ref{lem:tildeginj} 
s'applique ainsi et assure que le morphisme $\tilde g_E : R^\psi(\vv_0, \ttt, 
\rhobar) \to R_\expl$ est injectif.
À ce stade, il reste à déterminer son image par la méthode du \S 
\ref{par:interlude}. Pour cela, nous remarquons que la matrice de 
$\phi$ agissant sur $k_E \otimes_{R_\expl} M_\expl$ s'écrit :
$$\begin{pmatrix}
0 & \overline{\delta} v^{d_1+2} \\
v^{p + 1} & Xv + Y v^{d_1 + 1} + Z v^{d_1+2}
\end{pmatrix}.$$
Dans la nouvelle base où le second vecteur a été multiplié par $v^{p + 1}$,
elle prend la forme :
$$\begin{pmatrix}
0 & \theta^{-1} v^{p^3+ p^2 + d_1+2} \\
1 & X v^{p^3 + p^2-p} + Y v^{p^3 + p^2-p+d_1} + Z v^{p^3 + p^2- p + d_1+1}
\end{pmatrix}$$
qui a été considérée dans la partie \ref{sssec:calculExt1}. Notons $X^\star 
: R_\expl \to k_E[\varepsilon]$ (avec, rappelons-le, $\varepsilon^2 = 0$) 
le morphisme de $\oE$-algèbres qui envoie la variable $X$ sur $\varepsilon$ 
et les autres variables $Y$ et $Z$ sur $0$. Définissons également $Y^\star$
et $Z^\star$ de manière analogue.
Un calcul élémentaire montre que les images de $X^\star$, $Y^\star$ 
et $Z^\star$ dans $\Ext^1_{\Goo} (\rhobar, \rhobar)$ sont respectivement
$$(\theta v^{-p-d_1-2}, 0), \quad  (\theta v^{-p-2}, 0), \quad (\theta v^{-p-1}, 0)$$
où, ici, nous avons identifié $\Ext^1_{\Goo}(\rhobar, \rhobar)$ avec
sa description explicite donnée par la proposition \ref{propdecompositionext1}.
Il résulte du lemme \ref{lemdecompositionext1} que ces images sont
linéairement indépendantes dans $\Ext^1_{\Goo}(\rhobar, \rhobar)$ et, par 
suite, que l'application tangente à $\tilde g_E$ est injective. Nous en 
déduisons que $\tilde g_E$ est surjectif ; c'est donc un isomorphisme.

Au final, nous avons donc démontré la proposition suivante.

\begin{prop}
\label{prop:XYp2}
Soit $\rhobar = \Ind_{G_{F'}}^{G_F} (\omega_4^{1 + r_0} \cdot \nr' (\theta))$ pour un
entier $r_0$ dans $\{0, \ldots, p-3\}$.
Alors l'anneau de déformations $R^\psi(\vv_0, \ttt, \rhobar)$ s'identifie à
$\frac{\oE[[X,Y,Z]]}{(XY + p^2)}$.
\end{prop}

\begin{proof}
D'après ce qui précède, l'application $\tilde g_E$ réalise un isomorphisme 
entre $R^\psi(\vv_0, \ttt, \rhobar)$ et $\frac{\oE[[X,Y,Z]]}{(XY + \delta 
p^2)}$. Or cette dernière $\oE$-algèbre est isomorphe à 
$\frac{\oE[[X,Y,Z]]}{(XY + p^2)}$, l'isomorphisme étant obtenu par exemple 
en envoyant $X$ sur $\delta^{-1} X$, $Y$ sur $Y$ et $Z$ sur $Z$.
\end{proof}

\subsection{Calcul des multiplicités intrinsèques}\ 
\label{ssec:mipS}

Dans cette dernière partie, nous appliquons les résultats précédents aux calculs des multiplicités intrinsèques $(m_{\rhobar}(\sigma))_{\sigma \in \D(\rhobar)}$ des poids de Serre de $\rhobar$.
Le th\'eor\`eme~A de \cite{GK} fournit la conjecture \ref{conjprincipale} pour $\vv_0$ et tout type galoisien $\ttt$.
Il assure également que $m_{\rhobar}(\sigma)$ est non nulle si et seulement si $\sigma$ est un poids de $\rhobar$.
Le théorème  s'\'ecrit donc ici
$$
\mu_{\Gal}(\vv_0, \ttt, \rhobar) = \sum\limits_{\sigma \in  \D(\rhobar) \cap \D(\vv_0, \ttt)} m_{\vv_0,\ttt}(\sigma) m_{\rhobar}(\sigma).
$$
Pour les types $\ttt$ consid\'er\'es ici, nous avons $\D(\vv_0, \ttt) = 
\D(\ttt)$, la multiplicité $m_{\vv_0,\ttt}(\sigma)$ vaut~$1$ pour tout 
$\sigma$ dans $\D(\ttt)$ et l'ensemble $\D(\ttt)$ est décrit par une 
formule combinatoire rappelée dans la partie \ref{sssec:poidstype} (voir \cite{BP} et \cite{Da}). Nous obtenons donc l'équation 
$$
\mu_{\Gal}(\vv_0, \ttt, \rhobar) = \sum\limits_{\sigma \in  \D(\rhobar) \cap \D(\ttt)}  m_{\rhobar}(\sigma),
$$
avec $m_{\rhobar}(\sigma)$ dans $\N^{*}$ pour tout $\sigma$ dans $\D(\rhobar)$.
Le tableau de la figure \ref{fig:poidsserre} (page \pageref{fig:poidsserre})
rassemble, pour les représentations $\rhobar$ non génériques, les 
types~$\ttt$ ayant des poids de Serre en commun avec $\rhobar$ et 
l'ensemble $ \D(\rhobar) \cap \D(\ttt)$ de ces poids communs. Le code
couleur utilisé dans ce tableau est le même que celui que nous avions
introduit pour le tableau de la figure \ref{fig:engeances} : le bleu 
correspond aux couples $(\rhobar, \ttt)$ pour lesquels la variété de Kisin
est isomorphe à $\P^1_{k_E}$.
Nous nous apercevons que les lignes bleues sont exactement celles qui 
contiennent un poids modifié de $\rhobar$. Les résultats de la partie 
\ref{ssec:recollement} se résument donc comme suit.

\renewcommand{\arraystretch}{1.3}

\begin{figure}
\begin{center}
\small
\begin{tabular}{| c | c | c | c |}
\mline
\multicolumn{2}{| c |}{Représentation  $\rhobar = \Ind_{G_{F'}}^{G_F} (\chi \cdot \nr' (\theta) )$}  &Type $\ttt = \eta \oplus \eta'$  & Poids de Serre dans $\D(\ttt) \cap \D(\rhobar)$ \\
\mline
\multirow{6}{*}{$\chi = \omega_4^{ 1 + r_0}$} & \multirow{2}{*}{$0 \leq r_0 \leq p-2$}& \multirow{2}{*}{$ \omega_2^{r_0} \oplus \omega_2^{-p}$} & $(p-2 - r_0, 0 ) \otimes \det^{1 + r_0 - p}$ \\
& & &  $( p -1 - r_0, p-2) \otimes \det^{r_0}$ \\
\cline{2 - 4}
& \multirow{2}{*}{$1 \leq r_0 \leq p-2$}& \multirow{2}{*}{$\omega_2^{r_0 - p} \oplus \mathds{1}$}  &  $( p-1 - r_0, p-2) \otimes \det^{r_0}$ \\
& & &  $( r_0 - 1, p -1 ) \otimes \mathds{1}$ \\
 \cline{2-4}
& \multirow{2}{*}{\color{bleu} $0 \leq r_0 \leq p-3$} 
& \multirow{2}{*}{\color{bleu} $ \omega_2^{1 + r_0 - p} \oplus \omega_2^{-1}$}  
& {\color{bleu} $( r_0 + 1, p-1 ) \otimes \det^{-1}$} \\
& & & {\color{bleu} $( p - 2 - r_0, 0) \otimes \det^{1 + r_0 - p}$} \\
\mline
 \multirow{6}{*}{$ \chi = \omega_4^{ p(2 + r_1)}$}  & \multirow{2}{*}{$-1 \leq r_1 \leq p-3$} & \multirow{2}{*}{$ \omega_2^{p(1 + r_1)} \oplus \omega_2^{-1}$} & $( 0 , p-3 - r_1) \otimes \det^{-1 + p(2+r_1)}$ \\
& & & $( p -2,  p-2 - r_1) \otimes \det^{p(1 + r_1)}$ \\
\cline{2-4}
 & \multirow{2}{*}{$ 0 \leq r_1 \leq p-3$} &\multirow{2}{*}{ $ \omega_2^{ - 1 + p (1 + r_1) } \oplus \mathds{1}$ } & $( p -2,  p-2 - r_1) \otimes \det^{p(1 + r_1)}$ \\
& & &  $( p -1, r_1 ) \otimes \mathds{1}$  \\
\cline{2-4}
& \multirow{2}{*}{\color{bleu} $-1 \leq r_1 \leq p-4$} 
& \multirow{2}{*}{\color{bleu} $\omega_2^{- 1 + p(r_1 + 2) } \oplus \omega_2^{-p}$}  
& {\color{bleu} $( p -1, r_1 + 2 ) \otimes \det^{-p}$} \\
& & & {\color{bleu} $( 0 , p-3 - r_1) \otimes \det^{-1 + p(2+r_1)}$} \\
\mline
\end{tabular}
\end{center}
\caption{Poids de Serre communs à la représentation $\rhobar$ et au type $\ttt$}
\label{fig:poidsserre}
\end{figure}

\begin{thm}\label{thm:KisinBT} 
Soit $\rhobar$ une représentation irréductible et $\ttt$ un type galoisien comme dans la partie \ref{sec:methode}.
\begin{enumerate}[(i)]
\item Supposons $\rhobar$ non totalement non générique. Alors nous avons :
$$
\begin{array}{rcll}
R^\psi(\vv_0,\ttt,\rhobar) & \simeq & \{0\} & \text{si } \D(\ttt) \cap \D(\rhobar) = \emptyset \; ; \medskip \\
& \simeq & \displaystyle\frac{\oE[[ X , Y , T ]]}{ (XY + p^2)}& \text{si } \D(\ttt) \cap \D(\rhobar) \text{ contient un poids modifié de } \rhobar \; ; \medskip\\
& \simeq & \displaystyle \frac{\oE[[ X , Y , T ]]}{(XY + p)} & \text{sinon}.\\
\end{array}
$$
\item Supposons $\rhobar$ totalement non générique, c'est-à-dire de la forme $\Ind_{G_{F'}}^{G_F} (\omega_4 \cdot \nr' (\theta) )$ ou $\Ind_{G_{F'}}^{G_F} (\omega_4^p \cdot \nr' (\theta) )$.
Soit $\ttt$ un type galoisien tel que $\D(\ttt)$ contient le poids modifié de $\rhobar$ ; alors nous avons
$$ R^\psi(\vv_0,\ttt,\rhobar) \simeq \displaystyle\frac{\oE[[ X , Y , T ]]}{ (XY + p^2)}.$$
\end{enumerate}
\end{thm}

\begin{rem}
Le seul cas du tableau de la figure \ref{fig:poidsserre} (page \pageref{fig:poidsserre}) qui n'est pas traité par le théorème précédent est celui où la représentation $\rhobar$ est totalement non générique et où les poids communs à $\ttt$ et $\rhobar$ sont exactement les deux poids non modifiés de $\rhobar$ (rappelons que $\rhobar$ a alors trois poids, dont un seul modifié, voir le début de la partie \ref{sec:degre2}).
D'après la remarque \ref{rem:totnongen}, l'anneau $R^\psi(\vv_0,\ttt,\rhobar)$ s'identifie alors à un sous-anneau strict de $\displaystyle \frac{\oE[[ X , Y , T ]]}{(XY + p)} $.

\end{rem}

Nous observons que, dans tous les cas du théorème \ref{thm:KisinBT} pour lesquels l'anneau $R^\psi(\vv_0,\ttt,\rhobar)$ n'est pas nul, la multiplicité galoisienne $\mu_{\Gal}(\vv_0, \ttt, \rhobar)$ est $2$ et l'ensemble $\D(\ttt) \cap \D(\rhobar)$ est composé d'exactement deux poids.
Ainsi pour tout poids de Serre $\sigma$ de $\rhobar$, nous avons : si $\sigma$ est contenu dans un $\D(\ttt)$ pour un  $\ttt$ tel que $R^\psi(\vv_0,\ttt,\rhobar)$ est prédit par le théorème \ref{thm:KisinBT} et est non nul, alors $\sigma$ a multiplicité intrinsèque $1$ dans $\rhobar$.

Or, tout poids de Serre de $\rhobar$ est dans un tel $\D(\ttt)$, sauf :
\begin{itemize}
\item[$\bullet$] si $\rhobar$ est de la forme $ \Ind_{G_{F'}}^{G_F} (\omega_4^{p-1} \cdot \nr' (\theta) ) \otimes \omega_2^s$ ou $ \Ind_{G_{F'}}^{G_F} (\omega_4^{p(p-1)} \cdot \nr' (\theta) ) \otimes \omega_2^s$, le poids de Serre (modifié) totalement irrégulier $(p-1,p-1)\otimes\det^{s'}$ (avec $s' = s-1$ ou $s'=s-p$) ;
\item[$\bullet$] si $\rhobar$ est totalement non générique, l'unique poids de $\rhobar$ qui est non modifié et qui n'est pas le symétrique du poids modifié, c'est-à-dire le poids $(p-1,p-2)\otimes\det^{s}$ (resp. $(p-2,p-1)\otimes\det^{s}$) pour la représentation $ \Ind_{G_{F'}}^{G_F} (\omega_4 \cdot \nr' (\theta) ) \otimes \omega_2^s$ (resp. $ \Ind_{G_{F'}}^{G_F} (\omega_4^{p} \cdot \nr' (\theta) ) \otimes \omega_2^s$).
\end{itemize}
Nous en déduisons le corollaire suivant.

\begin{cor}
\label{coro:multiplicites}
Soit $\rhobar$ une représentation continue irréductible de~$G_{\Q_{p^2}}$ dans~$\GL_2(\Fpbar)$
\begin{enumerate}[(i)]
\item Supposons que $\rhobar$ n'est pas totalement non générique et n'a pas de poids totalement irrégulier. Alors les quatre poids de $\rhobar$ ont pour multiplicité intrinsèque $1$.
\item Supposons que $\rhobar$ possède un poids totalement irrégulier. Alors $\rhobar$ n'est pas totalement non générique et les trois poids de $\rhobar$ qui ne sont pas totalement irréguliers ont pour multiplicité intrinsèque $1$.
\item Supposons que $\rhobar$ est totalement non générique. Alors le poids modifié de $\rhobar$ et son symétrique ont pour multiplicité intrinsèque $1$.
\end{enumerate}
\end{cor}

\end{document}